\documentclass[reqno]{amsart}

\usepackage{amssymb, amsmath,mathtools}
\usepackage{mathrsfs}%stix}
\usepackage{amscd}
\usepackage{verbatim}
\usepackage{stmaryrd}
\usepackage[dvipsnames]{xcolor}
\usepackage{a4wide}

\usepackage{enumerate}

\usepackage[colorlinks,linkcolor={blue},citecolor={blue},urlcolor={purple},]{hyperref}

\usepackage[colorinlistoftodos,prependcaption,textsize=tiny]{todonotes}

\mathtoolsset{showonlyrefs}

\usepackage{forest}

\usepackage{tabularx}
\newcolumntype{L}{>{\arraybackslash}X}
\usepackage{multirow}

\usepackage{tikz}
 
\tikzset{>=latex}

\theoremstyle{plain}
\newtheorem{theorem}{Theorem}[section]
\theoremstyle{remark}
\newtheorem{remark}[theorem]{Remark}

\theoremstyle{plain}

\newtheorem{lemma}[theorem]{Lemma}
\newtheorem{proposition}[theorem]{Proposition}

\numberwithin{equation}{section}

\newcommand\bN{\mathbb{N}}

\newcommand\bR{\mathbb{R}}

\newcommand\bT{\mathbb{T}}
\newcommand\bP{\mathbb{P}}
\renewcommand\P{\mathbb{P}}

\newcommand\bZ{\mathbb{Z}}

\newcommand\cF{\mathscr{F}}
\newcommand\cG{\mathscr{G}}

\newcommand\cL{\mathcal{L}}
\newcommand\calL{\mathcal{L}}

\newcommand\cO{\mathcal{O}}

\newcommand\cR{\mathscr{R}}

\newcommand\cW{\mathcal{W}}

\def\N{{\mathbb N}}
\def\Z{{\mathbb Z}}

\def\R{{\mathbb R}}
\def\C{{\mathbb C}}

\newcommand{\E}{{\mathbb E}}
\renewcommand{\P}{{\mathbb P}}
\newcommand{\F}{{\mathscr F}}
\newcommand{\g}{\gamma}

\renewcommand{\O}{\Omega}

\newcommand{\Domm}{\mathcal{O}}

\newcommand{\T}{\mathbb{T}}
\newcommand{\W}{\mathcal{W}}

\newcommand{\A}{{\mathcal A}}

\newcommand{\loc}{\mathrm{loc}}

\newcommand{\wt}{\widetilde}

\newcommand {\Schw}{\mathcal{S}}
\newcommand{\Tr}{\mathrm{Tr}}

\newcommand{\di}{\mathrm{d}}

\newcommand{\one}{{{\bf 1}}}
\newcommand{\1}{\one}
\newcommand{\embed}{\hookrightarrow}

\newcommand{\supp}{\mathrm{supp}\,}

\newcommand{\norm}[1]{{\left\vert\kern-0.25ex\left\vert\kern-0.25ex\left\vert #1
    \right\vert\kern-0.25ex\right\vert\kern-0.25ex\right\vert}}

\renewcommand{\S}{\mathcal{S}}
\def\XXint#1#2#3{{\setbox0=\hbox{$#1{#2#3}{\int}$ }
\vcenter{\hbox{$#2#3$ }}\kern-.6\wd0}}

\newcommand{\Sf}{\mathcal{S}_{f}}

\allowdisplaybreaks

\begin{document}

\author{Antonio Agresti}
\address{Department of Mathematics Guido Castelnuovo, Sapienza University of Rome,
P.le Aldo Moro 5, 00185 Rome, Italy}
\email{antonio.agresti92@gmail.com}

\author{Fabian Germ}
\address{Delft Institute of Applied Mathematics\\
Delft University of Technology \\ P.O. Box 5031\\ 2600 GA Delft\\The
Netherlands.} \email{f.germ@tudelft.nl }

\author{Mark Veraar}
\address{Delft Institute of Applied Mathematics\\
Delft University of Technology \\ P.O. Box 5031\\ 2600 GA Delft\\The
Netherlands.} \email{M.C.Veraar@tudelft.nl}

\thanks{The first author is a member of GNAMPA (INdAM). The second and third authors have received funding from the VICI subsidy VI.C.212.027 of the Netherlands Organisation for Scientific Research (NWO)}

\date\today

\title[Sharp bounds for non-trace class noise]{Sharp bounds for non-trace class noise \\ and applications to SPDE\lowercase{s}}



%Primary:
%60B12: Limit theorems for vector-valued random variables (infinite-dimensional case)
%60H15: Stochastic partial differential equations.
%46E35: Sobolev spaces and other spaces of "smooth" functions, embedding theorems, trace theorems.
%Secondary:
%60B11: Probability theory on linear topological spaces
%60G15: Gaussian processes.
%47B10: Operators belonging to operator ideals (nuclear, $p$-summing, in the Schatten-von Neumann classes, etc.).
%35R60: Partial differential equations with randomness, stochastic partial differential equations.

\begin{abstract}
In the study of stochastic PDEs with colored, non-trace class space-time noise, one frequently encounters Gaussian series of the form
$$
\textstyle{g \,
\sum_{n\geq 1} \gamma_n \mu_n f_n,}
$$ 
where $(\gamma_n)_{n}$ is a sequence of independent standard Gaussian variables, $g$ is an $L^\eta(\mathcal{O})$ function, $(\mu_n)_{n}$ is a sequence of scalars, and $(f_n)_n$ is an orthonormal system in $L^2(\mathcal{O})$ where $\mathcal{O} \subseteq \mathbb{R}^d$ is an open set.  
In this manuscript, we establish necessary and sufficient conditions for the above sum to converge in Bessel potential spaces $H^{-s,q}(\mathcal{O})$. The latter can be interpreted as a Sobolev embedding for Gaussian series. 
Our main theorem is formulated using weighted sequence spaces that encode the $L^\infty$-growth of the orthonormal system $(f_n)_{n}$, a feature that is crucial for obtaining sharp estimates. 
We apply our results to the stochastic heat equation with additive non-trace class noise. In this case, our conditions capture the scaling relationship between the heat operator and the coloring of the noise. 
\end{abstract}

\maketitle

\tableofcontents

\section{Introduction}
\label{s:introduction}
In the study of stochastic PDEs (SPDEs), one often encounters Gaussian series formally written as
\begin{align}
\label{eq: intro general noise}
    \cW (t,x) = \sum_{n\geq 1} \mu_nf_n(x) w_n(t),
\end{align}
where $\mu = (\mu_n)_{n\geq 1}$ is a sequence of scalars, $\Sf = (f_n)_{n\geq 1}$ is an orthonormal system in $L^2(\cO)$ on an open set $\cO\subseteq\bR^d$ and $(w_n)_{n\geq 1}$ is a family of independent real-valued Brownian motions,  on a probability space $(\Omega,(\cF_t)_{t\geq 0},\bP)$. It is well-known that if $\mu_n\equiv 1$, then the time derivative of the series \eqref{eq: intro general noise} represents \emph{space-time white noise} on $\cO$. On the other hand, if $(\mu_n\|f_n\|_{L^\infty(\Domm)})_{n\geq 1}\in \ell^2$, then the series \eqref{eq: intro general noise} is well-defined in $L^2(\Omega;L^q(\cO))$ for all $t>0$ for all $q<\infty$ (if $\cO$ is bounded), and the corresponding derivative is often referred to as \emph{trace-class noise} on $\cO$. If $(\mu_n)_{n\geq 1}$ is in $\ell^2$, the same holds for $q=2$ for any orthonormal system $\Sf$ (without any $L^\infty$-condition). 

In the intermediate regimes, where $\mu\in \ell^\zeta$ or $\sum_{n\geq 1}|\mu_n|^\zeta \delta_n<\infty$ for $\zeta\in (2,\infty)$ and a weight $\delta_n$ (often dependent on $\Sf$), then the time derivative of $\cW$ in \eqref{eq: intro general noise} is called \textit{non-trace class noise} on $\cO$. Often in the literature, the coefficients $\mu=(\mu_n)_{n\geq 1}$ are referred to as ``coloring'' of the noise.
This distinction can also be appreciated by looking at the regularity of solutions to the stochastic
heat equation on the open set $\Domm\subseteq \R^d$:
\begin{equation}
\label{eq:SHE_intro}
\partial_t u(t,x)-\Delta u(t,x) =g (t,x)\, \partial_t\cW(t,x)
\end{equation}
supplemented with, for instance, Dirichlet boundary conditions if $\partial\Domm$ is non-empty, and zero initial conditions. 
Here, we additionally consider the random field $g$, which can be used to study nonlinear variants of the above SPDE. 
In the pure noise case $g\equiv 1$, it is well-known that trace-class noise implies that $\nabla u\in L^2_{\loc}([0,\infty)\times \Domm;\R^d)$ a.s., while, in the case of space-time white noise, one only has $u(t)\in H^{1-d/2-}(\Domm)$ a.s.\ for all $t>0$ (in particular, in $d\geq 2$ it is merely a distribution).

\smallskip

The aim of this paper is to prove \emph{sharp estimates} for the Gaussian series $\cW$ and corresponding sharp estimates for parabolic SPDEs, in the case where $\W$  is any Gaussian series whose time derivative is ``between'' trace-class and space-time white noise, with the extremes being included. 
By stochastic maximal $L^p$-regularity estimates \cite{MaximalLpregularity} (or also \cite[Section 3]{AV25_survey}), which provide optimal space-time bounds for solutions to \eqref{eq:SHE_intro}, these two results are intimately related. 
Illustrations of our results are given in Theorems \ref{thm:intro} and \ref{thm: spde intro} below, respectively. 
In turn, our central goal is to establish necessary and sufficient conditions for the convergence of the following Gaussian series 
\begin{equation}
\label{eq:BM}
g\,\sum_{n\geq 1} \g_n \mu_n f_n
\end{equation}
in Sobolev (more precisely, Bessel potential) spaces of \emph{negative} smoothness under suitable assumptions on the objects involved. In the above, $(\gamma_n)_{n\geq 1}$ is a sequence of Gaussian random variables, which appears in so-called $\gamma$-radonifying operators (see Subsection \ref{sss:gamma_spaces}) arising in sharp two-sided estimates for stochastic integrals in spaces such as $L^p(\Omega;L^q(\cO))$ as explained in the survey \cite{NVW13}. These sharp estimates play a key role in maximal regularity estimates for solutions to SPDEs \cite{AV25_survey,MaximalLpregularity}.
For the sake of concreteness, here we will mostly focus on the stochastic heat equation, although our setting allows for a much wider class of SPDEs (see Section \ref{sec: SPDE section}).
We emphasize that results for the ``endpoint cases'', i.e. trace-class and space-time white noises, will be presented as extreme cases of our results obtained. 
The importance of optimal (or sharp) estimates in terms of $g$ appearing on the right-hand side of \eqref{eq:SHE_intro} becomes apparent when dealing with nonlinear SPDEs, in which case 
$
g(t,x) = G(t,x,u(t,x))
$
for a function $G$, which is often assumed to satisfy certain (local) Lipschitz conditions in the $u$-variable. Then, it is desirable to require as little space integrability as possible on the coefficient $g(t,x,u)$, so that, for instance, in a fixed point argument, one can work with lower powers of $u$, avoiding a loss of smoothness in Sobolev embeddings. Such a situation was indeed the initial motivation for the present article, and it will be thoroughly explored in the forthcoming work \cite{AGVlocal}.

\smallskip

We now discuss the convergence of the Gaussian series \eqref{eq:BM}, with the coloring $\mu$ being between the case trace-class noise $(\mu_n \|f_n\|_{L^\infty(\cO)})_{n\geq 1}\in \ell^2$ and space-time white noise $\mu_n\equiv 1$ or, more generally $(\mu_n)_{n\geq 1}\in \ell^\infty$. 
To capture noises in this intermediate range, we introduce the following class of weighted sequences
\begin{align}
\label{eq: intro l zeta Sf condition}
    \|\mu\|_{\ell^\zeta(\Sf)}^\zeta:=\sum_{n\geq 1}|\mu_n|^\zeta 
    \|f_n\|_{L^\infty(\cO)}^2<\infty \quad \text{ for } \ \zeta\in [2,\infty),
\end{align}
and  $\|\mu\|_{\ell^\infty(\Sf)}:=\sup_{n\geq 1} |\mu_n| .$
We point out that the above space of sequences naturally arises by interpolating the above-mentioned condition in the extreme cases of trace-class ($\zeta=2$) and space-time ($\zeta=\infty$) noises. 
In many applications, the weights $(\|f_n\|_{L^\infty(\cO)})_{n\geq 1}$ are \emph{not} bounded, and therefore $\ell^\zeta(\Sf)$ is a true space of weighted sequences. 
Indeed, in the SPDE literature (see Subsection \ref{ss:comparison_intro} for references), the orthonormal system $\Sf=(f_n)_{n\geq 1}$ typically consists of $L^2$-normalized eigenfunctions of the Dirichlet Laplacian on $\cO$. If the latter is bounded and smooth, the best-known bound for such eigenfunctions is 
\begin{equation}
\label{eq:f_n_growth_in_general}
  \|f_n\|_{L^\infty(\cO)}\lesssim n^{(d-1)/2d},
\end{equation}
 with implicit constant independent of $n$ by \cite{grieser2002uniform} and Weyl's law for the eigenvalues $\lambda_n\sim n^{2/d}$, see \cite[p. 45]{hormander2007analysis}. 
The optimality can be proven, for instance, if $\cO$ is the unit ball in $\R^d$. Hence, $\ell^\zeta(\Sf)$ is typically a space of weighted sequences. 
As we will show in Section \ref{s:weight_necessary}, the unweighted viewpoint is actually not appropriate to study convergence of Gaussian series as in \eqref{eq:BM}, and thus to study parabolic SPDEs like \eqref{eq:SHE_intro}. Finally, we mention that, in case $\cO=\T^d$ and $\Sf$ is the standard Fourier basis for which $\|f_n\|_{L^\infty(\cO)}\eqsim 1$ and thus $\ell^\zeta(\Sf)$ is unweighted, then our results are \emph{also new}. The reader is referred to Subsections \ref{ss:comparison_intro} and \ref{ss:periodic_setting} for further comments.

\smallskip

As mentioned above, we study the convergence of the Gaussian series \eqref{eq:BM} in the Bessel potential-type space
$$
\wt{H}^{-s,q}(\cO)=\big\{U\in H^{-s,q}(\R^d)\,:\, \supp\, U \subseteq \overline{\cO} \big\}\quad \text{ for }\quad s>0 \  \text{ and }\  q\in (1,\infty),
$$
while we assume that $g\in L^\eta(\cO)$ for some $\eta\in (1,\infty)$ (see Subsection \ref{subs:func} for basic properties). 
In summary, the parameters $(s,q,\zeta,\eta)$ used throughout this article play the following role:
\begin{itemize}
\item $(-s,q)$ -- Spatial smoothness and integrability of the noise.
\item $\zeta$ -- Noise intensity.
\item $\eta$ -- Spatial integrability of  $g$.
\end{itemize}

Before formulating our main result on the convergence of the Gaussian series \eqref{eq:BM}, let us remark that whenever the series in \eqref{eq:BM} converges in $L^2(\Omega;\wt{H}^{-s,q}(\cO))$, then it also does with $(\gamma_n)_{n\geq 1}$ replaced by a sequence of independent Brownian motions $W_n(t)$ for any $t\geq 0$.
It is then routine to show that the limiting object is a Brownian motion with values in $\wt{H}^{-s,q}(\cO)$. As a consequence, the convergence holds in much more restrictive topologies, such as $L^p(\Omega;C^\alpha([0,T]; \wt{H}^{-s,q}(\cO)))$, for any $\alpha\in (0,1/2)$, $p\in (0,\infty)$ and $T\in (0,\infty)$. Indeed, the $L^p(\Omega)$-setting follows from the Kahane-Khintchine inequalities (see e.g.\ \cite[Section 6.2b]{HNVW2}), and the $C^\alpha$-regularity from the Kolmogorov continuity criterion. In this paper, the focus is on the convergence for fixed $t$ only, and the other types of convergence one can get afterwards by applying the mentioned results. 

The following can be seen as a Gaussian variant of the Sobolev embedding, and is a consequence of Theorem \ref{thm:MgTmu delta}, which is our main result on the convergence of Gaussian series as in \eqref{eq:BM}.

\begin{theorem}[Sobolev embedding for Gaussian series]
\label{thm:intro}
Let $\mathcal{O}\subseteq \R^d$ be an open set. Let $\Sf = (f_n)_{n\geq 1}$ be an orthonormal system in $L^2(\mathcal{O})$, and $\mu\in \ell^\zeta(\Sf)$.
Assume that $q\in (1, \infty)$, $\eta\in (1, q)$, $\zeta\in [2, \infty]$, and $s\in(0,d)$ satisfy 
\begin{align}\label{eq:conditionsmainintro}
\frac{s}{d} + \frac{1}{q}  \geq \frac{1}{\eta} + \frac{1}{2} - \frac{1}{\zeta} \qquad \text{ and } 
\qquad   \frac{1}{\eta} - \frac{1}{\zeta}<\frac12.
\end{align}
Then, for all $g\in L^{\eta}(\mathcal{O})$ 
the series \eqref{eq:BM} converges in $L^2(\Omega;\wt{H}^{-s,q}(\mathcal{O}))$ and 
\begin{align}\label{eq:main2}
\Big(\E\Big\|g\, \sum_{n\geq 1} \gamma_n  \mu_n  f_n \Big\|_{\wt{H}^{-s,q}(\mathcal{O})}^2\Big)^{1/2}
\lesssim_{d,s,q,\eta,\zeta} \|\mu\|_{\ell^{\zeta}(\Sf)} \|g\|_{L^{\eta}(\mathcal{O})}.
\end{align}
Finally, if $\Domm=\R^d$ and the first condition in \eqref{eq:conditionsmainintro} holds with equality, 
the series \eqref{eq:BM} even converges in $L^2(\Omega;\dot{H}^{-s,q}(\R^d))$ and \eqref{eq:main2} holds with $\Domm=\R^d$ and $\wt{H}^{-s,q}$ replaced by the homogeneous Sobolev space $\dot{H}^{-s,q}$.
\end{theorem}
As we discuss in further detail below, in particular in Subsection \ref{subsection SPDE comparison}, the range of parameters in \eqref{eq:conditionsmainintro} greatly improves previous results in this direction, and their applications to SPDEs. A comparison to the preceding literature is given in Subsections \ref{ss:comparison_intro} and \ref{subsection SPDE comparison}.

 One could expect that the estimate \eqref{eq:main2} can be proved via paraproduct techniques. However, it seems that this would require positive smoothness of $g$, which we do not assume. An idea of the proof of Theorem  \ref{thm:intro} will be sketched in Subsection \ref{ss:proof_strategy}.
\smallskip

In Theorem \ref{thm:intro}, further endpoints can be included in special cases. 
For this, the reader is referred to Theorem \ref{thm:MgTmu delta} and results in Subsection \ref{ss:endpoint_case_main_result}. 
Before going further, let us also note that the convergence properties in \eqref{eq:BM} are the same for complex and real Gaussian sequences $(\gamma_n)_{n\geq 1}$, see for instance \cite[Proposition 6.1.21]{HNVW2}.

\smallskip

Theorem \ref{thm:intro} has the following consequence for \eqref{eq:SHE_intro}.
For simplicity, in the case $\cO$ has a non-empty boundary, we only consider homogeneous Dirichlet boundary conditions. 

\begin{theorem}[Optimal space regularity estimates for the stochastic heat equation]
\label{thm: spde intro}
Let $\mathcal{O}$ be open and bounded and smooth, or $\mathcal{O} \in \{\R^d_+, \R^d,\T^d\}$. 
Assume that the conditions of Theorem \ref{thm:intro} are satisfied, and in addition, $q\geq 2$. Then,
for $p\in (2,\infty)$, $T<\infty$, and progressively measurable $g:\Omega\times(0,T)\to L^\eta(\Domm)$, the stochastic heat equation \eqref{eq:SHE_intro} with homogeneous Dirichlet boundary conditions in space and zero initial condition, has a unique solution $u$ satisfying
\begin{equation}
\label{eq:homogeneous_estimate_intro_0}
\E\|u\|_{L^p(0,T;H^{1-s,q}(\Domm))}^p
\lesssim_{T,d,s,q,\eta,\zeta} \|\mu\|_{\ell^{\zeta}(\Sf)}^p\,\E\|g\|^p_{L^p(0,T;L^{\eta}(\mathcal{O}))}.
\end{equation}
Finally, if $\Domm=\R^d$, $T=\infty$ and $
\frac{s}{d} + \frac{1}{q}  = \frac{1}{\eta} + \frac{1}{2} - \frac{1}{\zeta} $, then it also holds that 
\begin{equation}
\label{eq:homogeneous_estimate_intro}
\E\|u\|_{L^p(\R_+;\dot{H}^{1-s,q}(\R^d))}^p
\lesssim_{d,s,q,\eta,\zeta} \|\mu\|_{\ell^{\zeta}(\Sf)}^p \,\E\|g\|^p_{L^p(\R_+;L^{\eta}(\R^d))}.
\end{equation}
\end{theorem}

The previous is a special case of Theorem \ref{thm: spde with delta no time}. In the latter result, we also prove optimal space-time regularity estimates. In particular, the $L^p(0,T;H^{1-s,q}(\cO))$-norm of the left-hand side of \eqref{eq:homogeneous_estimate_intro_0} is replaced by $C([0,T];B^{1-s-2/p}_{q,p}(\cO))$ while keeping the same norm on the right hand side. Here $B^{\sigma}_{q,p}$ denotes the Besov space of smoothness $\sigma$ and spatial integrability $q$ (see \cite{Tri95}). The Besov spaces are known to be optimal in maximal estimates (see \cite{ALV21}). For $1-s-\frac2p - \frac{d}{q}>0$, the previous also provides estimates in the space $C([0,T]\times \overline{\cO})$ (by Sobolev embedding), which improves the existing results, see Subsection \ref{subsection SPDE comparison} for details. Therefore, our results can also be useful for the random field approach to SPDEs. 
A similar statement holds for \eqref{eq:homogeneous_estimate_intro}, with the space $B$ replaced by its homogeneous counterpart. 

\smallskip

Next, we comment on some of the advantages of our approach to Theorem \ref{thm:intro} (and hence Theorem \ref{thm: spde intro}). For the proof strategy and a comparison to the literature, the reader is referred to Subsections  \ref{ss:comparison_intro} and \ref{ss:proof_strategy}, respectively.
Among others, as we do not rely on explicit computations, we avoid several technical difficulties which are often solved by imposing more conditions on the system $\Sf$, or its relation to the differential operator involved, see for instance \cite{cerrai2009khasminskii}.
More precisely, our approach has several significant advantages:
\begin{itemize}
  \item \emph{(Sharpness of the conditions).} We prove the necessity and sufficiency of our conditions on the parameters $(s,q,d,\eta,\zeta)$. Among others, we show the first condition in \eqref{eq:conditionsmainintro}, i.e.\ 
    \begin{align*}
    \frac{s}{d}+\frac{1}{q} \geq \frac{1}{\eta}+\frac12 -\frac{1}{\zeta}
    \end{align*}
    is a necessary and sufficient condition for the convergence of \eqref{eq:BM}, see Proposition \ref{p:necessityLinfty_hom}.
    \item  \emph{(Reaching criticality in nonlinear reaction-diffusion equations).} The sharpness and flexibility in all parameters $(s,q,\eta,\zeta)$ allow us to cover critical situations for stochastic reaction-diffusion equations with (non-)trace class noise in \cite{AGVlocal, AVreaction-local}, in which situation the critical space is an $L^r$-space where $r$ is related to $\zeta$ and the polynomial growth of the nonlinearities.
    \item \emph{(Flexibility in the $g$-integrability).} The exponent $\eta$ for which $g\in L^\eta(\cO)$ can be taken smaller than in the preceding results, and even improving the norm on the left-hand side of the corresponding estimate, see the comments in Subsection \ref{sss:comparison_new_2} for details.
    \item \emph{(General orthonormal systems $\Sf$).}  The sequence $\Sf = (f_n)_{n\geq 1}$ is merely assumed to be an orthonormal system, which does not need to have any relation to the differential operator in \eqref{eq:SHE_intro}. In particular, in Theorem \ref{thm: spde intro} (and its extensions in Subsection \ref{ss:results_heat_eq_general}), we can replace the Laplace operator by an operator in divergence form (see Remark \ref{rem:generalization_SPDE}) with \emph{no a-priori} knowledge of the eigenfunctions of the divergence form operator. 
    This fact greatly \emph{simplifies and improves} the obtained bounds in the literature, see Subsection \ref{subsection SPDE comparison} for a detailed comparison. Let us point out that, due to the generic optimality of the bound \eqref{eq:f_n_growth_in_general}, bounds for \eqref{eq:BM} in terms of $\|g\|_{L^2(\cO)}$ could so far only be obtained for $d<6$, while we obtain them for any dimension $d$.
    \item \emph{(Unbounded domains).} Since we do not need that $\Sf$ is associated to a differential operator, we can also allow unbounded domains.  For the full space $\R^d$, we also present a Fourier multiplier approach to estimates \eqref{eq:main2}, avoiding orthonormal bases altogether.
\end{itemize}

The sharpness of our results is related to the \emph{scaling invariance} of the corresponding estimates. In the case of Theorem \ref{thm:intro}, the optimality via scaling is performed in the proof of Proposition \ref{p:necessityLinfty_hom}. Consequently, the optimality of Theorem \ref{thm:intro} and of stochastic maximal $L^p$-regularity estimates \cite{MaximalLpregularity} (see also \cite[Subsection 6.2]{ALV21} for homogeneous spaces) imply the sharpness of Theorem \ref{thm: spde intro}. However, for illustrative purposes, in the following subsection, we provide a direct argument of the optimality of the first condition in \eqref{eq:conditionsmainintro} by a (formal) scaling argument on the stochastic heat equation \eqref{eq:SHE_intro} with non-trace class noise.

\subsection{Scaling for the stochastic heat equation with colored noise}
\label{ss:scaling_intro}
Consider the stochastic heat equation \eqref{eq:SHE_intro} on $\cO=\R^d$ with colored noise $\cW^{(\alpha)}$, where the parameter $\alpha\geq 0$ determines the behavior of the noise under the parabolic space-time scaling: for $\lambda>0$, $x\in\R^d$ and $t>0$,
\begin{equation}
\label{eq:scaling_alpha_intro}
\partial_t \cW^{(\alpha)}(\lambda^2 t,\lambda x)\stackrel{\text{law}}{=}\lambda^{-1-d/2+\alpha}\,\partial_t \cW^{(\alpha)}( t, x).
\end{equation} 
As is well known, space-time white noise satisfies the above with $\alpha=0$.
Let $u$ be a solution to the stochastic equation \eqref{eq:SHE_intro} on $\R^d$, for all $\lambda>0$, $x\in \R^d$ and $t>0$, set 
$
u_\lambda (t,x)=u(\lambda^2 t ,\lambda x)
$. For the latter, by the scaling invariance of the noise \eqref{eq:scaling_alpha_intro}, we formally have
\begin{align*}
\partial_t u_\lambda(t,x) -\Delta u_\lambda(t,x) 
&  = \lambda^2 \big[\partial_t u(\lambda^2 t, \lambda x) - \Delta u(\lambda^2 t, \lambda x)\big]\\
& = \lambda^2 \, g(\lambda^2 t ,\lambda x )\,\partial_t \cW^{(\alpha)}(\lambda^2 t,\lambda x)\\
&\stackrel{{\rm law}}{=}\lambda^{1-d/2+\alpha} g(\lambda^2 t ,\lambda x )\,\partial_t\cW^{(\alpha)}(t,x)
= g_{\lambda}(t,x) \,\partial_t\cW^{(\alpha)}(t,x),
\end{align*}
where 
\begin{equation*}
g_\lambda(t,x)= \lambda^{1-d/2+\alpha} g(\lambda^2 t ,\lambda x).
\end{equation*}
To check that the estimate in \eqref{eq:homogeneous_estimate_intro} of Theorem \ref{thm: spde intro} is (roughly) \emph{invariant under the rescaling} $(u,g)\mapsto (u_\lambda,g_\lambda)$, we need to connect the scaling parameter $\alpha$ to the coloring exponent $\zeta$. To this end, we consider a noise defined by means of the multidimensional Haar basis of $L^2(\R^d)$:
\begin{equation}
\label{eq:noise_invariance_scaling_intro}
\cW^{(\alpha)}(t,x)=\sum_{\sigma\in \Sigma_d} \sum_{j\in \Z}\sum_{k\in \Z^d} 
(1+|k|^2)^{-\beta/2} 2^{-j\alpha} \psi_{j,k}^{(\sigma)} (x)\, w^{j,k}(t)
\end{equation}
where $\Sigma_d =\{0,1\}^d\setminus \{(0,\dots,0)\}$ is the orientation parameter, $\beta>d/2$ and $(w^{j,k})_{j\in\Z,k\in\Z^d}$ are  independent standard Brownian motions and 
$$
\psi^{(\sigma)}_{j,k}(x)= 2^{jd/2} \Psi^{(\sigma)}(2^{j} x +k)\ \ \text{ for all }\ j\in \Z, \ k\in \Z^d.
$$ 
with $\Psi^{(\sigma)}\in L^\infty(\R^d)$ the mother wavelet with type $\sigma\in \Sigma_d$. 
If $d=1$, then $\Sigma_1=1$ and $\Psi^{(1)}=\1_{[0,1/2)}-\one_{[1/2,1)}$. Here, to perform the scaling argument, we do not need the precise form of the function $\Psi^{(\sigma)}$. Details on the construction can be found in  \cite[Chapters 5 and 10]{D92_wavelets}. 
One can check that the noise in \eqref{eq:noise_invariance_scaling_intro} formally respects the scaling relation \eqref{eq:scaling_alpha_intro} in the case $\lambda=2^{m}$ for $m\in \Z$, while the parameter $\beta$ is only needed to ensure spatial integrability, and therefore does not influence the scaling.
Let $\Sf$ be the Haar basis of $L^2(\R^d)$, and $\mu=(\mu_{j,k})_{j\in \Z,k\in \Z^d}$ with $\mu_{j,k}=(1+|k|^2)^{-\beta/2} 2^{-j\alpha}$. From \eqref{eq: intro l zeta Sf condition} and $\|\psi_{j,k}^{(\sigma)}\|_{L^\infty(\R^d)}=2^{jd/2}$, for $\zeta\in [2,\infty)$, one has
$$
 \sum_{|j|\leq N}\sum_{k\in \Z^d} |\mu_{j,k}|^\zeta\, 2^{jd}
= 
\left\{
\begin{aligned}
&\text{diverges logarithmically as $N\to \infty$ for }\zeta = d/\alpha,\\
&\text{diverges exponentially as $N\to \infty$ otherwise}.
\end{aligned}
\right.
$$
In particular, the value of $\zeta=d/\alpha$ is of specific interest, and to obtain a finite $\ell^\zeta(\Sf)$-norm of the coloring $\mu$ (see \eqref{eq: intro l zeta Sf condition}), it suffices to modify the one in \eqref{eq:noise_invariance_scaling_intro} at a logarithmic scale compared to the amplitude $2^{-j\alpha}$.
We emphasize that it is natural to have such logarithmic corrections. Indeed, this is already the case of the L\'evy modulus of continuity for Brownian motion, and to avoid them  requires working, even in the space variable, with Besov spaces like $B^{1/2}_{q,\infty}$ and their variants rather than Sobolev or Bessel potential spaces, see e.g.\ \cite{ArOh,V11_regularity_Besov}. 
However, it is worth stressing that scaling \emph{cannot} highlight such logarithmic corrections, as for instance $B^{1/2}_{q,\infty}$ has the same scaling as $H^{1/2,q}$, but Brownian paths live in the former but not in the latter space. 
Finally, we point out that, for similar reasons but in a quite different context, logarithmic corrections often appear also in the context of critical singular SPDEs, see \cite[Subsection 1.3]{cannizzaro2024lecture} and \cite[Sections 4-6]{caravenna2025disordered}.

For the sake of the scaling argument, we therefore ignore the required logarithmic correction. 
Now, if $\zeta=d/\alpha$, then the estimate \eqref{eq:homogeneous_estimate_intro} in Theorem \ref{thm: spde intro} is \emph{invariant} under the natural scaling $(u,g)\mapsto (u_\lambda,g_\lambda)$ of the SPDE \eqref{eq:SHE_intro} on $\cO=\R^d$ if and only if the parameters $(s,q,\eta,\zeta)$ satisfy 
\begin{equation}
\label{eq:condition_equality_scaling}
\frac{s}{d}+\frac{1}{q} = \frac{1}{\eta}+\frac12 -\frac{1}{\zeta},
\end{equation}
that is the first condition in \eqref{eq:conditionsmainintro} with equality. 
Indeed, for all $q\in [2,\infty)$ and $p\in (2,\infty)$,
\begin{align*}
\lambda^{1-s-d/q-2/p}
(\E\|u\|_{L^p(\R_+;\dot{H}^{1-s,q}(\R^d))}^p)^{1/p}
& \eqsim 
(\E\|u_{\lambda}\|_{L^p(\R_+;\dot{H}^{1-s,q}(\R^d))}^p)^{1/p}\\
&\lesssim
 (\E\|g_\lambda\|_{L^p(\R_+;L^\eta(\R^d))}^p)^{1/p}\\
& =\lambda^{1-d/2+d/\zeta-d/\eta-2/p} (\E\|g\|_{L^p(\R_+;L^\eta(\R^d))}^p)^{1/p}.
\end{align*}
The exponents in the above match if and only if \eqref{eq:condition_equality_scaling} holds. Hence, the estimates of Theorem \ref{thm: spde intro} are sharp. It is routine to check that the above arguments can also be applied to space-time estimates of Section \ref{sec: SPDE section}, therefore also proving their optimality.

\subsection{Comparison with the literature}
\label{ss:comparison_intro}
 
 Early works deriving estimates for \eqref{eq:BM} for $q=2$
 are \cite{C03, CR04} and subsequent works such as \cite{CR05}. 
 There,
the condition \eqref{eq: intro l zeta Sf condition} is replaced by the condition
\begin{align}
\label{eq: intro sup condition}
    \sup_{n\geq 1}\|f_n\|_{L^\infty(\cO)}<\infty,
\end{align}
together with the requirement that $\mu\in\ell^\zeta$. 
It is additionally assumed that the orthonormal system $\Sf$ consists of the eigenfunctions of the differential operator involved. Without this last assumption, under the condition \eqref{eq: intro sup condition}, we 
are able to prove essentially the same estimate for the stochastic convolution under less restrictive conditions on the parameters $(\zeta,d)$
(see Subsection \ref{subsection SPDE comparison} for details). 
    To the best of our knowledge, a condition of the form \eqref{eq: intro l zeta Sf condition} first appeared in \cite{cerrai2009khasminskii}, where again the system $\Sf$ is assumed to also diagonalize the differential operator involved with eigenvalues $(\alpha_n)_{n\geq 1}$ and an additional constraint of the form $\Lambda:=\sum_{n\geq 1} \alpha_n^{-\beta}\|f_n\|_{L^\infty}^2<\infty$ is imposed, where $\beta(\zeta-2)<\zeta$. However, this last condition turns out to be quite restrictive in many realistic situations. While we give a detailed comparison in Subsection \ref{subsection SPDE comparison}, we only mention here that, for instance, in dimension $d=2$, the  condition $\Lambda<\infty$ together with \eqref{eq:f_n_growth_in_general} leads to the constraint  $\zeta<6$; a restriction we simply do not get here. 
    Instead, since we do not assume this extra condition $\Lambda<\infty$, we obtain the desired estimates for any $d<\infty$ (see, again, Subsection \ref{subsection SPDE comparison} for a more detailed comparison).
    The same condition as in \cite{cerrai2009khasminskii} also appears in  \cite{cerrai2011averaging, cerrai2017averaging} and follow-up articles.
    In the very recent article \cite{cerrai2025nonlinear} estimates for a stochastic convolution were derived only under the assumption that $\mu\in\ell^\zeta(\Sf)$, for $\zeta<2d/(d-2)$ in $d\geq 2$ and $\mu\in\ell^\infty$ for $d=1$ and $g\in L^\infty$, without assuming an additional summability condition in the eigenvalues $(\alpha_n)_{n\geq 1}$ of the differential operator involved. In Subsection \ref{subsection SPDE comparison} we show that our Theorem \ref{thm:generalONB} allows for the same estimate under the same condition $\zeta<2d/(d-2)$, without assuming that the $\Sf$ are the eigenfunctions of the Laplacian.
    Apart from the works cited above and their follow-up or preceding articles, 
    we are not aware of other works using the condition $\mu\in\ell^\zeta(\Sf)$ for $\zeta\in [2,\infty)$ to obtain associated estimates of the noise term. 
    Indeed, instead of \eqref{eq: intro l zeta Sf condition} often a trace condition is assumed on either the covariance operator itself, or a power of the leading order differential operator $A$ or its semigroup $S$ applied to a  power of $R$, i.e. $\Tr (S(t)RR^*S^*(t))<\infty$, see for instance \cite{hairer2009introduction}.
    Similarly the authors in \cite{blomker2020stochastic} formulate a condition resembling \eqref{eq: intro l zeta Sf condition} for $q=2$, but with $\zeta=2$ and $\|f_n\|_{L^\infty}$ replaced by $\|f_n\|_{L^2}$, where the $(f_n)_{n\geq 1}$ are additionally assumed to form a complete orthonormal basis.

To the best of our knowledge, we are the first to provide a sharp characterisation of the convergence of \eqref{eq:BM} in Bessel potential spaces.

\subsection{Proofs strategy and structure of the paper}
\label{ss:proof_strategy}
We begin by giving an informal overview of the proof of Theorem \ref{thm:MgTmu delta} (which, in particular, contains Theorem \ref{thm:intro}). Using the language of $\g$-radonifying operators (see Subsection \ref{sss:gamma_spaces}), \eqref{eq:main2} can be equivalently formulated as 
\begin{equation}
\label{eq:mapping_Mg_estimates}
\|M_g R_\mu\|_{\g(L^2(\cO),H^{-s,q}(\cO))}\lesssim \|\mu\|_{\ell^\zeta(\Sf)}\|g\|_{L^\eta(\cO)},
\end{equation}
where $M_g$ is the multiplication operator by $g$, and the operator $R_\mu$ is defined as 
$
L^2(\cO) \ni h \mapsto \sum_{n\geq 1} \mu_n f_n (h,e_n)_{L^2(\cO)},
$
for a given $(e_n)_{n\geq 1 }$ complete orthonormal system of $L^2(\cO)$. 
From the equivalence of \eqref{eq:main2} and \eqref{eq:mapping_Mg_estimates}, the choice of such an orthonormal basis $(e_n)_{n\geq 1}$ is irrelevant. 

The main idea behind our approach to \eqref{eq:mapping_Mg_estimates} is to look at $M_g R_\mu$ as a \emph{bilinear} operator 
\begin{equation}
\label{eq:abstract_view_point_intro}
(g,\mu)\mapsto M_gR_\mu
\end{equation} 
for $g$ and $\mu$ taken from appropriate spaces. As shown in Section \ref{s:main_result}, the power of this viewpoint is that it allows us to use \emph{multilinear interpolation} (see Proposition \ref{prop: interpolation} or \cite[Section 4.4]{BeLo}) and reduce the full proof of \eqref{eq:mapping_Mg_estimates} with $\mu\in \ell^\zeta(\Sf)$ and $g\in L^\eta(\cO)$ to the following limiting cases: 
\begin{enumerate}[{\rm(1)}]
\item\label{it:into_interpolation_1} $\zeta=2$, $s\in (0,d)$,\  $q\in (1,\infty)$,\ $\eta\in (1,q)$\ and \ $\frac{s}{d} + \frac{1}{q}  \geq \frac{1}{\eta}$.
\item\label{it:into_interpolation_2} $\zeta=\infty$,   $s\in (\frac{d}{2},d)$,\ $q\in (2,\infty)$,\ $\eta\in (2,q)$\ and \
$\frac{s}{d} + \frac{1}{q}  \geq \frac{1}{\eta} + \frac{1}{2}$.
\end{enumerate}
One can readily check that the conditions \eqref{it:into_interpolation_1} and \eqref{it:into_interpolation_2} are indeed included in \eqref{eq:conditionsmainintro}.
The proofs of \eqref{eq:mapping_Mg_estimates} in the above cases are given in Subsection \ref{ss:endpoint_case_main_result}.

\smallskip

The case \eqref{it:into_interpolation_1} is the easiest, and is proven in Lemma \ref{lem:l2 2 delta}. It follows from Sobolev embeddings and the identification of $\gamma$-radonifying operators $\g(L^2(\cO),L^\eta(\cO))=L^\eta(\cO;L^2(\cO))$, see \eqref{eq: square fct char gamma}.

The case \eqref{it:into_interpolation_2} is much more involved. We discuss only the case $\cO=\R^d$, as the general one follows by extending by zero. In this case, the core idea is as follows. 
As $R_\mu\in \calL(L^2(\R^d))$ since $\mu\in \ell^\infty(\Sf)=\ell^\infty$ by assumption, and in light of the ideal property of $\g$-radonifying operators (see \eqref{eq:gamma_ideal}), it suffices to show 
\begin{equation}
\label{eq:M_g_gamma_radonifying}
M_g \in \g(L^2(\R^d),H^{-s,q}(\R^d)).
\end{equation}
Clearly, the above is equivalent to proving $ (1-\Delta)^{-s/2} M_g \in \g(L^2(\R^d),L^q(\R^d))$. Thus, to obtain \eqref{eq:M_g_gamma_radonifying} under the condition, it is enough to show that the following operator is $\g$-radonifying: 
\begin{align}
\label{eq:intro_gamma_operator}
L^2(\R^d)\ni
f\mapsto \big[(1-\Delta)^{-s/2} M_g f\big](x)= \int_{\R^d} \mathscr{G}_s(x-y) g(y)f (y)\,\di y\in L^\eta(\R^d)
\end{align}
where $\mathscr{G}_s=\mathscr{F}^{-1} ((1+4\pi^2|\cdot|^2)^{-s/2})$ is the Bessel potential kernel. 
Interestingly, the above operator, of convolution type, has already appeared in the literature in the context of Schr\"odinger operators \cite{Cwikel,LevSuZan,Simon76}, see the beginning of Section \ref{s:gamma_consequences} for additional comments.
The question whether the operator of convolution type in \eqref{eq:intro_gamma_operator} is $\g$-radonifying will be explored in Section \ref{s:gamma_consequences}. In particular, we prove an estimate that can be regarded as a $\g$-radonifying version of the Young inequality for convolutions involving \emph{weak Lebesgue spaces}, see Theorem \ref{t:gammaYoung}. We emphasize that working with weak-type spaces is essential to accommodate the singularity of the Bessel potential kernel, and the proofs of the operator in \eqref{eq:intro_gamma_operator} being $\g$-radonifying essentially rely on the use of weak Lebesgue spaces, see Remark \ref{r:necessity_of_weak_spaces}.

\smallskip

Behind the bilinear interpolation argument \eqref{eq:abstract_view_point_intro}, there is actually a much more general framework in which one can prove the estimate \eqref{eq:mapping_Mg_estimates}, for which the weighted sequence case $\mu\in \ell^\zeta(\Sf)$ is only a very special case. This will be illustrated in Subsection \ref{ss:Bird}, where, as a prototype example, we will consider the assignment $\mu \mapsto R_\mu$ replaced by $m\mapsto T_m$ where $T_m$ is the Fourier multiplier with symbol $m$ and $\cO=\R^d$, see \eqref{eq:Fourier_multiplier}. This allows us to cover Gaussian noises that are not only in the form \eqref{eq: intro general noise} but also defined via convolution on the whole space.
Explicit examples are given by the \emph{Mat\'ern random fields} in Theorem \ref{thm:regulartity_matern_with_multiplicative_g}.

In Section \ref{s:weight_necessary}, we again apply the general paradigm outlined in Subsection \ref{ss:Bird} using the interpolation argument \eqref{eq:abstract_view_point_intro}, but with $\mu$ from an unweighted sequence space. The main reason for that is to show the necessity of the weights  $\|f_n\|_{L^\infty(\cO)}^2$ in the condition \eqref{eq: intro l zeta Sf condition} to capture how the noise coloring $\zeta$ enters in the condition ruling $(s,q,\eta)$ (i.e.\ the first in \eqref{eq:conditionsmainintro}), and therefore, also necessary for sharpness and/or scaling invariance in applications to SPDEs, see Subsection \ref{ss:scaling_intro}.

\smallskip

In Section \ref{sec: SPDE section}, we investigate the consequences of Theorem \ref{thm:MgTmu delta} to SPDEs of elliptic and parabolic type. 
Here we focus on elliptic equations of Mat\'ern type, and the stochastic heat equation \eqref{eq:SHE_intro}. For the latter, we formally rewrite the noisy part of the equation $g\, \partial_t\cW$ as 
\begin{equation}
\label{eq:identity_cW_to_R_mu_W}
g \, \partial_t \cW = M_g R_\mu \dot{W}
\end{equation}
where $\cW$ is as in \eqref{eq: intro general noise} and $W$ its associated $L^2(\cO)$-cylindrical Brownian motion.
With the above identification and Theorem \ref{thm:MgTmu delta} at our disposal, Theorem \ref{thm: spde with delta no time} (which extends Theorem \ref{thm: spde intro}) is a consequence of stochastic maximal $L^p$-regularity estimates proven in \cite{MaximalLpregularity} (and \cite[Subsection 6.2]{ALV21} for homogeneous spaces).

\subsection*{Notation}
We collect here some basic notation. Additional conventions will be introduced as needed later in the text. In particular, the function spaces used here are introduced in Subsection \ref{subs:func}, the $\g$-radonifying operators $\g(H,X)$ in Subsection \ref{sss:gamma_spaces}.

As usual, we write $A \lesssim_{p_1,\dots,p_n} B$ or $A \gtrsim_{p_1,\dots,p_n} B$ in case there exists a constant $C$ depending only on the parameters $p_1,\dots,p_n$ such that $A\leq C B$ or $A\geq C B$. We write $A \eqsim_{p_1,\dots,p_n} B$ if $A \lesssim_{p_1,\dots,p_n} B$ and $B \lesssim_{p_1,\dots,p_n} A$ both hold.
The open half-line is denoted by $\R_+=(0,\infty)$. 
For $d\geq 1$, $\cO$ denotes either an open set of $\R^d$ or the $d$-dimensional torus $\T^d=\R^d/\Z^d$. 
The Lebesgue measure of a measurable set $S\subseteq \cO$ is denoted by $|S|$.
For two Banach spaces $X$ and $Y$, 
the set of bounded linear operators from $X$ to $Y$ is denoted by $\calL(X,Y)$. As usual, we set $\calL(X)=\calL(X,X)$. If $X$ and $Y$ are Hilbert spaces, then $\cL_2(X,Y)$ denotes the set of Hilbert-Schmidt operators from $X$ to $Y$, and $\cL_2(X)=\cL_2(X,X)$.

\emph{Probabilistic set-up.}  $(\Omega,\A,\P)$ is a given probability space. As usual, $\E[\cdot ]=\int_{\Omega} \cdot  \, \di \P$ denotes the corresponding expectation value.

\emph{Interpolation theory.} For $\theta\in (0,1)$ and $p\in (1,\infty)$, we denote by $(X_0,X_1)_{\theta,p}$ and $[X_0,X_1]_\theta$ the real and complex interpolation of the compatible spaces $X_0$ and $X_1$ (i.e.\ for $j\in \{0,1\}$, it holds that $X_j\embed V$ for some topological Hausdorff vector space $V$), respectively. For details on interpolation theory, the reader is referred to \cite{BeLo,InterpolationLunardi} or \cite[Appendix C]{HNVW1}.

\section{Preliminaries}
\label{sec: preliminaries}

\subsection{Function spaces}\label{subs:func}

For details on function spaces, the reader is referred to \cite{HNVW1} and \cite{Tri83,Tri95}. For $p\in [1,\infty]$ and a Banach space $X$ and an interval $I\subset\bR_+$ we define $L^p(I;X)$ as the space of strongly measurable functions $f:I\to X$, such that 
\begin{align}
    \|f\|_{L^p(I;X)}:=\Big( \int_I \|f(t)\|_{X}^p \,\di t\Big)^{1/p}<\infty.
\end{align}
For $p\in [1,\infty)$ and a measure space $(S,\Sigma,\nu)$ we define
the space $L^{p,\infty}(S,X)$, often called the weak $L^p$ space, as the space of strongly measurable functions $f:S\to X$ for which
\begin{align}
    \|f\|_{L^{p,\infty}(S;X)}:=\sup_{r>0}r\,\big[\nu(\|f\|_{X}>r)\big]^{1/p}<\infty.
\end{align}
If $X = \bR$ we simply write $L^{p,\infty}(S)$.

For $\sigma\in \R$ and $q\in (1, \infty)$ define the Bessel potential spaces $H^{\sigma,q}(\R^d)$ as all $f\in \Schw'(\R^d)$ such that $\|f\|_{H^{\sigma,q}(\R^d)} = \|(1-\Delta)^{\sigma/2} f\|_{L^q(\R^d)}<\infty$. 
Similarly one can define $\dot{H}^{s,q}(\R^d)$ and for details see \cite[Chapter 5]{Tri83}.
Similarly (see \cite[Subsection 1.3.3]{Grafakos2}), one defines the homogeneous
Bessel potential spaces $\dot{H}^{\sigma,q}(\R^d)$ as the set of all $f\in \S'(\R^d)/\mathcal{P}$ such that $\|(-\Delta)^{\sigma/2}f\|_{L^{q}(\R^d)}<\infty$, where $\mathcal{P}$ is the set of all real polynomials on $\R^d$. It is worth noting that homogeneous Sobolev functions can often be identified as standard distributions by suitably removing polynomials, exploiting their smoothness, see e.g.\ \cite[Theorem 2.31]{S18_Besov}. Interestingly, the latter result shows that the homogeneous spaces of our interest $\dot{H}^{-s,q}(\R^d)$ can be realized as a subset of the Schwartz distributions $ \S'(\R^d)$, for all $s\geq 0$ and $q\in (1,\infty)$.

For an open set $\mathcal{O}\subseteq \R^d$ and 
\(s\in\mathbb R\) and \(q\in(1,\infty)\), we define
\[
H^{s,q}(\mathcal O)
=
\{U|_{\mathcal O}: U\in H^{s,q}(\mathbb R^d)\},
\]
with the natural quotient norm
$
\|u\|_{H^{s,q}(\mathcal O)}
=
\inf\{\|U\|_{H^{s,q}(\mathbb R^d)}: U|_{\mathcal O}=u\}.
$
By duality, one can check that $H^{-s,q}(\cO)=(H^{s,q'}_0(\cO))^*$ provided $s-1/q'\not\in \N$ where $q'$ is the conjugate exponent of $q$ and $H^{s,q'}_0(\cO)$ is the closure of $C^\infty_{{\rm c}}(\cO)$ in $H^{s,q'}(\cO)$.
Furthermore, we set
\[
\widetilde H^{s,q}(\mathcal O)
=
\big\{U\in H^{s,q}(\mathbb R^d): \operatorname{supp} U\subseteq \overline{\mathcal O}\big\},
\]
endowed with the norm inherited from \(H^{s,q}(\mathbb R^d)\).
With \(q'=q/(q-1)\), we often employ the canonical duality identifications
\begin{equation}\label{eq:dualHtilde}
\big(H^{s,q}(\mathcal O)\big)^*=\widetilde H^{-s,q'}(\mathcal O),
\quad \text{ and }\quad 
\big(\widetilde H^{s,q}(\mathcal O)\big)^*
=
H^{-s,q'}(\mathcal O),
\end{equation}
with respect to the distributional duality inherited from \(\mathbb R^d\). 
These identifications hold for all \(s\in\mathbb R\) see \cite[Theorem 4.8.1]{Tri95}, and can be proven via the duality of Bessel potential spaces on $\R^d$-case and standard duality results, see e.g.\ \cite[Proposition B.1.4]{HNVW1}. Let us mention that $\wt{H}^{-s,q}(\cO)=H^{-s,q}(\cO)$ if $s<1/q'=1-1/q$, and in case $s-1/q'\not\in \N$, by the above mentioned duality of $H^{-s,q}(\cO)$ one has:
\begin{equation}
\label{eq:restriction_properties_of_tilde_spaces}
\big\|u|_{H^{s,q}_0(\cO)}\big\|_{H^{-s,q}(\cO)}
\lesssim \|u\|_{\wt{H}^{-s,q}(\cO)} \ \ \text{ for all } \ u\in \wt{H}^{-s,q}(\cO).
\end{equation}
Note that the above does not give an embedding of $\wt{H}^{-s,q}(\cO)$ into $H^{s,q}(\cO)$, as the former space might contain distributions supported on $\partial\cO$.

We will frequently use that, by Sobolev embeddings, if $\sigma_1-\frac{d}{q_1}\geq \sigma_0-\frac{d}{q_0}$ and $\sigma_1>\sigma_0$ and $q_0>q_1$, then $H^{s_1, q_1}(\mathcal{O})\hookrightarrow H^{s_0, q_0}(\mathcal{O})$, 
and analogous results hold for the space $\wt{H}^{s,q}(\cO)$.

\subsection{$\gamma$-radonifying operators}
\label{sss:gamma_spaces}
Here, we give definitions and basic results on $\gamma$-radonifying operators.
The reader is referred to \cite[Chapter 9]{HNVW2} for further results.

\smallskip

Let $H$ be a separable Hilbert space and $X$ a Banach space.
Let $(\g_n)_{n\geq 1}$ be a sequence of standard independent Gaussian variables over a given probability space $(\O,\mathcal{A},\P)$, and let $(h_n)_{n\geq 1}$ be an orthonormal basis of $H$. 
We say that a linear bounded operator $T:H\to X$ is \emph{$\g$-radonifying} if $\sum_{n\geq 1} \g_n T h_n$ converges in $L^2(\O;X)$ and we set
\begin{equation}
\label{eq:def_gamma_norm}
\|T\|_{\g(H,X)}=
\Big(\E \Big\| \sum_{n\geq 1} \g_n T h_n\Big\|_{X}^2\Big)^{1/2}.
\end{equation}
Note that $\g(H,X)\embed \calL(H,X)$. Moreover, from the comments below \cite[Definition 9.1.1]{HNVW2}, it also follows that the exponent $2$ can be replaced by any exponent $p\in [1,\infty)$.
Moreover, \cite[Theorem 9.1.17]{HNVW2} ensures that the above definition is independent of the orthonormal basis $(h_n)_{n\geq 1}$ chosen on $H$. 
In particular, any finite-rank operator is $\gamma$-radonifying.

An important property of the $\gamma$-spaces is the so-called \emph{ideal property}, see \cite[Theorem 9.1.10]{HNVW2}. The latter ensures that, if $R\in \calL(H,K)$, $T\in \g(K,X)$ and $S\in \calL(X,Y)$, the space of bounded operators from $X$ to $Y$, where $Y$ is another Banach space, then 
$STR\in \g(H,Y)$ and 
\begin{equation}
\label{eq:gamma_ideal}
\|STR\|_{\g(H,Y)}\leq\|S\|_{\calL(X,Y)} \|T\|_{\g(K,X)} \|R\|_{\calL(H,K)}.
\end{equation}

An alternative description of $\g(H,X)$ is available when $X$ is an $L^p(S)$ space over a $\sigma$-finite measure space $(S,\Sigma,\nu)$. Indeed, in the latter situation, it holds that $\g(H,L^p(S))=L^p(S,H)$ and 
\begin{align}
\label{eq: square fct char gamma}
    \|T\|_{\gamma(H,L^p(S))} \eqsim_p \Big\| \Big(  \sum_{n\geq 1} |Th_n|^2 \Big)^{1/2} \Big\|_{L^p(S)},
\end{align}
due to the $\gamma$-Fubini theorem \cite[Theorem 9.4.8]{HNVW2}. Moreover, $\gamma(H,H)=\calL_2(H)$ coincides with the Hilbert-Schmidt operators (see \cite[Proposition 9.1.9]{HNVW2}).

\section{The $\gamma$-Young inequality and consequences}
\label{s:gamma_consequences}
In this section, we study the mapping properties of the  convolution-type operator
\begin{align}\label{eq:Afg}
A_{f,g} h(x) := \int_{\R^d} f(x-y) g(y) h(y) \, \di y,\quad h\in L^0(\bR^d),
\end{align}
whenever it exists as a Lebesgue integral for almost all $x\in \R^d$ for given measurable functions $f,g:\bR^d\to \bR$.

The above operator $A_{f,g}$ appeared already in the study of eigenvalues of Schr\"odinger operators (e.g.\ $-\Delta+V$ with potential $V\in L^{d/2}(\R^d)$), and related operators in terms of Fourier multipliers were studied in \cite{Cwikel,Simon76, Simon05} in the context of Schatten class spaces. The reader is also referred to  \cite{LevSuZan} for further results and a historical account.
In this section, we give sufficient conditions under which $A_{f,g}$ defines an element of $\gamma(L^2(\R^d), L^q(\R^d))$ for $q\in [2, \infty)$. 

\begin{theorem}[$\gamma$-Young inequality]\label{t:gammaYoung}
Let $\eta,r,q\in (2, \infty)$ be such that $\frac{1}{q} + \frac{1}{2} = \frac{1}{r} + \frac{1}{\eta}$. Let $f\in L^{r,\infty}(\R^d)$ and $g\in L^{\eta}(\R^d)$. Then 
$
 A_{f,g}\in \gamma(L^2(\R^d),L^q(\R^d))
$
and 
\begin{align}
\label{eq:gamma_young}
\|A_{f,g}\|_{\gamma(L^2(\R^d),L^q(\R^d))} \lesssim_{\eta,q,r} \|f\|_{L^{r,\infty}(\R^d)} \|g\|_{L^{\eta}(\R^d)}. 
\end{align}
\end{theorem}

Note that the range of parameters $(q,r,\eta)$ in the above implies $q>\eta$ and $q>r$.
From the proof of the above, it follows that for $f,g\in L^0(\R^d)$ and $q\in [2,\infty)$ (see \eqref{eq:identityAfg} below),
\begin{equation}
\label{eq:gamma_estimate_for_A_is_symmetric}
\|A_{f,g}\|_{\g(L^2(\R^d),L^q(\R^d))}\eqsim_q \|A_{g,f}\|_{\g(L^2(\R^d),L^q(\R^d))}.
\end{equation}
In particular, one can reverse the roles of the weak-Lebesgue spaces in \eqref{eq:gamma_young}. More precisely, under the assumptions of Theorem \ref{t:gammaYoung}, it also holds that 
$$
\|A_{f,g}\|_{\gamma(L^2(\R^d),L^q(\R^d))} 
\lesssim_{\eta,q,r} \|f\|_{L^{r}(\R^d)} \|g\|_{L^{\eta,\infty}(\R^d)}.
$$
In Subsection \ref{ss:gamma_young_proof}, we also investigate the endpoint cases of the $\gamma$-Young inequality, namely the cases $\eta=2$ and $\eta=q$, see Propositions \ref{prop:eta2} and \ref{prop:etaq}, respectively. In both cases, we show that the $\gamma$-Young inequality \eqref{eq:gamma_young} can hold only if the weak Lebesgue space $L^{r,\infty}(\R^d)$ is replaced by the usual Lebesgue space $L^r(\R^d)$.

\smallskip

From Theorem \ref{t:gammaYoung} and its proof, we derive the following result, which is of central importance in the study of Sobolev embeddings of Gaussian random sums as Theorems \ref{thm:intro} and \ref{thm:MgTmu delta}.

\begin{proposition}[$\gamma$-bounds for multiplication operators in Sobolev spaces]
\label{prop:M_g_Sobolev}
Let $M_g$ be the multiplication operator with a measurable function $g:\R^d \to \R$. 
Then the following hold:
\begin{enumerate}[{\rm(1)}]
\item\label{it:M_g_Sobolev1} {\rm (Sufficient conditions)}
If $(s,q,\eta)$ satisfy one of the following conditions:
\begin{enumerate}[{\rm(a)}]
\item\label{it:M_g_Sobolev1a}
$s\in(\frac{d}{2},d)$, $ q\in (2, \infty)$, $\eta\in (2, q)$ and $\frac{s}{d} + \frac{1}{q}  \geq \frac{1}{\eta} + \frac{1}{2}$;
\item\label{it:M_g_Sobolev1b} 
$s>\frac{d}{2}$ and $\eta=q\geq 2$;

\item\label{it:M_g_Sobolev1c}
$s>\frac{d}{2}$, $\eta=2$, $q\in [2,\infty)$ and $\frac{s}{d} + \frac{1}{q} > 1$,
\end{enumerate}
then for all $g\in L^\eta(\R^d)$ one has $M_g\in \g(L^2(\R^d),H^{-s,q}(\R^d))$ and 
\[
\|M_g\|_{\gamma(L^2(\R^d), H^{-s,q}(\R^d))} 
\lesssim_{d,s,q,\eta} \|g\|_{L^\eta(\R^d)}.
\]
\item\label{it:M_g_Sobolev2} {\rm (Necessary conditions)} If for all $g\in L^{\eta}(\R^d)$ one has $\|M_g\|_{\gamma(L^2(\R^d), H^{-s,q}(\R^d))} 
\lesssim_{d,s,q,\eta} \|g\|_{L^\eta(\R^d)}
$
for some $\eta,q\in [2,\infty)$ and $s\geq 0$, then 
$\frac{s}{d}+\frac{1}{q}\geq \frac{1}{\eta}+\frac{1}{2}$, $\eta\leq q$, and $s>\frac{d}{2}$.
\item\label{it:M_g_Sobolev3} {\rm (Necessary conditions)} If additionally to \eqref{it:M_g_Sobolev2} one has $\frac{s}{d}+\frac{1}{q}= \frac{1}{\eta}+\frac{1}{2}$, then $s\in (\frac{d}{2},d)$ and $\eta\in (2,q)$.
\end{enumerate}
\end{proposition}

A similar result also holds for homogeneous Sobolev spaces.

\begin{proposition}[$\gamma$-bounds for multiplication operators in homogeneous Sobolev spaces]
\label{prop:M_g_Sobolev_hom} 
Assume that $s\geq 0$, $\eta,q\in [2,\infty)$ and $g\in L^\eta(\R^d)$. 
Let $M_g$ be the multiplication operator by $g$.  
Then $M_g\in \g(L^2(\R^d), \dot{H}^{-s,q}(\R^d))$ and 
\[
\|M_g\|_{\gamma(L^2(\R^d), \dot{H}^{-s,q}(\R^d))} 
\lesssim_{d,s,q,\eta} \|g\|_{L^\eta(\R^d)}.
\]
if and only if 
\begin{equation}
\label{eq:conditions_parameters_homogeneous_gamma}
\frac{s}{d}+\frac{1}{q}= \frac{1}{\eta}+\frac{1}{2} , \qquad s\in \Big(\frac{d}{2},d\Big) \ \quad \text{ and }\quad \  \eta\in (2,q).
\end{equation}
\end{proposition}

The conditions in the previous result are also contained in item \eqref{it:M_g_Sobolev1a} of Proposition \ref{prop:M_g_Sobolev}. In light of Proposition \ref{prop:M_g_Sobolev_hom} and the scaling invariance of the homogeneous Sobolev spaces $\dot{H}^{-s,q}(\R^d)$, the condition $\frac{s}{d}+\frac{1}{q}= \frac{1}{\eta}+\frac{1}{2}$
respects the scaling of the spaces involved, and will be referred to as the \emph{sharp} case, as in all the other cases considered in Proposition \ref{prop:M_g_Sobolev}, there is a loss of smoothness (and hence, a break of the scaling invariance). To obtain the sharp case, we necessarily require the weak Lebesgue spaces in Theorem \ref{t:gammaYoung}, see Remark \ref{r:necessity_of_weak_spaces}.

\smallskip

Before presenting the proofs, let us discuss the connection between Theorem \ref{t:gammaYoung} and Proposition \ref{prop:M_g_Sobolev} (the one with Proposition \ref{prop:M_g_Sobolev_hom} is similar). Note that the condition $M_g\in \g(L^2(\R^d),H^{-s,q}(\R^d))$ holds if and only if $(1-\Delta)^{-s/2} M_g \in \g(L^2(\R^d),L^q(\R^d))$. Moreover, since 
\begin{equation}
\label{eq:delta_Bessel_potential_kernel}
\big[(1-\Delta)^{-s/2} M_g f\big](x)= \int_{\R^d} \mathscr{G}_s(x-y) g(y)f (y)\,\di y = A_{\mathscr{G}_s,g}f(x)
\end{equation}
where $\mathscr{G}_s=\mathscr{F}^{-1} ((1+4\pi^2|\cdot|^2)^{-s/2})$ is the Bessel potential kernel, to prove Proposition \ref{prop:M_g_Sobolev} it is enough to estimate the weak Lebesgue norm of $\cG_s$. In particular, we will see that the use of the weak Lebesgue spaces in Theorem \ref{t:gammaYoung} is necessary to obtain sharp estimates for Gaussian series in the next section, see Remark \ref{r:necessity_of_weak_spaces}. 

\begin{remark}\label{rem:groups}
It will be clear from the proofs below that Theorem \ref{t:gammaYoung} holds in case $\R^d$ is replaced by a locally compact Abelian group equipped with the Haar measure (e.g.\ $\T^d$).  
In particular, this implies that Proposition \ref{prop:M_g_Sobolev} holds true also in the periodic case.  
\end{remark}

This section is organized as follows. In Subsection \ref{ss:gamma_young_proof}, we prove Theorem \ref{t:gammaYoung} and we investigate the above-mentioned endpoint cases $\eta=2$ and $\eta=q$ in Propositions \ref{prop:eta2} and \ref{prop:etaq}, respectively. Finally, in Subsection \ref{ss:proof_corollary_Mg}, we prove Proposition \ref{prop:M_g_Sobolev}.

\subsection{Proof of Theorem \ref{t:gammaYoung} and endpoints}
\label{ss:gamma_young_proof}
We begin by proving Theorem \ref{t:gammaYoung}.

\begin{proof}[Proof of Theorem \ref{t:gammaYoung}]
Let $(h_k)_{k\geq 1}$ be an orthonormal basis for $L^2(\R^d)$. From the square-function characterization of $\gamma(L^2(\R^d),L^q(\R^d))$ in \eqref{eq: square fct char gamma}, it follows that 
\begin{align*}
\|A_{f,g}\|_{\gamma(L^2(\R^d),L^q(\R^d))}& \eqsim_q \Big\|\Big(\sum_{k\geq 1} |A_{f,g} h_k|^2 \Big)^{1/2} \Big\|_{L^q(\R^d)}.
\end{align*}
The Parseval identity in $L^2(\R^d)$ applied pointwise in $\R^d$ implies
\begin{align*}
 \sum_{k\geq 1} |A_{f,g} h_k(x)|^2 
 &=
  \sum_{k\geq 1} \Big|\int_{\R^d} f(x-y) g(y) h_k(y)\,\di y\Big|^2   = \int_{\R^d}|f(x-y)|^2 |g(y)|^2 \, \di y,
\end{align*}
whenever $x\in \R^d$ is such that the last term in the previous is finite. Combining the previous observations, we have
\begin{align}\label{eq:identityAfg}
\|A_{f,g}\|_{\gamma(L^2(\R^d),L^q(\R^d))}^2\eqsim_q \Big\|x\mapsto \int_{\R^d} F(x-y) G(y) \,\di y\Big\|_{L^{q/2}(\R^d)},
\end{align}
where $F = |f|^2$ and $G = |g|^2$.  Applying Young's inequality for weak type spaces (see \cite[Theorem 1.4.25]{Grafakos1}) we obtain, for $\tfrac{2}{q}+1 = \tfrac{2}{r}+\tfrac{2}{\eta}$,
\begin{align*}
\Big\|x\mapsto \int_{\R^d} F(x-y) G(y) \, \di y\Big\|_{L^{q/2}(\R^d)} & \lesssim_{\eta,q,r} \|F\|_{L^{r/2,\infty}(\R^d)} \|G\|_{L^{\eta/2}(\R^d)} 
 = \|f\|_{L^{r,\infty}(\R^d)}^2 \|g\|_{L^{\eta}(\R^d)}^2. 
\end{align*}
The claimed estimate \eqref{eq:gamma_young} now follows from the above and \eqref{eq:identityAfg}. 
\end{proof}

As commented below Theorem \ref{t:gammaYoung}, below we discuss the validity of the $\gamma$-Young inequality \eqref{eq:gamma_young}. For that purpose, let us note that, from \cite[Chapter 4]{Simon05}, it follows that $A_{f,g}\in \gamma(L^2(\R^d), L^2(\R^d)) = S^2(L^2(\R^d))$ if and only if $f,g\in L^2(\R^d)$. Moreover, in this case 
\begin{align}\label{eq:Simon}
\|A_{f,g}\|_{\gamma(L^2(\R^d), L^2(\R^d))} = \|A_{f,g}\|_{S^2(L^2(\R^d))} = \|f\|_{L^2(\R^d)} \|g\|_{L^2(\R^d)}.
\end{align}
Below, we extend this result to $q>2$ if $\eta = 2$. 

\begin{proposition}[Endpoint case $\eta=2$]\label{prop:eta2}
Let $q\geq 2$ and $\eta=2$. For a measurable function $f:\R^d\to \C$, the following are equivalent 
  \begin{enumerate}[\rm (1)]
  \item\label{it:eta21} There is $C$ such that for all $g\in L^2(\R^d)$, $\|A_{f,g}\|_{\gamma(L^2(\R^d), L^q(\R^d))}\leq C \|g\|_{L^2(\R^d)}$; \item\label{it:eta23} $f\in L^q(\R^d)$. 
  \end{enumerate}
Moreover, letting $C_{f}$ denote the infimum over all admissible constants $C$ in \eqref{it:eta21}, one has 
$C_{f} \eqsim_{q}  \|f\|_{L^q(\R^d)}.$ 
\end{proposition}

From \eqref{eq:gamma_estimate_for_A_is_symmetric}, the condition in \eqref{it:eta21} can be equivalently formulated with $A_{g,f}$ in place of $A_{f,g}$.

\begin{proof}
We first prove \eqref{it:eta23}$\Rightarrow$\eqref{it:eta21}. 
Let $g\in L^2(\R^d)$, and set $F = |f|^2$ and $G = |g|^2$. Using \eqref{eq:identityAfg} and Young's inequality for convolutions, we find that
\begin{align}
\label{eq: A f g F G ineq}
\|A_{f,g}\|_{\gamma(L^2(\R^d),L^q(\R^d))}^2  \eqsim_q \|F*G\|_{L^{q/2}(\R^d)} 
 \leq \|F\|_{L^{q/2}(\R^d)} \|G\|_{L^{1}(\R^d)} 
 = \|f\|_{L^q(\R^d)}^2 \|g\|_{L^2(\R^d)}^2. 
\end{align}
Next, we prove \eqref{it:eta21}$\Rightarrow$\eqref{it:eta23}. 
For \(m\geq 1\), set $f_{m}:=f\,\mathbf 1_{\{|f|\leq m\}}\mathbf 1_{B_m}$ and $F_{m}:=|f_{m}|^2$.
Then \(f_{m}\in L^q(\mathbb R^d)\). Moreover, for every
\(g\in L^2(\mathbb R^d)\), writing \(G=|g|^2\), we have
\(0\leq F_m\leq F\), and hence
\[
\|F_{m}*G\|_{L^{q/2}(\mathbb R^d)}
\leq
\|F*G\|_{L^{q/2}(\mathbb R^d)}.
\]
Let
\begin{align}
\label{eq: def constant C 2 f}
C_{f}:=\sup_{\|g\|_{L^2(\R^d)}\leq 1} \|A_{g,f}\|_{\gamma(L^2(\R^d),L^q(\R^d))}.
\end{align}
Using the square-function identity \eqref{eq:identityAfg}, this gives
\[
\|A_{f_{m},g}\|_{\gamma(L^2(\mathbb R^d),L^q(\mathbb R^d))}
\lesssim_q
\|A_{f,g}\|_{\gamma(L^2(\mathbb R^d),L^q(\mathbb R^d))}.
\]
Consequently,
\[
C_{f_{m}}
:=
\sup_{\|g\|_{L^2}\leq1}
\|A_{f_{m},g}\|_{\gamma(L^2,L^q)}
\lesssim_q C_f .
\]

Let $G$ be nonnegative and such that $\|G\|_{L^1(\bR^d)}=1$. 
Then let $G_n(x) = n^dG(nx)$, for $n\geq 1$ and define $g_n:=G_n^{1/2}$, so that 
$\|g_n\|_{L^2}^2 = \|G_n\|_{L^1} = 1$ for all $n\geq 1$.
Thus, using \cite[Proposition 1.2.32]{HNVW1} and \eqref{eq:identityAfg}, we find that 
\begin{align}
\|f_{m}\|_{L^q(\R^d)}^2 = \|F_{m}\|_{L^{q/2}(\R^d)}& = \lim_{n\to \infty} \|G_n*F_{m}\|_{L^{q/2}(\R^d)} \eqsim_q\lim_{n\to\infty}\|A_{f_m,g_n}\|_{\gamma(L^2(\R^d),L^q(\R^d))}^2
\leq C_{f}^2.
\end{align}
Letting $m\to \infty$, we obtain $f\in L^q(\R^d)$ and the required estimate. For the final assertion, it suffices to combine \eqref{eq: def constant C 2 f} with \eqref{eq: A f g F G ineq}.
\end{proof}

\begin{proposition}[Endpoint case $\eta=q$]\label{prop:etaq}
Let 
$\eta\in [2, \infty)$. For a measurable function $f:\R^d\to \C$, the following are equivalent 
  \begin{enumerate}[\rm (1)]
  \item\label{it:etaq1} There is $C$ such that for all $g\in L^\eta(\R^d)$, $\|A_{f,g}\|_{\gamma(L^2(\R^d), L^{\eta}(\R^d))}\leq C \|g\|_{L^\eta(\R^d)}$;
  \item\label{it:etaq3} $f\in L^2(\R^d)$. 
  \end{enumerate}
Moreover, letting $C_{f}$ denote the infimum over all admissible constants $C$ in \eqref{it:etaq1}, then one has 
$C_{f} \eqsim  \|f\|_{L^2(\R^d)}.$
\end{proposition}

As above, from \eqref{eq:gamma_estimate_for_A_is_symmetric}, the condition in \eqref{it:etaq1} can be equivalently formulated with $A_{f,g}$ replaced by $A_{g,f}$.
To prove the above, we use the following lemma for convolutions with positive kernels. 

\begin{lemma}
\label{lem:positive}
Let $k\geq 0$ be measurable and let $p\in [1,\infty)$. Then the following are equivalent 
\begin{enumerate}[{\rm(1)}]
\item\label{it:positive1} There exists a $C$ such that for all $\phi\in L^p(\R^d)$, $\|k*\phi\|_{L^p(\R^d)}\leq C\|\phi\|_{L^p(\R^d)}$.
\item\label{it:positive2} There exists a $C$ such that for all $\phi\in L^p_+(\R^d)$, $\|k*\phi\|_{L^p(\R^d)}\leq C\|\phi\|_{L^p(\R^d)}$.
\item\label{it:positive3} $k\in L^1(\R^d)$.
\end{enumerate}
Moreover, in this case $\|k\|_{L^1(\R^d)} = \inf C$, where the infimum is taken over all admissible constants $C$ in \eqref{it:positive1}. 
\end{lemma}
\begin{proof}
\eqref{it:positive1}$\Rightarrow$\eqref{it:positive2} is trivial. 
\eqref{it:positive3}$\Rightarrow$\eqref{it:positive1} is immediate from Minkowski's convolution inequality. 
To prove
\eqref{it:positive2}$\Rightarrow$\eqref{it:positive3}, let $B_R$ and $B_L$ denote the balls centered at $0$ of radius $R$ and $L$, respectively. Define $k_L = \one_{B_L} k$, and $f_R =  \one_{B_R}$. 
Then, by the positivity of $k$, we see that
\begin{align*}
\|k_L*f_R\|_{L^p(\R^d)} \leq \|k*f_R\|_{L^p(\R^d)} \leq C\|f_R\|_{p} = C |B_R|^{1/p},
\end{align*}
where we used again Minkowski's convolution inequality, and where $|B_R|$ denotes the Lebesgue measure of the set $B_R$.
On the other hand, by Fubini's theorem and H\"older's inequality, one has
\begin{align}
\label{eq: conv equality k}
  \|k_L\|_{L^1(\R^d)} \|f_R\|_{L^1(\R^d)} &= \|k_L*f_R\|_{L^1(\R^d)} 
  \\ & \leq \|k_L*f_R\|_{L^p(\R^d)} |B_{R+L}|^{1/p'}
   \leq C |B_R|^{1/p} | B_{R+L}|^{1/p'},
\end{align}
where for $p>1$, $p' = \tfrac{p}{p-1}$. If $p=1$, then \eqref{eq: conv equality k} already gives the result.
Therefore, $\|k_L\|_{L^1(\R^d)}\leq C \frac{|B_{R+L}|^{1/p'}}{|B_R|^{1/p'}}$.
Letting $R\to \infty$ we see that $\|k_L\|_{L^1(\R^d)}\leq C$. Letting $L\to \infty$, the monotone convergence theorem gives that $\|k\|_{L^1(\R^d)}\leq C$.
\end{proof}

\begin{proof}[Proof of Proposition \ref{prop:etaq}]
\eqref{it:etaq3}$\Rightarrow$\eqref{it:etaq1} can be proved as in Proposition \ref{prop:eta2} by using \eqref{eq:identityAfg} and the Young inequality. 
To prove \eqref{it:etaq1}$\Rightarrow$\eqref{it:etaq3} note that as in \eqref{eq:identityAfg} we see that for all $G:=|g|^2\in L^{\eta/2}(\R^d)$ and $F:=|f|^2$,
\begin{align*}
\|G*F\|_{L^{\eta/2}(\R^d)} \eqsim_{\eta} \|A_{f,g}\|_{\gamma(L^2(\R^d),L^\eta(\R^d))}^2 \leq C_{f}^2 \|g\|_{L^\eta(\R^d)}^2 = C_{f}^2 \|G\|_{L^{\eta/2}(\R^d)}. 
\end{align*}
Since $F\geq 0$, it follows from Lemma \ref{lem:positive} that $\|F\|_{L^1(\R^d)}\lesssim_{\eta} C_{f}$. 
\end{proof}

\begin{remark}[Optimality of the conditions $\eta\geq 2$ and $\eta>q$]
\label{rem:noneetaq}
In this remark, we prove that none of the results proven in this subsection holds for either $\eta>q\geq 2$ or $\eta\in (0,2)$. 
The first follows from \eqref{eq:identityAfg} and the fact that, if $G\mapsto F*G$ is bounded from $L^{\eta/2}(\R^d)$ into $L^{q/2}(\R^d)$, and $\eta> q\geq 2$, then $F = 0$ due to \cite[Theorem 2.5.6]{Grafakos1}.  Next, we provide an alternative argument that works for any $q,\eta\in (0,\infty)$. Fix $F\in C^\infty_{{\rm c}}(\R^d)$, and suppose that $\|F*G\|_{L^{q/2}(\R^d)}\lesssim \|G\|_{L^{\eta/2}(\R^d)}$ holds for all $G\in C^\infty_c(\R^d)$ with implicit constant independent of $G$. For $\phi\in C^\infty_{{\rm c}}(\R^d)$ and $N\geq 1$, consider $G(x) = \sum_{k=1}^N \phi(x-kM)$, where $M$ will be chosen below. Then $F*G(x) = \sum_{k=1}^N F*\phi(x-kM)$. Choosing $M$ large enough, we have that the functions in the sum appearing in $G$ and $F*G$ have disjoint support. Therefore, 
\[\|G\|_{L^{\eta/2}(\R^d)} \eqsim_{\phi} N^{2/\eta} \qquad \text{ and } 
\qquad \|F*G\|_{L^{q/2}(\R^d)}\eqsim_{\phi, F} N^{2/q}.\]
Sending $N\to \infty$ in the estimate $\|F*G\|_{L^{q/2}(\R^d)}\lesssim \|G\|_{L^{\eta/2}(\R^d)}$ yields $\eta\leq q$. 

Finally, we show the necessity of $\eta\geq 2$. Indeed, if $\eta\in (0,2)$ and $F\in L^{q/2}(\R^d)$ takes values in $[0,\infty)$, then $\|F*G\|_{L^{q/2}(\R^d)}\lesssim \|G\|_{L^{\eta/2}(\R^d)}$ cannot hold with implicit constant independent of $G$. 
Indeed, suppose that the estimate holds. If $\eta\in (0,2)$ and $G\in C^{\infty}_{{\rm c}}(\R)$ such that $\int_{\R^d}G(x)\,\di x =1$, then taking $G_{\varepsilon} = \varepsilon^{-d} G(x/\varepsilon)$, it follows that $F*G_{\varepsilon}\to F$ in $L^{q/2}(\R^d)$. On the other hand, $\|G_{\varepsilon}\|_{L^{\eta/2}(\R^d)}\to 0$ as $\varepsilon\downarrow 0$, implying $F=0$. 
\end{remark}

\subsection{Proof of Propositions \ref{prop:M_g_Sobolev} and \ref{prop:M_g_Sobolev_hom}}
\label{ss:proof_corollary_Mg}
We begin by proving Proposition \ref{prop:M_g_Sobolev}, which concerns inhomogeneous Sobolev spaces.

\begin{proof}[Proof of Proposition \ref{prop:M_g_Sobolev}]
We begin by collecting some facts. First, as noticed below Proposition \ref{prop:M_g_Sobolev}, it is enough to show
$(1-\Delta)^{-s/2} M_g\in\g(L^2(\R^d),L^{q}(\R^d))$ with a corresponding estimate. As in \eqref{eq:delta_Bessel_potential_kernel}, we have
\begin{equation}
\label{eq:identity_AGg_proof}
(1-\Delta)^{-s/2} M_g = A_{\mathscr{G}_s,g} \qquad \text{ where }\qquad \mathscr{G}_s=\mathscr{F}^{-1}((1+|\cdot|^2)^{-s/2}),
\end{equation}
where the operator $A_{\cG_s,g}$ is given in \eqref{eq:Afg}, with $(\cG_s,g)$ in place of $(f,g)$, and $\mathscr{F}^{-1}$ denotes the inverse Fourier transform. It is well-known that (see for instance \cite[p.16]{Grafakos2})
\begin{align}
\label{eq: bessel potential piecewise}
    \cG_s(x) \eqsim \begin{cases}
        |x|^{s-d} & \text{ if }\  |x|\leq 2,\\
        e^{-\frac{|x|}{2}} & \text{ if } \ |x|>2.
    \end{cases}
    \end{align}
We now divide the proof into several steps.

\smallskip

\emph{Proof of \eqref{it:M_g_Sobolev1} in case $
s\in(\frac{d}{2},d)$, $ q\in (2, \infty)$, $\eta\in (2, q)$ and $\frac{s}{d} + \frac{1}{q}  \geq \frac{1}{\eta} + \frac{1}{2}$.} Without loss of generality, we can assume the latter estimate is an identity. 
Let $r\in (2,\infty)$ be given by 
$\frac{1}{q}+\frac{1}{2}=\frac{1}{\eta}+\frac{1}{r}$.
Then $r\in (2, \infty)$ due to $s\in (d/2, d)$. 
From \eqref{eq:identity_AGg_proof} and Theorem \ref{t:gammaYoung}, it follows that 
$$
\|(1-\Delta)^{-s/2}M_g\|_{\g(L^2(\R^d),L^q(\R^d))}
\lesssim_{\eta,q,d} \|\cG_s\|_{L^{r,\infty}}\|g\|_{L^\eta}.
$$
From \eqref{eq: bessel potential piecewise}, it follows that, 
\begin{align*}
|\{x\in\R^d\,:\, \cG_s(x)\geq \rho\}|
\lesssim \rho^{d/(s-d)} = \rho^{-r} \ \text{ for all }\ \rho>0.
\end{align*}
Hence $\cG_s\in L^{r,\infty}(\R^d)$. 

\smallskip

\emph{Step 2: Remaining cases in \eqref{it:M_g_Sobolev1}.}
Let $\eta=q\geq 2$ and $s>d/2$. Then from \eqref{eq: bessel potential piecewise} it follows that $\cG_s\in L^2(\bR^d)$. Hence, by Proposition \ref{prop:etaq},
\begin{align}
\label{eq: example Bessel pot estimate 2}
    \|(1-\Delta)^{-s/2} M_g\|_{\gamma(L^2(\R^d), L^\eta(\R^d))} \lesssim_{d,s,\eta} \|g\|_{L^\eta(\R^d)}.
\end{align}
Meanwhile, in the case $s>d/2$, $\eta=2$, $q\in [2,\infty)$ and $\frac{s}{d} + \frac{1}{q}  > 1$, we have $\cG_s\in L^q(\R^d)$. Thus, Proposition \ref{prop:eta2} gives
\begin{align}
    \|(1-\Delta)^{-s/2} M_g\|_{\gamma(L^2(\R^d), L^q(\R^d))} \lesssim_{d,s} \|g\|_{L^2(\R^d)}.
\end{align}

\smallskip

\emph{Step 3: Proof of \eqref{it:M_g_Sobolev2}}.
The assumptions in \eqref{it:M_g_Sobolev2} and  \eqref{eq:identityAfg} imply,
\begin{align}
\label{eq: cond tilde g Bessel}
\Big\|x\mapsto \int_{\R^d} |\cG_s(x-y)|^2  \wt{g}(y)\,\di y\Big\|_{L^{q/2}(\R^d)} 
 \lesssim \|\wt{ g}\|_{L^{\eta/2}(\R^d)}, \ \ \ \wt{g}\in L^{\eta/2}(\R^d) \ \text{with} \ \wt{g}\geq 0.
 \end{align}
First, let us note that, as in Remark \ref{rem:noneetaq}, the above inequality implies $\eta\leq q$. 

Let $B_\delta$ be the ball centered at the origin with radius $\delta$. Applying the above with $\wt{ g} = \one_{B_{\delta}}$ for $\delta \in (0,1)$ and using the behaviour of $\cG_s$ on $B_1$ in \eqref{eq: bessel potential piecewise}, 
 \begin{align}
 \nonumber
 \Big\| x\mapsto \int_{B_\delta} |x-y|^{2(s-d)}\one_{B_\delta}(y)\,\di y \Big\|_{L^{q/2}(\R^d)}
       &\lesssim
     \Big\| x\mapsto \int_{\R^d} |\cG_s(x-y)|^2\one_{B_\delta}(y)\,\di y \Big\|_{L^{q/2}(\R^d)}\\ &\stackrel{\eqref{eq: cond tilde g Bessel}}{\lesssim} \delta^{\frac{2d}{\eta}}.
      \label{eq: upper bound integral delta}
 \end{align}
 From the above, it follows that $s>\tfrac{d}{2}$, as otherwise the left-most side of the above inequality would diverge. To see the necessity of $\tfrac{s}{d}+\tfrac{1}{q}\geq \tfrac{1}{\eta}+\tfrac{1}{2}$, for $x\in \R^d$ such that $|x|<\delta$, we compute
\begin{align}\label{eq: integral estimate 1}
    \int_{B_\delta} |x-y|^{2(s-d)}\,\di y & \geq \int_{|z|<\delta -|x|} |z|^{2(s-d)}\,\di z \\ & \eqsim_d \int_0^{\delta-|x|} r^{2(s-d)+d-1}\,\di r
    \eqsim_{s,d} (\delta - |x|)^{2s-d}.
\end{align}
Next, we compute 
\begin{align}
    \int_{B_\delta}(\delta-|x|)^{\frac{q}{2}(2s-d)}\,\di x 
    &\eqsim_d \int_0^\delta (\delta - r)^{\frac{q}{2}(2s-d)}r^{d-1}\,\di  r\\
    \label{eq: beta function}
    & = \delta^{\frac{q}{2}(2s-d)+d}  \int_0^1 (1-r)^{\frac{q}{2}(2s-d)}r^{d-1}\, \di r \eqsim_{q,s,d} \delta^{\frac{q}{2}(2s-d)+d},
\end{align}
where the integral on the right-hand side in \eqref{eq: beta function} is a Beta function, evaluating to a constant independent of $\delta$.
Combining the latter with \eqref{eq: upper bound integral delta} and \eqref{eq: integral estimate 1}, we get
\begin{align}
    \delta^{2s-d+\frac{2d}{q}}\lesssim  
    \Big\|x\mapsto \int_{B_\delta} |x-y|^{2(s-d)}\one_{B_\delta(0)}(y)\,\di y \Big\|_{L^{q/2}(\R^d)}\lesssim \delta^{\frac{2d}{\eta}}.
\end{align}
The above, for $\delta\to 0$, implies $2s-d+\frac{2d}{q}\geq \tfrac{2d}{\eta}$, that is $\tfrac{s}{d}+\tfrac{1}{q}\geq \tfrac{1}{\eta}+\tfrac{1}{2}$ as desired.

\smallskip

Next, we prove the final assertion in \eqref{it:M_g_Sobolev3}.
Suppose that $\frac{s}{d}+\frac{1}{q}=\frac{1}{\eta} + \frac12$, then due to $s>d/2$ we see that $\eta<q$. 
Moreover, as $\eta\geq 1$, we have $\frac{s}{d}\leq \frac{1}{\eta} + \frac12\leq 1$ and thus $s\leq d$.
Now we suppose that $\eta=2$ and show that this leads to a contradiction. Indeed, as before, from \eqref{eq: cond tilde g Bessel} we get 
\begin{align}\label{eq:L1est}
\Big\|x\mapsto \int_{B_r} |\cG_s(x-y)|^2  \widetilde{g}(y) \,\di y\Big\|_{L^{q/2}(B_r)} 
 \lesssim \|\widetilde{g}\|_{L^{1}(B_r)},
 \end{align}
 for $r>0$.
Taking $ \widetilde{g} := \widetilde{g}_\delta:= |B_{\delta}|^{-1}\one_{B_{\delta}}$ with $\delta<r<1$ and letting $\delta\downarrow 0$ we obtain 
that (using that $\cG_s$ is smooth outside zero)
\begin{align*}
\int_{B_r} |\cG_s(x-y)|^2  \wt{g}_{\delta}(y) \,\di y\to |\cG_s(x)|^2 \ \ \text{for all $x\neq 0$.}
\end{align*}
Therefore, from Fatou's lemma and \eqref{eq:L1est} we obtain that 
\begin{align}
    \|\cG_s\|_{L^q(B_r)}^2 &= \big\||\cG_s|^2\big\|_{L^{q/2}(B_r)} \leq \lim_{\delta \downarrow 0} \Big\|x\mapsto \int_{B_r} |\cG_s(x-y)|^2  \wt{g}_{\delta}(y)\,\di y\Big\|_{L^{q/2}(B_r)} \lesssim 1.
\end{align}
Since $\cG_s(x)\sim |x|^{s-d}$ by \eqref{eq: bessel potential piecewise} 
for $x\to 0$, the previous bound implies the integrability of $|x|^{q(s-d)}$ near $0$, and therefore $(s-d)q>-d$, or equivalently $\frac{s}{d} -1 > -\frac{1}{q}$. 
Due to $\frac{s}{d}+\frac{1}{q}=\frac{1}{\eta} + \frac12$ we get $\tfrac{1}{\eta}>\tfrac{1}{2}$, which contradicts $\eta=2$.  
\end{proof}

It remains to prove Proposition \ref{prop:M_g_Sobolev_hom}.

\begin{proof}[Proof of Proposition \ref{prop:M_g_Sobolev_hom}]
We divide the proof into two steps.

\smallskip

\emph{Step 1: If \eqref{eq:conditions_parameters_homogeneous_gamma} holds, then $\|M_g\|_{\g(L^2(\R^d),\dot{H}^{-s,q}(\R^d))}\lesssim \|g\|_{L^\eta(\R^d)}$.} 
To begin, similarly to \eqref{eq:delta_Bessel_potential_kernel} it is enough to consider the operator $(-\Delta)^{-s/2} M_g$. Note that 
\begin{equation}\label{eq:RieszA}
(-\Delta)^{-s/2} M_g= A_{\cR_s,g} 
\end{equation}
where $\cR_s=(2\pi)^{s}\mathscr{F}^{-1}(|\cdot|^{-s})$ denotes the Riesz kernel. It is well-known that
$$
\cR_s(x)= C_{d,s} |x|^{s-d} \ \text{ for }\  x\in \R^d,
$$
see for instance \cite[p.9]{Grafakos2}.
In particular, $|\{x\in \R^d\,:\, \cR_s(x)\geq \rho\}|\eqsim \rho^{d/(s-d)}$ for $\rho>0$ and hence  $\cR_s\in L^{r,\infty}(\R^d)$
for $r=d/(d-s)$. Thus, from Theorem \ref{t:gammaYoung}, as before
$$
\|(-\Delta)^{-s/2} M_g\|_{\g(L^2(\R^d),L^q(\R^d))}
\lesssim \|\cR_s\|_{L^{r,\infty}(\R^d)}\|g\|_{L^\eta(\R^d)}
\lesssim \|g\|_{L^\eta(\R^d)}.
$$

\emph{Step 2: The estimate $\|M_g\|_{\g(L^2(\R^d),\dot{H}^{-s,q}(\R^d))}\lesssim \|g\|_{L^\eta(\R^d)}$ implies \eqref{eq:conditions_parameters_homogeneous_gamma}.}
Before we can repeat the arguments in Step 3 of Proposition \ref{prop:M_g_Sobolev}, we first need to derive $s<d$ before we can use that $(-\Delta)^{-s/2}$ can be written in terms of Riesz kernels. However, using a dilation argument for $g$ together with the homogeneity of $\dot{H}^{-s,q}(\R^d)$ one can check that $
\frac{s}{d}+\frac{1}{q}= \frac{1}{\eta}+\frac{1}{2}$. Since $\eta\geq 2$, this implies $s<d$. Thus as before, we can use \eqref{eq:RieszA}.

Arguing as in Step 3 of Proposition \ref{prop:M_g_Sobolev}, the assumptions and \eqref{eq:identityAfg} imply that
\begin{align}
\label{eq: cond tilde g Riesz}
\Big\|x\mapsto \int_{\R^d} |\cR_s(x-y)|^2  \wt{g}(y)\,\di y\Big\|_{L^{q/2}(\R^d)} 
 \lesssim \|\wt{ g}\|_{L^{\eta/2}(\R^d)}, \ \ \ \wt{g}\in L^{\eta/2}(\R^d) \ \text{with} \ \wt{g}\geq 0. 
 \end{align}
 Since the Riesz kernel has the same behaviour as the Bessel potential kernel $\mathcal{G}_s$ on $|x|\leq 2$, the estimate \eqref{eq: upper bound integral delta} holds also in the present situation. Hence, $s>\frac{d}{2}$ follows. Similarly $\eta>2$ by the argument below \eqref{eq:L1est}.
\end{proof}

\begin{remark}[Sharpness $\&$ weak Lebesgue space]
\label{r:necessity_of_weak_spaces}
From the proof of Proposition \ref{prop:M_g_Sobolev}\eqref{it:M_g_Sobolev1a} and Proposition \ref{prop:M_g_Sobolev_hom}, it is clear that the presence of the weak Lebesgue space on the right-hand side of the $\gamma$-Young's inequality \eqref{eq:gamma_young} is essential in the sharp case, i.e.\ in the case
$
\frac{s}{d} + \frac{1}{q}  = \frac{1}{\eta} + \frac{1}{2}.
$
\end{remark}

\section{Main results -- Sharp Sobolev embeddings for Gaussian series}
\label{s:main_result}
Here, we state and prove the main result of this manuscript, which generalizes Theorem \ref{thm:intro} to the case of orthonormal systems with (arbitrary) growth. To state it, we introduce some notation. Let 
$$\Sf=(f_n)_{n\geq 1}\subseteq L^2(\Domm)$$ be an orthonormal system of $L^2(\Domm)$.
Let $(e_n)_{n\geq 1}$ be a complete orthonormal basis of $L^2(\Domm)$, and $(\mu_n)_{n\geq 1}\subseteq \ell^0$ be a sequence. Consider the linear operator 
\begin{equation}
\label{eq:Tmu_def} 
R_\mu = \sum_{n\geq  1} \mu_n \, e_n \otimes  f_n,
\end{equation}
where $(e_n\otimes  f_n)(h)=(h,e_n)_{L^2(\Domm)} f_n$ for $h\in L^2(\Domm)$. 
Next, we introduce a space of sequences that encodes the possible growth of the orthonormal system $\Sf$ in the $L^\infty$-norm. 
For  $\Sf=(f_n)_{n\geq 1}$ as above, and $\zeta\in [1,\infty)$, we let
\begin{equation}
\begin{aligned}
\label{eq:ellzeta_F}
\ell^\zeta(\Sf):= 
\Big\{ \mu=(\mu_n)_{n\geq 1}\in \ell^0\,:\, \sum_{n\geq 1} |\mu_n|^\zeta \|f_n\|_{L^\infty(\cO)}^{2}<\infty \Big\}, \  \quad
\ell^\infty(\Sf):=\ell^\infty,
\end{aligned}
\end{equation}
endowed with the natural norms. In other words, for $\zeta$ finite, $\ell^\zeta(\Sf)$ denotes the space of sequences with weight $(\|f_n\|_{L^\infty(\cO)}^{2} )_{n\geq 1}$, while $\ell^\infty(\Sf)$ denotes the usual space of bounded sequences. In particular, the weight is independent of $\zeta$, and this is a consequence of Proposition \ref{prop: interpolation} below.

As we will discuss in Section \ref{s:weight_necessary}, considering the above weighted sequences is not only natural, but essential to capture scaling properties of SPDEs as commented in Subsection \ref{ss:scaling_intro}.

\smallskip

The following is the main result of this manuscript. As above, $M_g$ denotes the multiplication operator with a measurable function $g:\Domm\to \R$.

\begin{theorem}[Sobolev embeddings for Gaussian series -- Weighted sequences]
\label{thm:MgTmu delta}
Let $\Domm$ be an open set in $\R^d$. 
Let $q\in (1, \infty)$, $\eta\in (1, q)$, $\zeta\in [2, \infty]$
and $s\in(0,d)$ be such that 
\begin{align}\label{eq:conditionsmain delta}
\frac{s}{d} + \frac{1}{q}  \geq \frac{1}{\eta} + \frac{1}{2} - \frac{1}{\zeta}\qquad \text{ and }\qquad 
\frac{1}{\eta} - \frac{1}{\zeta}<\frac{1}{2}.
\end{align}
Then for each $g\in L^{\eta}(\mathcal{O})$ and $\mu\in \ell^\zeta(\Sf)$, the operator 
$M_g R_{\mu}:L^2(\mathcal{O})\to \wt{H}^{-s,q}(\mathcal{O})$ is $\gamma$-radonifying
and 
\begin{align}\label{eq:maingammabound}
\|M_g R_{\mu}\|_{\gamma(L^2(\mathcal{O}), \wt{H}^{-s,q}(\mathcal{O}))}\lesssim_{d,s,q,\eta,\zeta} \|\mu\|_{\ell^{\zeta}(\Sf)} \|g\|_{L^{\eta}(\mathcal{O})}.
\end{align} 
\end{theorem}

A version of Theorem \ref{thm:MgTmu delta} also holds for homogeneous Sobolev spaces.

\begin{theorem}[Sobolev embeddings for Gaussian series -- Weighted sequences $\&$ Homogeneous spaces]
\label{thm:MgTmu delta_hom}
Suppose that $\Domm=\R^d$. 
Let $q\in (1, \infty)$, $\eta\in (1, q)$, $\zeta\in [2, \infty]$ and $s\in(0,d)$ be such that 
\begin{align}\label{eq:conditionsmain delta_hom}
\frac{s}{d} + \frac{1}{q}  = \frac{1}{\eta} + \frac{1}{2} - \frac{1}{\zeta}\qquad \text{ and }\qquad 
\frac{1}{\eta} - \frac{1}{\zeta}<\frac{1}{2}.
\end{align}
Then for each $g\in L^{\eta}(\R^d)$ and $\mu\in \ell^\zeta(\Sf)$, the operator 
$M_g R_{\mu}:L^2(\R^d)\to \dot{H}^{-s,q}(\R^d)$ is $\gamma$-radonifying and 
\begin{align}\label{eq:maingammabound_hom}
\|M_g R_{\mu}\|_{\gamma(L^2(\R^d),\dot{H}^{-s,q}(\R^d))}\lesssim_{d,s,q,\eta,\zeta}
 \|\mu\|_{\ell^{\zeta}(\Sf)} \|g\|_{L^{\eta}(\R^d)}.
\end{align} 
\end{theorem}

We now collect some observations regarding the above results. We will mainly comment on Theorem \ref{thm:MgTmu delta}, as similar comments extend almost verbatim to Theorem \ref{thm:MgTmu delta_hom}.  
From \eqref{eq:def_gamma_norm}, the estimate \eqref{eq:maingammabound} can be equivalently reformulated by using Gaussian random sums as follows. Let $(\g_n)_{n \geq 1}$ be independent standard Gaussian random variables $(\g_n)_{n\geq 1}$ on a probability space $(\O,\mathcal{A},\P)$. Under the assumptions of Theorem \ref{thm:MgTmu delta}, the sum $g\sum_{n\geq 1} \mu_n \g_n  f_n$ converges in $L^2(\O; H^{-s,q}(\Domm))$ and 
\begin{equation*}
\Big(\E\Big\|g \sum_{n\geq 1}  \mu_n \g_n f_n \Big\|_{H^{-s,q}(\Domm)}^2\Big)^{1/2}
\lesssim_{d,s,q,\eta,\zeta} \|\mu\|_{\ell^{\zeta}(\Sf)} \|g\|_{L^{\eta}(\mathcal{O})}.
\end{equation*}
In particular, Theorem \ref{thm:MgTmu delta} does not depend on the choice of the orthonormal basis $(e_n)_{n\geq 1}$.

As in Subsection \ref{ss:scaling_intro} (see also the comments below Propositions \ref{prop:M_g_Sobolev} and \ref{prop:M_g_Sobolev_hom} for a similar situation),
the first condition in \eqref{eq:conditionsmain delta} (with equality) and \eqref{eq:conditionsmain delta_hom}
are \emph{sharp} and they captures the scaling of the estimates \eqref{eq:maingammabound} and \eqref{eq:maingammabound_hom}.
Therefore, all the other cases considered in Theorem \ref{thm:MgTmu delta} will be referred to as \emph{non-sharp} cases, as there is a loss of smoothness and a break of the scaling. For completeness, we discuss further non-sharp cases in Remark \ref{rem:zetainfty2}.

\smallskip

Next, we discuss the necessity of the first condition in \eqref{eq:conditionsmain delta_hom}. Interestingly, the latter proof only needs the estimate \eqref{eq:maingammabound_hom} to hold for rank one operators, and its proof extends also to the inhomogeneous case \eqref{eq:maingammabound}, therefore providing the optimality of the first condition in \eqref{eq:conditionsmain delta}.

In the following, we use that, if $\S_f$ consists of a single element $f\in L^2(\cO)\cap L^\infty(\cO)$ with $\|f\|_{L^2(\cO)}=1$, then $\ell^\zeta(\S_f)$ is isomorphic to $\R$ and 
$$
\|\mu\|_{\ell^\zeta(\S_f)}= |\mu| \|f\|_{L^{\infty}(\cO)}^{2/\zeta} \ \ \text{ for all } \ \mu\in \ell^\zeta(\S_f).
$$

\begin{proposition}[Necessity of the first condition in \eqref{eq:conditionsmain delta} and \eqref{eq:conditionsmain delta_hom} -- rank one operators]
\label{p:necessityLinfty_hom}
Let $\Domm$ be an open set, and 
$q\in (1, \infty)$, $\eta\in (1, \infty)$, $\zeta\in [2, \infty]$, and $s\in(0,d)$.
Suppose that there is a constant $C>0$ such that for all $\mu> 0$, $f,g,e\in C^\infty_{{\rm c}}(\Domm)$ satisfying $\|e\|_{L^2(\Domm)}=\|f\|_{L^2(\Domm)}=1$, one has
\begin{equation}
\label{eq:estimate_required_necessity_II}
\|M_g R_\mu \|_{\gamma(L^2(\Domm),H^{-s,q}(\Domm))}\leq C \|\mu\|_{\ell^{\zeta}(\Sf)} \|g\|_{L^{\eta}(\Domm)},
\end{equation}
where $R_\mu= \mu\, e\otimes f$ and $\S_f =\{f\}$, then 
$
\frac{s}{d} + \frac{1}{q}  \geq \frac{1}{\eta} + \frac{1}{2} - \frac{1}{\zeta}.
$

Finally, if $\Domm=\R^d$, then \eqref{eq:estimate_required_necessity_II} with $H^{-s,q}(\R^d)$ replaced by $\dot{H}^{-s,q}(\R^d)$ and $g,R_\mu$ as above implies that
$
\frac{s}{d} + \frac{1}{q}  = \frac{1}{\eta} + \frac{1}{2} - \frac{1}{\zeta}.
$
\end{proposition}

As $\|M_g R_\mu\|_{\gamma(L^2(\Domm),H^{-s,q}(\Domm))}=\|g f\|_{H^{-s,q}(\cO)}$ by \eqref{eq:def_gamma_norm}, it follows from \eqref{eq:restriction_properties_of_tilde_spaces} that \eqref{eq:estimate_required_necessity_II} is weaker than assuming the estimate \eqref{eq:estimate_required_necessity_II} with $H^{-s,q}(\cO)$ replaced by $\wt{H}^{-s,q}(\cO)$.
In Proposition \ref{p:necessityweirdcond} we will show that even when restricted to the preiodic case the second condition in \eqref{eq:conditionsmain delta} is necessary.
The second condition in \eqref{eq:conditionsmain delta} does not seem to provide a limitation in applications to \emph{nonlinear} SPDEs \cite{AGVlocal}.

\smallskip

The method of Theorem \ref{thm:MgTmu delta} is applicable to other  situations. We will see two of these later on. A general framework can be found in Subsection \ref{ss:Bird} with an application to random fields. A further application to the case where we no longer need $f_n\in L^\infty(\cO)$ can be found in Section \ref{s:weight_necessary}.

\smallskip

This section is organized as follows. In Subsection \ref{ss:endpoint_case_main_result}, we prove two special cases of Theorem \ref{thm:MgTmu delta_hom} in which we consider the limiting $\zeta=2$ and $\zeta=\infty$, respectively. The latter cases will serve in Subsection \ref{ss:proof_generalized_growth} as endpoints for an interpolation argument leading to Theorems \ref{thm:MgTmu delta} and \ref{thm:MgTmu delta_hom}.

\subsection{Two endpoints of Theorem \ref{thm:MgTmu delta}}
\label{ss:endpoint_case_main_result}
We begin by proving Theorem \ref{thm:MgTmu delta} in the case $\zeta=2$.
Below, $\Domm$ is an open set in $\R^d$.

\begin{lemma}
\label{lem:l2 2 delta}
Suppose that
\begin{align}
\label{eq: bound l2 2 delta}
s\geq 0, \qquad q\in (1, \infty), \qquad   \frac{s}{d} + \frac{1}{q} \geq \frac{1}{\eta}, \qquad  \eta\in (1, q]. 
\end{align}
Then for each $g\in L^{\eta}(\mathcal{O})$ and $R\in \calL(L^2(\cO), L^\infty(\cO))$, the operator $M_g R:L^2(\mathcal{O})\to \wt{H}^{-s,q}(\mathcal{O})$ is $\gamma$-radonifying and 
\begin{align}\label{eq:MgRLinfty}
\|M_g R\|_{\gamma(L^2(\mathcal{O}), \wt{H}^{-s,q}(\mathcal{O}))}\lesssim_{d,s,q,\eta}
\|g\|_{L^\eta(\mathcal{O})} \|R\|_{\calL(L^2(\cO), L^\infty(\cO))}.
\end{align}
\end{lemma}

If $R$ is in the form \eqref{eq:Tmu_def} (i.e.\ $R=R_\mu$ for some $\mu\in \ell^0$), then 
\[\|R_{\mu}\|_{\calL(L^2(\cO), L^\infty(\cO))} = \Big\|\Big(\sum_{n\geq 1} |\mu_n|^2 |f_n|^2\Big)^{1/2} \Big\|_{L^\infty(\cO)}\leq \|\mu\|_{\ell^2(\Sf)},\]
where we recall that the latter norm was defined in \eqref{eq:ellzeta_F}. As a consequence of \eqref{eq:MgRLinfty}, we obtain 
\begin{equation}\label{eq:MgRLinfty2}
\|M_g R_{\mu}\|_{\gamma(L^2(\mathcal{O}),H^{-s,q}(\mathcal{O}))}\lesssim_{d,s,q,\eta}
\|g\|_{L^\eta(\mathcal{O})} \|\mu\|_{\ell^2(\Sf)}.
\end{equation}
The estimate \eqref{eq:MgRLinfty2} corresponds to the  endpoint case $\zeta=2$ of Theorem \ref{thm:MgTmu delta}.

\begin{proof}
As $- \frac{d}{\eta}\geq -s -\frac{d}{q}$, by extending by zero the Sobolev embedding on $\R^d$ gives $L^\eta(\Domm)\embed \wt{H}^{-s,q}(\Domm)$. This and the square function characterizations \eqref{eq: square fct char gamma} of $\g(L^2(\Domm),L^q(\Domm))$ imply 
\begin{align*}
  \|M_g R\|_{\gamma(L^2(\mathcal{O}),\wt{H}^{-s,q}(\mathcal{O}))} & \lesssim \|M_g R\|_{\gamma(L^2(\mathcal{O}),L^\eta(\mathcal{O}))} \\ 
  & \eqsim_q  \Big\||g| \Big(\sum_{n\geq 1} |R e_n |^2\Big)^{1/2}\Big\|_{L^\eta(\mathcal{O})}\\
  & \leq  \|g\|_{L^\eta(\Domm)} \Big\| \Big(\sum_{n\geq 1} |Re_n|^2\Big)^{1/2}\Big\|_{L^\infty(\mathcal{O})}\\
 & \leq  \|g\|_{L^\eta(\Domm)}  \|R\|_{\calL(L^2(\cO), L^\infty(\cO))},
\end{align*}
where the last estimate follows from the proof of \cite[Corollary 9.3.3]{HNVW2}.
\end{proof}

\smallskip

Now, we turn our attention to the case $\zeta=\infty$ of Theorem \ref{thm:MgTmu delta}. In this case, we need the results of the previous section, namely Proposition \ref{prop:M_g_Sobolev}.
Interestingly, in this situation, we do not need any structure from the operator $R$.

\begin{lemma}
\label{lem:zetainfty}
Suppose that
\begin{align}
\label{eq:boundHsqsharp}
s\in \Big(\frac{d}{2},d\Big), 
\qquad 2<\eta<q<\infty, \qquad   \frac{s}{d} + \frac{1}{q} \geq \frac{1}{\eta}+\frac{1}{2}. 
\end{align}
Then for each $g\in L^{\eta}(\mathcal{O})$ and $R\in \calL(L^2(\mathcal{O}))$, the operator $M_g R:L^2(\mathcal{O})\to H^{-s,q}(\mathcal{O})$ is $\gamma$-radonifying and 
\begin{equation}
\label{eq:gamma_L2_Leta_endpoint}
\|M_g R\|_{\gamma(L^2(\mathcal{O}), \wt{H}^{-s,q}(\mathcal{O}))}\lesssim_{d,s,q,\eta} \|g\|_{L^\eta(\mathcal{O})} 
\|R\|_{\calL(L^2(\mathcal{O}))}.
\end{equation}
\end{lemma}

Let us point out that, if $R$ is in the form \eqref{eq:Tmu_def} (i.e.\ $R=R_\mu$ for some $\mu\in \ell^0$), then 
$
\|R_\mu\|_{\calL(L^2(\mathcal{O}))}=\|\mu\|_{\ell^\infty},
$
and hence \eqref{eq:gamma_L2_Leta_endpoint} provides the bounds 
\begin{equation}
\label{eq:endpoint_infty_sob_embedding}
\|M_g R\|_{\gamma(L^2(\mathcal{O}), \wt{H}^{-s,q}(\mathcal{O}))}\lesssim_{d,s,q,\eta}\|g\|_{L^\eta(\mathcal{O})} \|\mu\|_{\ell^\infty(\S_f)},
\end{equation} 
where we used $\ell^\infty(\S_f) =\ell^\infty$ by the definition in \eqref{eq:ellzeta_F}.
The estimate \eqref{eq:endpoint_infty_sob_embedding} corresponds to the endpoint case $\zeta=\infty$ of Theorem \ref{thm:MgTmu delta}. Finally, note also that if $\mathcal{O}$ is bounded, taking  $R = \mathrm{Id}_{L^2(\Domm)}$ and $g\equiv 1$, the above result implies that the Sobolev embedding $L^2(\mathcal{O})\hookrightarrow \wt{H}^{-s,q}(\mathcal{O})$ is radonifying.

\begin{proof}[Proof of Lemma \ref{lem:zetainfty}]
Extending all functions as zero outside $\mathcal{O}$, to prove the above, it is sufficient to consider the case $\mathcal{O} = \R^d$. 
By the right-ideal property of $\gamma$-radonifying operators  \eqref{eq:gamma_ideal}, it suffices to prove $\|M_g\|_{\gamma(L^2(\R^d),H^{-s,q}(\R^d))}\lesssim_{d,s,q,\eta} \|g\|_{L^\eta(\R^d)}$. 
Under assumption \eqref{eq:boundHsqsharp}, the above is a consequence of Proposition \ref{prop:M_g_Sobolev}\eqref{it:M_g_Sobolev1a}.
\end{proof}  

Note that we do not need $f_n\in L^\infty(\cO)$ in Lemma \ref{lem:zetainfty} since it only uses  Proposition \ref{prop:M_g_Sobolev}. 
From the above proof and Proposition \ref{prop:M_g_Sobolev}\eqref{it:M_g_Sobolev1b}-\eqref{it:M_g_Sobolev1c} we also get variants of Lemma \ref{lem:zetainfty}. However, as commented below Proposition \ref{prop:M_g_Sobolev} and Theorem \ref{thm:MgTmu delta}, the latter are non-sharp results. For completeness, we include them in the following remark.

\begin{remark}[Non-sharp sufficient conditions for the validity of \eqref{eq:gamma_L2_Leta_endpoint}]
\label{rem:zetainfty2}
The conclusion of Lemma \ref{lem:zetainfty} also holds in the case 
\begin{itemize}
\item 
$
s>\frac{d}{2}$ and $\eta=q\geq 2$;
\item 
$s>\frac{d}{2}$, $\eta=2$, $q\in [2,\infty)$ and $\frac{s}{d} + \frac{1}{q} > 1$.
\end{itemize}
To see this, as in the proof of Lemma \ref{lem:zetainfty}, it is enough to consider $\Domm=\R^d$. Under the previous conditions, the estimate \eqref{eq:gamma_L2_Leta_endpoint} follows as above by noting that $M_g \in \g(L^2(\R^d),H^{-s,q}(\R^d))$ due to Proposition \ref{prop:M_g_Sobolev}\eqref{it:M_g_Sobolev1a}-\eqref{it:M_g_Sobolev1b}.
\end{remark}

\subsection{Proofs of Theorems \ref{thm:MgTmu delta} and \ref{thm:MgTmu delta_hom} and Proposition \ref{p:necessityLinfty_hom}}
\label{ss:proof_generalized_growth}
We begin by proving Theorem \ref{thm:MgTmu delta}; the proof of Theorem \ref{thm:MgTmu delta_hom} is similar. 
As explained in Subsection \ref{ss:proof_strategy}, the key idea in the proof of the former is to view the mapping
$$
(g,\mu)\mapsto M_g R_\mu
$$ 
as a bilinear operator on appropriate spaces, and interpolate the two endpoint cases analyzed in Lemmas \ref{lem:l2 2 delta} and \ref{lem:zetainfty}, using Proposition \ref{prop: interpolation}. 
To this end, we need the following proposition. While the result of this proposition is classical, we provide it as a way to conveniently show how we apply multilinear complex interpolation in the proof of Theorems \ref{thm:MgTmu delta} and \ref{thm:MgTmu delta_hom}. For this reason, we only sketch the proof.

To formulate it, we introduce the space
\begin{align}
    \ell^\zeta_w:=\Big\{(\mu_k)_{k\geq 1}\in \bR^\bN:\,\|\mu\|_{\ell^\zeta}^\zeta:=\sum_{k\geq 1}|\mu_k|^\zeta w(k)<\infty,\, \text{for}\,w = (w(k))_{k\geq 1}\in \bR_+^\bN \Big\}
\end{align}
if $\zeta<\infty$, and $\ell^\infty_w = \ell^\infty$. 

\begin{proposition}
    \label{prop: interpolation}
    Let parameters $\eta_0,\eta_1\in (1,\infty)$,  $\zeta_0, \zeta_1\in [2, \infty]$, $s_0,s_1\in \bR$, $p_0,p_1\in (1,\infty)$ and $q_0,q_1\in (1,\infty)$.
    Let $w_0 = (w_0(k))_{k\geq 1}\in \bR_+^\bN$ and $w_1 = (w_1(k))_{k\geq 1}\in\bR_+^\bN$. 
Suppose  that the bilinear operator
    \begin{align}
        R:=\Bigg\{\begin{split}
        L^{\eta_i}(\R^d)\times \ell^{\zeta_i}_{w_i}&\rightarrow \gamma(L^2(\R^d),H^{s_i,q_i}(\R^d)),\\
        (g,\mu)&\mapsto M_gR_\mu,
        \end{split}
        \quad i = 0,1
    \end{align}
    satisfies $\|R\|_{\cL(L^{\eta_i}\times\ell^{\zeta_i}_{w_i},\gamma (L^2,H^{s_i,q_i}))}\leq C_i$ for $i = 0,1$.
    If $\theta\in (0,1)$ and
    \begin{align}
        \frac{1}{p} = \frac{1-\theta}{p_0}+\frac{\theta}{p_1},\quad s = (1-\theta)s_0 + \theta s_1 
        \quad\text{and}\quad
        w(k) = (w_0(k))^{(1-\theta)\frac{\zeta}{\zeta_0}}(w_1(k))^{\theta\frac{\zeta}{\zeta_1}},\quad k\geq 1,
    \end{align}
    for $(p,p_0,p_1)\in \{(\eta,\eta_0, \eta_1),(\zeta,\zeta_0, \zeta_1),(q,q_0, q_1)\}$, then
    \begin{align}
        R:
        L^{\eta}(\R^d)\times \ell^{\zeta}_w\rightarrow \gamma(L^2(\R^d),H^{s,q}(\R^d)),
    \end{align}
    with norm $\|R\|_{\cL(L^{\eta}\times\ell^{\zeta}_w,\gamma (L^2,H^{s,q}))}\leq C_0^{1-\theta}C_1^\theta$.
\end{proposition}

\begin{proof}
If the underlying spaces are over the real scalars, then we apply multilinear complex interpolation to a complexification of the spaces. Moreover, in the Gaussian series, the complex and real Gaussians can be interchanged up to some universal constants (see \cite[Proposition 6.1.21]{HNVW2}). We first collect the individual interpolation spaces. While they are stated in much more generality in the references provided below, we state them here only for the special cases from the present article.
    By \cite[Theorem 2.2.6]{HNVW1} and \cite[Theorem 14.3.1]{HNVW3},
    we get for $\theta\in (0,1)$ and $d\geq 1$, respectively,
    \begin{align}
        [L^{\eta_0}(\R^d),L^{\eta_1}(\R^d)]_\theta = L^\eta(\R^d)
        \quad\text{and}\quad
        [\ell^{\zeta_0}_{w_0},\ell^{\zeta_1}_{w_1}]_\theta = \ell^\zeta_w.
    \end{align}
    Moreover, by \cite[Theorem 14.7.12]{HNVW3} we have
    \begin{align}
        [H^{s_0,q_0}(\R^d),H^{s_1,q_1}(\R^d)]_\theta = H^{s,q}(\R^d).
    \end{align}
    We remark that all the equalities stated above hold with equivalent norms. Finally, by \cite[Theorem 9.1.25]{HNVW2} we have
    \begin{align}
        [\gamma(L^2(\R^d),H^{s_0,q_0}(\R^d)),\gamma(L^2(\R^d),H^{s_1,q_1}(\R^d))]_\theta = \gamma(L^2(\R^d),H^{s,q}(\R^d)).
    \end{align}
    Then the required result follows by an application of multilinear complex interpolation (see \cite[Section 4.4]{BeLo}).
\end{proof}

Before we prove the main result of this subsection, we derive two necessary conditions if one interpolates \eqref{eq: bound l2 2 delta} and \eqref{eq:boundHsqsharp}. Let $(s_0, q_0, \eta_0, \zeta_0)$ and $(s_1, q_1, \eta_1, \zeta_1)$ denote the parameters from \eqref{eq: bound l2 2 delta} and \eqref{eq:boundHsqsharp}, respectively. In particular, $\zeta_0=2$ and $\zeta_1= \infty$. 

First we observe that in \eqref{eq: bound l2 2 delta} and \eqref{eq:boundHsqsharp} we have $\eta_0\leq q_0$ and $\eta_1< q_1$. Therefore, after interpolation, we also have $\eta<q$.
Secondly, interpolation gives that $\frac1\zeta = \frac{1-\theta}{2} + \frac{\theta}{\infty}$, and thus $\theta = 1-\frac{2}{\zeta}$. This implies
\begin{align*}
\frac1\eta = \frac{1-\theta}{\eta_0} + \frac{\theta}{\eta_1} <1-\theta + \frac{\theta}{2} = \frac12 + \frac{1}{\zeta}
\end{align*}
Therefore, a necessary condition for the interpolation to work is that $\frac1\eta-\frac{1}{\zeta}<\frac12$.

\begin{proof}[Proof of Theorem \ref{thm:MgTmu delta}]
Note that the cases $\zeta = 2$ and $\zeta=\infty$ are contained in Lemmas \ref{lem:l2 2 delta} and \ref{lem:zetainfty}, respectively. Note that we specialized to $R_{\mu}$ in  \eqref{eq:MgRLinfty2} and \eqref{eq:endpoint_infty_sob_embedding}. Thus from now on, we may assume $\zeta\in (2, \infty)$. 
As before, we can reduce to $\mathcal{O} = \R^d$. 
By definition, $\ell^2(\Sf)$ corresponds to the space $\ell^2_w$ of weighted $\ell^2$-sequences with weight $w=(\|f_n\|_{L^\infty(\Domm)}^2)_{n\geq 1}$ and norm 
$$
\|\mu\|_{\ell^2_w}=\|(\mu_n\|f_n\|_{L^\infty(\Domm)})_{n\geq 1}\|_{\ell^2}.
$$
Hence, to apply Proposition \ref{prop: interpolation}, we need to find  $(s_0, q_0, \eta_0, \zeta_0)$ and $(s_1, q_1, \eta_1, \zeta_1)$ such that \eqref{eq: bound l2 2 delta} and \eqref{eq:boundHsqsharp} hold, respectively, i.e.
\begin{align}
\label{eq: zeta 0 zeta 1 parameters}
    \zeta_0=2\quad\text{and}\quad
    \quad\zeta_1 = \infty,
\end{align}
and such that
\begin{align}
\label{eq:condtheta}
\theta &= 1-\frac{2}{\zeta},
\\
    \label{eq:cond1} \frac{1}{q_0} &= \frac{1}{1-\theta} \Big(\frac{1}{q} - \frac{\theta}{q_1}\Big),
\\
\label{eq:cond2} \frac{1}{\eta_0} &= \frac{1}{1-\theta} \Big(\frac{1}{\eta} - \frac{\theta}{\eta_1}\Big),
\\
\label{eq:cond4} s_0 &= \frac{s}{1-\theta}-\frac{\theta s_1}{1-\theta},
\end{align}
The conditions above stem from rewriting the definitions of the interpolation parameters in accordance with Proposition \ref{prop: interpolation} in a way that is convenient for this proof. In particular, we obtain \eqref{eq:condtheta} from $\frac{1}{\zeta} = \frac{1-\theta}{\zeta_0} + \frac{\theta}{\zeta_1}$ and \eqref{eq: zeta 0 zeta 1 parameters}.
Supposing that \eqref{eq: bound l2 2 delta} and \eqref{eq:boundHsqsharp}  hold with equality, we get additionally
\begin{align}
\label{eq:cond3} \frac{s_1}{d} &= - \frac{1}{q_1}  + \frac{1}{\eta_1}+\frac12,
\\
\label{eq:cond5} \frac{1}{\eta_0} &= \frac{s_0}{d} + \frac{1}{q_0}.
\end{align}
It is straightforward to compute that when expressing $s$ from \eqref{eq:cond4} and using the other identities, one obtains the first condition in \eqref{eq:conditionsmain delta}.
Now it suffices to check that there exist $(q_1, \eta_1)$ such that the conditions 
$s_0\geq 0$, $q_0\in (1, \infty), 1<\eta_0\leq q_0$, $s_1>0$ and $2<\eta_1<q_1<\infty$ are satisfied, which will give a range of admissible $(s,q,\eta,\zeta)$.
Due to \eqref{eq:cond1} and \eqref{eq:cond2}, the condition $\eta_0\leq q_0$ holds if and only if  
\begin{align}\label{eq:restrqeta}
  \frac{1}{\eta_1} \leq   \frac{1}{q_1}   + \frac{1}{\theta \eta} - \frac{1}{\theta q}.
\end{align} 
The condition $\eta_0>1$, together with $\tfrac{1}{\eta_0}>0$, applied to \eqref{eq:cond2} leads to  
\begin{align}\label{eq:restrqeta2}
\frac{1}{\theta \eta} - \frac{1}{\theta}+1<\frac{1}{\eta_1}<\frac{1}{\theta \eta}.
\end{align} 
Similarly, the restriction $\frac{1}{q_0}>0$ leads to 
\begin{align}\label{eq:restrqeta3}
\frac{1}{q_1}<\frac{1}{\theta q}.
\end{align} 
The condition $\tfrac{1}{q_0}<1$ yields $\tfrac{1}{\theta q}-\tfrac{1}{q_1}<\tfrac{1}{\theta}-1$, which can be verified by bounding $\tfrac{1}{\theta q}$ using \eqref{eq:restrqeta} and using \eqref{eq:restrqeta2}.

Moreover, we have the restriction $\frac{1}{q_1}<\frac{1}{\eta_1}<\frac12$. The condition $s_0\geq 0$ is satisfied due to $\eta_0\leq q_0$. The condition $s_1>\tfrac{d}{2}$  follows from \eqref{eq:cond3} due to $q_1>\eta_1$. 
 
For the existence of $2<\eta_1<q_1$ satisfying \eqref{eq:restrqeta}, \eqref{eq:restrqeta2} and \eqref{eq:restrqeta3} we therefore require
\begin{align}
\label{eq:cond max min Thm MgTmu delta}
\max\Big\{\frac{1}{q_1}, \frac{1}{\theta \eta} - \frac{1}{\theta}+1\Big\}<\min\Big\{ \frac{1}{q_1}   + \frac{1}{\theta \eta} - \frac{1}{\theta q}, \frac{1}{\theta \eta}, \frac12\Big\} \ \ \text{and} \ \ \frac{1}{q_1}<\frac{1}{\theta q}.
\end{align}
We check if this holds for a suitable choice of $q_1$. 
Clearly $\tfrac{1}{q_1}<\tfrac{1}{q_1}+\tfrac{1}{\theta\eta}-\tfrac{1}{\theta q}$, since $q>\eta$.
The condition $\frac{1}{\theta \eta} - \frac{1}{\theta}+1<\frac12$  is equivalent to the second part of \eqref{eq:conditionsmain delta}. Thus, an admissible  $\eta_1$ exists if 
\[\frac{1}{\theta q}- \frac{1}{\theta}+1<\frac{1}{q_1}<\min\Big\{\frac12,\frac{1}{\theta q}\Big\}.\] 
The latter is feasible if and only if 
\[\frac{1}{\theta q}- \frac{1}{\theta}+1 <\min\Big\{\frac12, \frac{1}{\theta q}\Big\}.\]
The first part of the minimum leads to 
$\frac{1}{q} - \frac1{\zeta}<\frac12$ which always holds due to the second part of \eqref{eq:conditionsmain delta} and $q>\eta$. The second part of the minimum leads to $-\frac{1}{\theta}+1<0$, which always holds. Hence, so long as the condition \eqref{eq:conditionsmain delta} holds, there exist $(s_0, q_0, \eta_0, \zeta_0)$ and $(s_1, q_1, \eta_1, \zeta_1)$ such that \eqref{eq: bound l2 2 delta} and \eqref{eq:boundHsqsharp} hold, respectively, and such that the conditions of Proposition \ref{prop: interpolation} are satisfied. The desired result is then an immediate consequence of Proposition \ref{prop: interpolation}.
\end{proof}

Next, we prove Theorem \ref{thm:MgTmu delta_hom}.

\begin{proof}[Proof of Theorem \ref{thm:MgTmu delta_hom} -- Sketch]
The proof is very similar to that of Theorem \ref{thm:MgTmu delta}. Therefore, we only provide a sketch of the proof.

\emph{Step 1: Homogeneous variant of Lemma \ref{lem:zetainfty} -- Suppose that 
$s\in (\frac{d}{2},d)$, $2<\eta<q<\infty$, and $\frac{s}{d} + \frac{1}{q} = \frac{1}{\eta}+\frac{1}{2}$. 
Then for each $g\in L^{\eta}(\mathcal{O})$ and $R\in \calL(L^2(\R^d))$, the operator $M_g R:L^2(\R^d)\to \dot{H}^{-s,q}(\R^d)$ is $\gamma$-radonifying and}
\begin{equation}
\label{eq:homogeneous_spaces_Mg_R_bounds_step1}
\|M_g R\|_{\gamma(L^2(\R^d),\dot{H}^{-s,q}(\R^d))}\lesssim_{d,s,q,\eta} \|g\|_{L^\eta(\R^d)} 
\|R\|_{\calL(L^2(\R^d))}.
\end{equation}

As in the proof of Lemma \ref{lem:zetainfty}, the above follows by the ideal property of $\g$-spaces and 
$
M_g\in \g(L^2(\R^d),\dot{H}^{-s,q}(\R^d)),
$
see  \eqref{eq:gamma_ideal} and Proposition \ref{prop:M_g_Sobolev_hom}, respectively.

\smallskip

\emph{Step 2: Conclusion}. The conclusion follows by interpolating the result of the homogeneous version of Lemma \ref{lem:l2 2 delta} in the sharp case $\frac{s}{d} + \frac{1}{q} = \frac{1}{\eta}$ and Step 1 via Proposition \ref{prop: interpolation} as in the proof of Theorem \ref{thm:MgTmu delta}.
\end{proof}

 It remains to prove Proposition  \ref{p:necessityLinfty_hom}.

\begin{proof}[Proof of Proposition \ref{p:necessityLinfty_hom}]
Before we turn to the proof we make a preliminary observation. 
Let \(K\subset \mathcal O\) be a compact set. Choose \(\chi\in C_{{\rm c}}^\infty(\mathcal O)\) such
that \(\chi=1\) on a neighbourhood of \(K\). Then, for every
\(\sigma\in\mathbb R\), \(q\in(1,\infty)\), and every
\(u\in H^{\sigma,q}(\mathcal O)\) with \(\operatorname{supp}u\subseteq K\),
one has
\begin{equation}\label{eq:ext0prop}
\|E_0u\|_{H^{\sigma,q}(\mathbb R^d)}
\lesssim_{\chi,\sigma,q}
\|u\|_{H^{\sigma,q}(\mathcal O)},
\end{equation}
where $E_0$ denotes the extension by zero. The proof uses the standard fact that  multiplication by
\(\chi\in C_{{\rm c}}^\infty(\mathbb R^d)\) is bounded on
\(H^{\sigma,q}(\mathbb R^d)\).

Next we turn to the proof of the proposition. We consider the case where $\Domm$ is an open set of $\R^d$.
Let $B:=B_r(x) \subseteq \Domm$ be an open ball. For simplicity, we assume $x=0$ and $r=1$; the general case is analogous.
Fix $h,g,e\in C^\infty_{{\rm c}}(B)$ and $\|e\|_{L^2(\Domm)}=1$. 
Applying \eqref{eq:estimate_required_necessity_II} to the operator  
$$
R= e\otimes h= \mu\, e\otimes (h/\|h\|_{L^2(\Domm)}) \ \ \text{ with }\ \ \mu= 
\|h\|_{L^2(\Domm)} \ \text{ and } \ f=h/\|h\|_{L^2(\Domm)},
$$
we get 
$$
\|g h\|_{H^{-s,q}(\Domm)}\leq C\|g\|_{L^\eta(\Domm)} \|h\|_{L^2(\Domm)}^{1-2/\zeta}\|h\|_{L^\infty(\Domm)}^{2/\zeta}.
$$
In the above, we used that 
$
\|M_g R\|_{\gamma(L^2(\Domm),H^{-s,q}(\Domm))} = \|g h\|_{H^{-s,q}(\Domm)}
$ 
and 
$
\|\mu\|_{\ell^\zeta(\Sf)}
= \|h\|_{L^2(\Domm)}^{1-2/\zeta}\|h\|_{L^\infty(\Domm)}^{2/\zeta},
$
which follows from \eqref{eq:def_gamma_norm} and \eqref{eq:ellzeta_F}, respectively.

Since $h,g\in C^\infty_{{\rm c}}(B)$, the above displayed inequality and \eqref{eq:ext0prop} imply
\begin{equation}
\label{eq:estimate_full_space_one_side}
\|g h\|_{H^{-s,q}(\R^d)}\leq C\|g\|_{L^\eta(\R^d)} \|h\|_{L^2(\R^d)}^{1-2/\zeta}\|h\|_{L^\infty(\R^d)}^{2/\zeta},
\end{equation}
where with $g,h$ we also denote the extension by zero outside $B$.

We now derive the claimed conditions on the parameters $(s,q,\eta,\zeta)$ by a scaling argument and the previous displayed inequality.
Recall the one-sided scaling property of the norm of $H^{-s,q}(\R^d)$, which follows by interpolation and the equivalence $H^{-s,q}(\R^d)=(H^{s,q'}(\R^d))^*$ where $q'=q/(q-1)$ is the conjugate exponent of $q$ (or alternatively, see \cite[3.4.1 and 2.5.6]{Tri83}). 
Let $n\geq 1$ be an integer, and note that $g_n = g(2^n \cdot)$ and $h_n = h(2^n\cdot)$ are still supported in the unit ball $B$. 
Hence, applying \eqref{eq:estimate_full_space_one_side} with $(g,h)$ replaced by $(g_n,h_n)$, we have 
\begin{align}
\label{eq:scaling_weighted_sequences}
2^{n(-s-d/q)} c_{g,h} & \lesssim \|g_n h_n\|_{H^{-s,q}(\R^d)} 
\\ 
\nonumber
&  \leq C \|g_n\|_{L^\eta(\R^d)} \|h_n\|_{L^2(\R^d)}^{1-2/\zeta}\|h_n\|_{L^\infty(\R^d)}^{2/\zeta}
\\ 
\nonumber
& \leq C 2^{-n d/\eta} 2^{-n(d/2-d/\zeta)} C_{g,h},
\end{align}
for constants $c_{g,h}, C_{g,h}$ depending only on $\|g\|_{L^\eta(\R^d)}$ and $\|h\|_{L^\infty(\R^d)}$.
From the arbitrariness of $n$, it follows that $-s-\frac{d}{q}\leq - \frac{d}{\eta} - \frac{d}{2}+\frac{d}{\zeta}$, which coincides with the required condition.

\smallskip

Finally, in case $\Domm=\R^d$ and \eqref{eq:estimate_required_necessity_II} holds with $H^{-s,q}(\R^d)$ replaced by $\dot{H}^{-s,q}(\R^d)$, then the argument in \eqref{eq:scaling_weighted_sequences} is also valid with $n\in \Z$. In the latter situation, the estimate \eqref{eq:scaling_weighted_sequences} implies $-s-\frac{d}{q}= - \frac{d}{\eta} - \frac{d}{2}+\frac{d}{\zeta}$, which is the claimed condition.
\end{proof}

\subsection{Bird's eye view on the interpolation argument and applications to random fields}\label{ss:Bird} 
From the proof of Theorem \ref{thm:MgTmu delta} and the fact that Lemmas \ref{lem:l2 2 delta} and \ref{lem:zetainfty} hold for general operators in $\calL(L^2(\cO),L^\infty(\cO))$ and $\calL(L^2(\cO))$, it is evident that the conclusion of Theorem \ref{thm:MgTmu delta} holds if we replace $R_\mu$ by an operator $R$ in the complex interpolation space
\begin{equation}
\label{eq:interpolation_HS}
[ \calL(L^2(\cO), L^\infty(\cO)),\calL(L^2(\cO))]_{1 - 2/\zeta}.  
\end{equation}
Note that the spaces $\calL(L^2(\cO), L^\infty(\cO))$ and $\calL(L^2(\cO))$ form an interpolation couple as they both embed into $\calL(L^2(\cO), L^2(\cO)+L^\infty(\cO))$ (recall that $\calL(X)=\calL(X,X)$). In particular, the above complex interpolation is well-defined (see e.g.\ \cite[Subsection 2.3]{BeLo} and \cite[Definition C.1.1]{HNVW1}).

The above observation and the proofs of Theorems \ref{thm:MgTmu delta} and \ref{thm:MgTmu delta_hom} yield the following:

\begin{proposition}[Abstract version of the Sobolev embedding for Gaussian series]\label{prop:abstract}
Let $\cO$ be an open set in $\R^d$. Let $q\in (1, \infty)$, $\eta\in (1, q)$, $\zeta\in [2, \infty]$ and $s\in(0,d)$ be such that 
\begin{align}\label{eq:conditionsmain delta2}
\frac{s}{d} + \frac{1}{q}  \geq \frac{1}{\eta} + \frac{1}{2} - \frac{1}{\zeta}\qquad \text{ and }\qquad 
\frac{1}{\eta} - \frac{1}{\zeta}<\frac{1}{2}.
\end{align}
Then for each $g\in L^{\eta}(\mathcal{O})$ and $R$ in the space \eqref{eq:interpolation_HS}, the operator 
$M_g R:L^2(\mathcal{O})\to \wt{H}^{-s,q}(\mathcal{O})$
is $\gamma$-radonifying and 
\begin{align}\label{eq:maingammabound2}
\|M_g R\|_{\gamma(L^2(\mathcal{O}), \wt{H}^{-s,q}(\mathcal{O}))}\lesssim_{d,s,q,\eta,\zeta} \|R\|_{[\calL(L^2(\cO), L^\infty(\cO)), \calL(L^2(\cO))]_{1 - 2/\zeta}} \|g\|_{L^{\eta}(\mathcal{O})}.
\end{align} 
If $\cO=\R^d$ and the first condition in \eqref{eq:conditionsmain delta2} is satisfied with equality, then \eqref{eq:maingammabound2} also holds with $H^{-s,q}(\R^d)$ replaced by the homogeneous space $\dot{H}^{-s,q}(\R^d)$.
\end{proposition}

Although the above result looks very promising, the drawback is that an explicit description of \eqref{eq:interpolation_HS} appears out of reach at the moment. However, note that, by the Riesz-Thorin interpolation, the space in \eqref{eq:interpolation_HS} embeds into $\calL(L^2(\cO),L^{\zeta}(\cO))$. Unfortunately, this is not so helpful to know. 

\smallskip

To avoid dealing with the large space \eqref{eq:interpolation_HS}, it is useful to specialize to concrete subspaces of operators, and then interpolate the latter. More precisely, one embeds (or \emph{represents}) the spaces $\mathcal{X}_0$ and $\mathcal{X}_1$ into $ \calL(L^2(\cO), L^\infty(\cO))$ and $\calL(L^2(\cO))$, respectively, via a linear assignment  $\vartheta\mapsto R_\vartheta$. Then, by interpolation, $[\mathcal{X}_0,\mathcal{X}_1]_{1-2/\zeta}$ embeds into \eqref{eq:interpolation_HS}, and Proposition \ref{prop:abstract} implies a corresponding estimate. This program has been successfully applied in the proofs of Theorems \ref{thm:MgTmu delta} and \ref{thm:MgTmu delta_hom}, where the representation is via the mapping $\mu\mapsto R_\mu$, $\mathcal{X}_0=\ell^2(\Sf)$, $\mathcal{X}_1=\ell^\infty(\Sf)$, where in \eqref{eq:Tmu_def}, $(e_n)_{n\geq 1}$ is an orthonormal basis and $\Sf=(f_n)_{n\geq 1}$ an orthonormal system such that $f_n\in L^\infty(\cO)$ for all $n\geq 1$. A similar method will be used in Section \ref{s:weight_necessary} for general orthonormal systems $(f_n)_{n\geq 1}$ (thus, without any boundedness assumption), where the endpoint provided by Lemma \ref{lem:l2 2 delta} will be changed accordingly to the setting. 

\smallskip

Below, we present another situation where the above construction can be performed in the case $\cO = \R^d$ to obtain regularity results for Gaussian series. In the latter case, it might be unnatural to work with a representation in terms of an orthonormal basis $(f_n)_{n\geq 1}$. Instead, Fourier multipliers form a more natural concept on $\R^d$. Applications will be given in Theorem \ref{thm:regulartity_matern_with_multiplicative_g} 
to  Mat\'ern random fields and  Theorem \ref{thm:heatMatern} to the stochastic heat equation on the full space. 
For details on Fourier multipliers, the reader is referred to \cite{Grafakos1, Horm60}. 
Also in this case, the operators $R$ in the subclass of \eqref{eq:interpolation_HS} can be identified as well. Now, the role of the representation parameter $\vartheta$ is played by the \emph{multiplier symbol}. 

To proceed further, we collect some notation. For a tempered distribution $m\in \S'(\R^d)$, let $T_m:\S(\R^d)\to \S'(\R^d)$ be defined by 
\begin{equation}
\label{eq:Fourier_multiplier}
T_m \phi= \F^{-1} (m \,\F\phi) =  (\F^{-1}{m})*\phi,
\end{equation}
where $\F$ and $\F^{-1}$ denote the Fourier transform and its inverse, respectively.
H\"older's inequality and Plancherel's identity ensure that, for all $\phi\in \S(\R^d)$,
\begin{align*}
\|T_m\phi\|_{L^\infty(\R^d)} &\leq \|\F^{-1}m\|_{L^2(\R^d)} \|\phi\|_{L^2(\R^d)} = \|m\|_{L^2(\R^d)} \|\phi\|_{L^2(\R^d)},
\\  \|T_m\phi\|_{L^2(\R^d)} &= \|m\, \F\phi\|_{L^2(\R^d)} \leq \|m\|_{L^\infty(\R^d)} \|\phi\|_{L^2(\R^d)}.
\end{align*}
In particular, $m\mapsto T_m$ gives a representation of $L^2(\R^d)$ and $L^\infty(\R^d)$ in $\calL(L^2(\R^d),L^\infty(\R^d))$ and $\calL(L^2(\R^d))$, respectively. 
Therefore, by interpolation, we obtain that for 
$$
m\in [L^2(\R^d),L^\infty(\R^d)]_{1 - 2/\zeta} = L^{\zeta}(\R^d),
$$ 
the operator
$T_m$ belongs in the space \eqref{eq:interpolation_HS} for $\cO = \R^d$ and satisfies 
\[\|T_m\|_{[\calL(L^2(\R^d), L^\infty(\R^d)), \calL(L^2(\R^d))]_{1 - 2/\zeta}}  \leq \|m\|_{L^{\zeta}(\R^d)}.  \]

Combining this with Proposition \ref{prop:abstract}, we obtain
 
\begin{theorem}[Random fields and Fourier multipliers]\label{thm:Randomfield}
Let $q\in (1, \infty)$, $\eta\in (1, q)$, $\zeta\in [2, \infty]$ and $s\in(0,d)$ be such that 
\begin{align}\label{eq:conditionsmain delta3}
\frac{s}{d} + \frac{1}{q}  \geq \frac{1}{\eta} + \frac{1}{2} - \frac{1}{\zeta}\qquad \text{ and }\qquad 
\frac{1}{\eta} - \frac{1}{\zeta}<\frac{1}{2}.
\end{align}
Then for each $g\in L^{\eta}(\R^d)$ and $m\in L^{\zeta}(\R^d)$, the operator 
$M_g T_m:L^2(\R^d)\to H^{-s,q}(\R^d)$ is $\gamma$-radonifying and 
\begin{align}\label{eq:maingammabound3}
\|M_g T_m\|_{\gamma(L^2(\R^d),H^{-s,q}(\R^d))}\lesssim_{d,s,q,\eta,\zeta} \|m\|_{L^{\zeta}(\R^d)} \|g\|_{L^{\eta}(\R^d)}.
\end{align} 
Moreover, if the first condition in \eqref{eq:conditionsmain delta3} is satisfied with equality, then \eqref{eq:maingammabound3} also holds with $H^{-s,q}(\R^d)$ replaced by the homogeneous space $\dot{H}^{-s,q}(\R^d)$.
\end{theorem}
Examples of multipliers $m$ leading to Mat\'ern fields will be considered in Theorems \ref{thm:regulartity_matern_with_multiplicative_g} and  \ref{thm:heatMatern}. It is worth noting that the first condition \eqref{eq:conditionsmain delta3} is necessary also in the Fourier multipliers case. Indeed, the estimate \eqref{eq:maingammabound3} implies 
\begin{equation}
\label{eq:byproduct_Fourier}
\|g T_m f\|_{H^{-s,q}(\R^d)}\lesssim \|m\|_{L^\zeta(\R^d)}\|g\|_{L^\eta(\R^d)}\|f\|_{L^2(\R^d)} \ \ \text{for all }\ f\in L^2(\R^d).
\end{equation}
By rescaling $(g,m,f)\mapsto (g_n,m_n,f_n):=(g(2^n\cdot),m(2^{-n}\cdot),f(2^n\cdot))$ in \eqref{eq:byproduct_Fourier}, and $T_{m_n} f_n= (T_m f)(2^n \cdot)$ as well as the one-sided scaling property of $H^{-s,q}(\R^d)$ (see the text preceding \eqref{eq:scaling_weighted_sequences}; 
and Proposition \ref{prop:M_g_Sobolev} for an analogous computation)
one can show the necessity of the first condition in \eqref{eq:conditionsmain delta3}, which is an equality in the case of homogeneous spaces.

\smallskip

It is clear from the proof of Theorem \ref{thm:Randomfield} that it extends to $\T^d$ as well. However, as we will see in the next subsection, this can also be obtained with the method of Theorem \ref{thm:MgTmu delta} (see Theorem \ref{thm:MgTmu deltaTd}) because of the special properties of the Fourier basis on $\T^d$.

\subsection{The periodic setting}
\label{ss:periodic_setting}
In this subsection, we comment on the extension of the above results in the case $\cO=\T^d$. 
First, by repeating the proofs of Section \ref{s:gamma_consequences} (see Remark \ref{rem:groups}) and Theorem \ref{thm:MgTmu delta}, it is clear that the latter also holds in the periodic setting. To formulate the latter, we use the periodic variant of $H^{-s,q}(\T^d)$ (see \cite[Chapter 9]{Tri83}). Thus, we obtain 

\begin{theorem}[Sobolev embeddings for Gaussian series -- Periodic weighted case]
\label{thm:MgTmu deltaTd}
The statement of Theorem \ref{thm:MgTmu delta} holds with $\mathcal{O}$ replaced by $\T^d$. 
\end{theorem}

As on the periodic setting $\T^d$, Fourier multipliers are diagonal operators with respect to the standard Fourier basis, i.e.\ $$
f_k(x)=e_k(x):=e^{2\pi i k\cdot x} \ \ \text{ for } \ k\in \Z^d
$$ 
the above result also yields the periodic version of Theorem \ref{thm:Randomfield}. 

Because of its fundamental role in our analysis, we provide an alternative proof of the necessity of the first condition in \eqref{eq:conditionsmain delta} on $\T^d$ in the case of diagonal operators, therefore yielding optimality also in the case of Fourier multipliers on $\T^d$. Moreover, we provide a proof of the necessity of the second condition in \eqref{eq:conditionsmain delta} in Proposition \ref{p:necessityweirdcond}. The case of diagonal operators is interesting as a scaling argument, as in the proof of Proposition  \ref{p:necessityLinfty_hom} does not seem possible due to the different rescaling of the Fourier symbol and functions, see below \eqref{eq:byproduct_Fourier}. 
Recall that $|e_k|\equiv 1$ pointwise, and thus $\ell^\zeta(\S_f)$ coincides with the \emph{unweighted} sequence space $\ell^\zeta$ if $\Sf$ is the Fourier basis on $\T^d$.

\begin{proposition}[Necessity of the first condition in \eqref{eq:conditionsmain delta} -- Diagonal operator on $\T^d$]
\label{p:necessityLinfty}
Let $
q\in (1, \infty)$, $ \eta\in (1, \infty)$, $\zeta\in [2, \infty]$, and   $s\in(0,d)$. 
For each $\mu\in \ell^\infty$, let $R_{\mu}\in \calL(L^2(\T^d),L^2(\T^d))$ be defined by $R_{\mu} e_k = \mu_k e_k$, for   $k\in \Z^d$. Suppose that there is a constant $C>0$ such that for all $g\in C^\infty(\T^d)$ and all finitely nonzero sequences $\mu$, one has
\begin{align*}
\|M_g R_\mu \|_{\gamma(L^2(\T^d),H^{-s,q}(\T^d))}\leq C \|\mu\|_{\ell^{\zeta}} \|g\|_{L^{\eta}(\T^d)},
\end{align*}
then $\frac{s}{d} + \frac{1}{q}  \geq \frac{1}{\eta} + \frac{1}{2} - \frac{1}{\zeta}$.
\end{proposition}

\begin{proof}
Let $N \geq 1$ be an integer and define a frequency block in $\Z^d$ by
$$
C_N = \{ n \in \Z^d : 2^N \leq n_i \leq 3 \cdot 2^{N-1} \text{ for } i=1, \dots, d \}.
$$
Clearly we have $|C_N|\eqsim 2^{dN}$, for $N\geq 1$.

Let $g = \sum_{n \in C_N} e_n$ and set $\mu_n = \one_{C_N}(n)$ for $n\in \bZ^d$, so that clearly $\mu = (\mu_n)_{n\in \bZ^d}$ only has finitely many non-zero entries. 
Then
$\|\mu\|_{\ell^\zeta} =  |C_N|^{\frac{1}{\zeta}} \eqsim 2^{\frac{dN}{\zeta}}$.
Regarding the function $g$, we first rewrite it as
\begin{align}
g(x) 
&= \sum_{n \in C_N} \prod_{j=1}^d e^{2\pi i n_j x_j} = \prod_{j=1}^d \Big( \sum_{n_j=2^N}^{3 \cdot 2^{N-1}} e^{2\pi i n_j x_j} \Big) =:\prod_{j=1}^d p_j(x_j).
\end{align}
Each factor $p_j$ is a shifted one-dimensional Dirichlet kernel with $3 \cdot 2^{N-1} - 2^N+1 = 2^{N-1}+1$ terms. 
Multiplying $p_j$ by a suitable unimodular number
one can easily check that 
\begin{align}
\|p_j\|_{L^\eta(\T)}= 
\|D_{2^{N-2}}\|_{L^\eta(\bT)},
\end{align}
where $D_{2^{N-2}}$ is the Dirichlet kernel $D_{2^{N-2}}(x_j) = \sum_{m =-2^{N-2}}^{2^{N-2}} e^{2\pi i m x_j}$. It is well-known that 
\begin{align}
\|D_{2^{N-2}}\|_{L^\eta(\T)}\eqsim 2^{N(1-\frac{1}{\eta})},
\end{align}
see e.g.\ \cite[Exercise 3.1.6]{Grafakos1}.
Therefore, we have
\begin{align}
\label{eq: g C_N norm}
\|g\|_{L^{\eta}(\T^d)} = \Big\| \prod_{j=1}^d p_j\Big\|_{L^\eta(\T^d)} = \prod_{j=1}^d \|D_{2^{N-2}}\|_{L^\eta(\T)} \eqsim 2^{dN(1-\frac{1}{\eta})}
\end{align}
Combining the estimates for $\mu$ and $g$ we get
\begin{align}
\label{eq: combined estimate g mu C_N}
\|g\|_{L^{\eta}(\T^d)} \|\mu\|_{\ell^\zeta} \eqsim 2^{dN(1-\frac{1}{\eta})}  2^{\frac{dN}{\zeta}} = 2^{dN(1 - \frac{1}{\eta} + \frac{1}{\zeta})}.
\end{align}
Next, we calculate a lower bound for the $\gamma(L^2,H^{-s,q})$-norm. By definition of $R_\mu$, we get
\begin{align}
\|M_g R_\mu\|_{\gamma(L^2(\T^d), H^{-s,q}(\T^d))}^2 &= \E \Big\| \sum_{n \in C_N} \gamma_n g e_n \Big\|_{H^{-s,q}(\T^d)}^2 
\label{eq: MgTmu C_N calculation 1}
= \E \| S_N \|_{H^{-s,q}(\T^d)}^2,
\end{align}
where $S_N = \sum_{n \in C_N} \gamma_n g e_n$.
Using the definition of $g$, we observe that
\begin{align}
S_N = \sum_{k \in C_N} \gamma_k \Big(\sum_{n \in C_N} e_n \Big) e_k = \sum_{k \in C_N} \sum_{n \in C_N} \gamma_k e_{n+k}
\end{align}
The components $m_j = n_j + k_j$ are in the range 
\begin{align}
m_j\in [2^N+2^N, 3 \cdot 2^{N-1} + 3 \cdot 2^{N-1}] = [2^{N+1}, 3 \cdot 2^N],\quad j=1,\dots,d.
\end{align}
Therefore, by the Bernstein-Nikolskii type estimate (see \cite[Section 2.1.1]{BCD}) for trigonometric polynomials with frequency support in an annulus $\{m\in \Z^d: a 2^{N} \leq |m|\leq b 2^{N}\}$, we obtain
\begin{align}
    \|S_N\|_{H^{-s,q}(\T^d)} &
\label{eq: MgTmu C_N calculation 2}
    \eqsim 2^{-Ns} \|S_N\|_{L^q(\T^d)}.
\end{align}
Further, by the square function characterisation, we have
\begin{align}
\label{eq: MgTmu C_N calculation 3}
    (\E\|S_N\|_{L^q(\T^d)}^2)^{1/2}&\eqsim \Big\|\Big(\sum_{n\in C_N}|ge_n|^2\Big)^{\frac{1}{2}} \Big\|_{L^q(\T^d)} \eqsim \|g\|_{L^q(\T^d)}|C_N|^{\frac{1}{2}}  = 2^{dN( 1-\frac{1}{q})}2^{\frac{dN}{2}},
\end{align}
where we used \eqref{eq: g C_N norm} for $q$ in place of $\eta$. Combining \eqref{eq: combined estimate g mu C_N}, \eqref{eq: MgTmu C_N calculation 1}, \eqref{eq: MgTmu C_N calculation 2} and \eqref{eq: MgTmu C_N calculation 3} we then get
\begin{align}
   2^{-sN + dN( 1-\frac{1}{q} + \frac{1}{2})}\lesssim  2^{dN(1-\frac{1}{\eta}+\frac{1}{\zeta} )}.
\end{align}
Since, by assumption, this must hold for $N$ arbitrarily large, the claimed estimate follows.
\end{proof}

\begin{proposition}[Necessity of the second condition in \eqref{eq:conditionsmain delta} -- Diagonal operator on $\T^d$]\label{p:necessityweirdcond}
Let $s>0$, $q\in (1, \infty)$, $\eta\in (1, \infty)$ and $\zeta\in(2,\infty)$ be such that 
\[
        \frac{s}{d} + \frac{1}{q}  = \frac{1}{\eta} + \frac{1}{2} - \frac{1}{\zeta}
        %\qquad
        %s=d\Bigl(1-\frac1q\Bigr).
\]
For each $\mu\in \ell^\infty$, let $R_{\mu}\in \calL(L^2(\T^d),L^2(\T^d))$ be defined by $R_{\mu} e_k = \mu_k e_k$, for   $k\in \Z^d$.
Suppose that there is a constant $C$ such hat for all $g\in C^\infty(\T^d)$ and all finitely nonzero sequences $\mu$, one has
\begin{align}
\label{eq:assumption_optimality_conditions_eta_zeta12}
\|M_g R_\mu \|_{\gamma(L^2(\T^d),H^{-s,q}(\T^d))}\leq C \|\mu\|_{\ell^{\zeta}} \|g\|_{L^{\eta}(\T^d)},
\end{align}
Then $\frac1\eta-\frac1\zeta<\frac12$. 
\end{proposition}

\begin{proof}
After we showed $\frac1\eta-\frac1\zeta\leq \frac12$, we will additionally show that strict inequality is necessary. Let $(\gamma_k)_{k\in \Z}$ be independent standard Gaussian variables.

On $H^{-s,q}(\mathbb T^d)$ we use the Bessel-potential norm
\[
        \|u\|_{H^{-s,q}(\mathbb T^d)}
        =
        \|J^{-s}u\|_{L^q(\mathbb T^d)}
        \quad \text{ where }\quad
        \widehat{J^{-s}u}(m)
        =
        (1+4\pi^2|m|^2)^{-s/2}\widehat u(m).
\]
Let $N\geq 2$ and $\Lambda_N:=\{-N,\ldots,N\}^d$. Then $|\Lambda_N| \eqsim N^{d}$. Let 
\[
        g:=\sum_{n\in \Lambda_N}e_{-n} \quad \text{and} \quad \mu_k:=\mathbf 1_{\Lambda_N}(k). 
\]
Since $g$ consists of a product of one-dimensional Dirichlet kernels of length $N$, we obtain that $\|g\|_{L^\eta}\simeq N^{d(1-1/\eta)}$ (compare with \eqref{eq: g C_N norm}) and hence
\[\|\mu\|_{\ell^\zeta}\|g\|_{L^\eta} \lesssim N^{d(1-1/\eta + 1/\zeta)}.\]
A straightforward computation shows that
$\widehat{ge_k}(m) = \1_{k-m \in\Lambda_N}$
and thus
\begin{align}
\label{eq: J^s computation}
    J^{-s}(ge_k)(m) &= \sum_{m\in\bZ^d}(1+4\pi^2|m|^2)^{-s/2}\1_{k-m \in\Lambda_N}e_m
    = \sum_{n\in\Lambda_N}\widehat{K}_s(k-n) e_{k-n}.
\end{align}
where $K_s$ is the Bessel kernel associated with $J^{-s}$, i.e. $\widehat{K}_s(m)=(1+4\pi^2|m|^2)^{-s/2}$.

On the other hand, since $\int_{\T^d} e_j \, \di x=0$ for $j\neq 0$, we obtain from the above formula that
\begin{align*}
        \|M_{g}R_{\mu}\|_{\gamma(L^2,H^{-s,q})}^2
        &=
        \mathbb E
        \Bigl\|
            \sum_{k\in \Lambda_N}\gamma_k J^{-s}(ge_k)
        \Bigr\|_{L^q}^2
\\ & \geq    \mathbb E
        \Bigl|\int_{\mathbb{T}^d}
            \sum_{k\in \Lambda_N}\gamma_k J^{-s}(ge_k)\, \di x\Bigr|^2
= \mathbb E
        \Bigl|
            \sum_{k\in \Lambda_N}\gamma_k 
        \Bigr |^2 \eqsim N^{d}.
\end{align*}
Note that the assumed estimate \eqref{eq:assumption_optimality_conditions_eta_zeta12} implies $N^{d/2}\leq C N^{d(1-1/\eta + 1/\zeta)}$. Since $N$ was arbitrary, this implies $\frac1\eta-\frac1\zeta\leq \frac12$. Note that the previous argument does not require the condition $\frac{s}{d} + \frac{1}{q}  = \frac{1}{\eta} + \frac{1}{2} - \frac{1}{\zeta}$.

Next suppose that $\frac1\eta-\frac1\zeta = \frac12$. We derive a contradiction.
Let 
\[
        A_N:=\{-2N,-2N+1,\ldots,2N\}^d,
        \qquad
        B_N:=\{-\lfloor N/2\rfloor,\ldots,\lfloor N/2\rfloor\}^d .
\]
Then $|A_N|\simeq_d N^d$ and $|B_N|\simeq_d N^d$. Set
\[
        g(x):=\sum_{n\in A_N}e_{-n}(x), \ \ \text{and} \  \ \mu_k:=\mathbf 1_{B_N}(k)
\]
Then $\|\mu\|_{\ell^\zeta}=|B_N|^{1/\zeta}\simeq_d N^{d/\zeta}$.
Moreover, since $g$ is a product of one-dimensional Dirichlet kernels of length
comparable to $N$, we find again that
\[
        \|\mu\|_{\ell^\zeta}\|g\|_{L^\eta}
        \lesssim
        N^{d/\zeta+d(1-1/\eta)} = N^{d/2}.
\]
For the left-hand side by definition of the
$\gamma$-norm,
\[
\begin{aligned}
        \|M_{g}R_{\mu}\|_{\gamma(L^2,H^{-s,q})}
        &=
        \Bigl(
        \mathbb E
        \Bigl\|
            \sum_{k\in B_N}\gamma_k ge_k
        \Bigr\|_{H^{-s,q}}^2
        \Bigr)^{1/2}                                                        
        =
        \Bigl(
        \mathbb E
        \Bigl\|
            \sum_{k\in B_N}\gamma_k J^{-s}(ge_k)
        \Bigr\|_{L^q}^2
        \Bigr)^{1/2}.
\end{aligned}
\]
Let $F_N$ be the one-dimensional Fej\'er kernel (see \cite[Sec. 3.1.3]{Grafakos1}),
\[F_N(t)=\sum_{|m|\leq N}\Big(1-\frac{|m|}{N+1}\Big)e^{2\pi i mt} = \frac{1}{N+1} \frac{\sin^2(\pi(N+1)x)}{\sin^2(\pi x)},\] 
and put
$V_N(x)=\prod_{j=1}^dF_N(x_j)$. Then $V_N\geq0$ and
$\int_{\mathbb T^d}V_N\,\di x=1$. Since $V_N$ is a positive
kernel of integral one, Minkowski's inequality applied pointwise in $\Omega$ implies that
\[
\begin{aligned}
        \|M_{g}R_{\mu}\|_{\gamma(L^2,H^{-s,q})}
        &\geq
        \Bigl(
        \mathbb E
        \Bigl\|
            V_N*\sum_{k\in B_N}\gamma_k J^{-s}(ge_k)
        \Bigr\|_{L^q}^2
        \Bigr)^{1/2}.
\end{aligned}
\]
Let $k\in B_N$. Recalling \eqref{eq: J^s computation} one has
$J^{-s}(ge_k)=\sum_{n\in A_N} \widehat K_s(k-n) e_{k-n}$.
If $\widehat{V}_N(m)\neq0$, then $|m_j|<N+1$ for all $j$. Since $|k_j|\leq N/2$, it follows that $k-m\in A_N$. Hence, 
\[
        V_N*J^{-s}(ge_k)
        =
        \sum_{m\in\mathbb Z^d}\widehat{V}_N(m)\widehat{K}_s(m)  e_m=:P_N,
\]
which is independent of $k$. 
It follows that
\[
        V_N*\sum_{k\in B_N}\gamma_k J^{-s} (ge_k)
        =
        \Bigl(\sum_{k\in B_N}\gamma_k\Bigr)P_N .
\]
Consequently,
\[
\begin{aligned}
        \|M_{g}R_{\mu}\|_{\gamma(L^2,H^{-s,q})}
        &\geq
        \Bigl(\mathbb E\Bigl|\sum_{k\in B_N}\gamma_k\Bigr|^2\Bigr)^{1/2}
        \|P_N\|_{L^q}                                                       
        =
        |B_N|^{1/2}\|P_N\|_{L^q}
        \gtrsim_d
        N^{d/2}\|P_N\|_{L^q}.
\end{aligned}
\]
It remains to estimate $P_N=V_N*K_s$ from below. %Since $0<s<d$, the kernel $K_s$ is locally integrable. Moreover, it is
%positive. 
Note that
\[
        \widehat K_s(m) = (1+4\pi^2|m|^2)^{-s/2}
        =
        \frac1{\Gamma(s/2)}
        \int_0^\infty t^{s/2-1}e^{-t}e^{-4\pi^2t|m|^2}\,\di t,
\]
and therefore
\[
        K_s(x)
        =
        \frac1{\Gamma(s/2)}
        \int_0^\infty t^{s/2-1}e^{-t}p_t(x)\,\di t,
\]
where $p_t(x) = \sum_{\ell\in\mathbb Z^d}
        (4\pi t)^{-d/2}
        \exp\left(-\frac{|x+\ell|^2}{4t}\right)$ is the periodic heat kernel. 
        Since $p_t\geq0$, it follows that
$K_s\geq0$. Considering $\ell=0$, one can check that there are
constants $c>0$ and $r_0>0$ such that
\[
        K_s(x)\geq c|x|^{s-d},
        \qquad 0<|x|\leq r_0.
\]
From the definition of the Fej\'er kernel one can check that there are
constants $a,c_0>0$, depending only on $d$, such that
\[
        \int_{|y|\leq a/N}V_N(y)\,\di y\geq c_0
\]
for all $N\geq2$. If $2a/N\leq |x|\leq r_0/2$ and $|y|\leq a/N$, then
$\frac12|x|\leq |x-y|\leq \frac32|x|$. Hence, by positivity of $V_N$ and
the lower bound for $K_s$,
\[
\begin{aligned}
        P_N(x)
%        &=
%        \int_{\mathbb T^d}V_N(y)K_s(x-y)\,dy                                \\
        &\geq
        \int_{|y|\leq a/N}V_N(y)K_s(x-y)\,\di y                                  
        \geq c|x|^{s-d}\int_{|y|\leq a/N}V_N(y)\,\di y
        \geq 
        cc_0 |x|^{s-d}.
\end{aligned}
\]
Therefore, since the equality in the Proposition statement implies $s=d(1-1/q)$,
\[
\begin{aligned}
        \|P_N\|_{L^q(\mathbb T^d)}^q
        &\geq
        \int_{\{C_0/N\leq |x|\leq r_0/2\}}|P_N(x)|^q\,\di x 
        \geq (cc_0)^q \int_{\{C_0/N\leq |x|\leq r_0/2\}}|x|^{(s-d)q}\, \di x\gtrsim
        \log N.
\end{aligned}
\]
We can conclude that
\[
        \|M_{g}R_{\mu}\|_{\gamma(L^2,H^{-s,q})}
        \gtrsim
        N^{d/2}(\log N)^{1/q}.
\]
Combining this with the upper estimate for
$\|\mu\|_{\ell^\zeta}\|g\|_{L^\eta}$, we obtain the desired contradiction by letting $N\to \infty$. 
\end{proof}

The following result is an interesting case that one can consider on sets with finite volume. It shows a self-improvement hinting towards the non-optimality of such results if the function $g$ is not present. A similar result for $\R^d$ cannot hold since $g=1$ is not in $L^{\eta}(\R^d)$. 
\begin{proposition}\label{prop:zetainftytorus}
Let $q\in (1, \infty)$ be given. The following are equivalent 
\begin{enumerate}[\rm (1)]
\item\label{it:zetainftytorus_1} The identity satisfies $\mathrm{Id}\in \gamma(L^2(\T^d), H^{-s,q}(\T^d))$.
\item\label{it:zetainftytorus_2} $s>d/2$. 
\end{enumerate}
\end{proposition}
\begin{proof}
\eqref{it:zetainftytorus_2} $\Rightarrow$ \eqref{it:zetainftytorus_1}: For $q>2$ this follows from Lemma \ref{lem:zetainfty} by taking $\eta\in (2, q)$, $g=1$, $\mu_k = 1$ for all $k\in \Z^d$ (see also Remark \ref{rem:groups}). Note that $\wt{q}\in (1, 2]$ is simply a consequence of $H^{-s,q}(\T^d) \hookrightarrow H^{-s,\wt{q}}(\T^d)$ and the ideal property. 

\eqref{it:zetainftytorus_1} $\Rightarrow$ \eqref{it:zetainftytorus_2}: 
By the same embedding as before, it suffices to consider $q\in (1, 2]$. By the cotype $2$ property of $H^{-s,q}(\T^d)$ (see \cite[Corollary 7.1.5]{HNVW2}) we obtain
\[\|\mathrm{Id}\|_{\gamma(L^2(\T^d), H^{-s,q}(\T^d))}\geq c \Big(\sum_{k\in \Z^d}\|e_k\|_{H^{-s,q}(\T^d)}^2\Big)^{1/2} =  c \Big(\sum_{k\in \Z^d} (1+|k|^2)^{-s}\Big)^{1/2},\]
and thus the convergence of the latter series implies $s>d/2$.
\end{proof}

As a consequence, Theorem \ref{thm:MgTmu deltaTd} does not extend to $\zeta=\infty$ for $q\in [1, \infty), \eta\in [q,\infty]$ in the limiting case where $\frac{s}{d} +\frac1q = \frac{1}{\eta}+\frac12$ in general. Indeed, we find that $\frac{s}{d} = \frac{1}{\eta} - \frac1q +\frac12\leq \frac{1}{2}$. However, if we let $g = 1$, and $\mu_k = 1$ for all $k\geq 1$, then $I = M_g R_{\mu} \in \gamma(L^2(\T^d), H^{-s,q}(\T^d))$ is equivalent to $s>d/2$ according to Proposition \ref{prop:zetainftytorus}. This contradicts $\frac{s}{d}\leq \frac12$.

We end this subsection on the periodic setting with a remark about a corrected bound from the literature. Let $S^\zeta(L^2(\R^d))$ denote the Schatten class (see \cite[Appendix D]{HNVW1} for its definition). 

\begin{remark}
\label{r:schatter_and_heat_eq}
The following correction of \cite[(3.14)]{C03} holds for any $\zeta\in [2, \infty]$:
\begin{align}\label{eq:sharpestsemigroup}
\|S(t) M_g T\|_{S^2(L^2(\T^d))}\leq C t^{-d/4} \|g\|_{L^2(\T^d)} \|T\|_{S^\zeta(L^2(\T^d))}, \ \ t\in (0,1).
\end{align}
For simplicity, we only consider the torus, but the same holds on $\R^d$ as follows from \eqref{eq:Simon} and the sharpness argument below. To prove the bound \eqref{eq:sharpestsemigroup}, it suffices to consider $\zeta=\infty$. By the ideal property of $S^2(L^2(\T^d))$, it suffices to prove the bound $\|S(t) M_g\|_{S^2(L^2(\T^d))}\leq C t^{-d/4}\|g\|_{L^2(\T^d)}$. Let $(e_k)_{k\in \Z^d}$ be the trigonometric system. Then $S(t) e_{k} = e^{-\lambda_{k} t}e_k$ with $\lambda_k\eqsim k^2$. Thus, we can write 
\begin{align*}
\|S(t) M_g\|_{S^2(L^2(\T^d))}^2 &= \sum_{k,\ell\geq 1}  |(S(t) M_g e_k, e_{\ell})|^2 
 \\ & = \sum_{\ell\geq 1} e^{-2\lambda_{\ell} t} \|M_g e_{\ell}\|_{L^2(\T^d)}^2 
 \leq \|g\|_{L^2(\T^d)}^2 \sum_{\ell\geq 1} e^{-2\lambda_{\ell} t} 
 \leq C t^{-d/2} \|g\|_{L^2(\T^d)}^2.
\end{align*}
Next, we prove that the power $d/4$ in the factor $t^{-d/4}$ in \eqref{eq:sharpestsemigroup} is sharp. To this end, it suffices to consider $\zeta=2$.  Let $T = e_0\otimes f$ with $f = \sqrt{k_t}$, where $k_t$ is the heat kernel. Then $\|f\|_{L^2(\T^d)} = 1$. Letting $g = f$, we obtain
\[\|S(t) M_g T\|_{S^2(L^2(\T^d))} = \|k_t* k_t\|_{L^2(\T^d)} = \|k_{2t}\|_{L^2(\T^d)}\eqsim t^{-d/4}, \ \ t\in (0,1).\]
On the other hand $\|g\|_{L^2(\T^d)} \|T\|_{S^2(L^2(\T^d))} = \|g\|_{L^2(\T^d)}\|f\|_{L^2(\T^d)} =1$. 
\end{remark}

\section{Sobolev embeddings for unweighted Gaussian sequences}
\label{s:weight_necessary}

In this section, we continue our investigation on Sobolev embeddings for Gaussian sequences, where $\mu$ is now taken from the \emph{unweighted} sequence space $\ell^\zeta$. For instance, this covers the case where $f_n\not\in L^\infty(\cO)$.
As commented in Remark \ref{r:schatter_and_heat_eq}, spaces of this type have already been used in the literature. However, it will turn out that it is \emph{not}  possible to prove a result under the sharp condition in Theorem \ref{thm:MgTmu delta} if $\zeta<\infty$, i.e.
\begin{equation}
\label{eq:schatten_section_sharp_condition}
\frac{s}{d} + \frac{1}{q}  = \frac{1}{\eta} + \frac{1}{2} - \frac{1}{\zeta}.
\end{equation}
In particular, the corresponding result does not capture the scaling of the stochastic heat equation, see Subsection \ref{ss:scaling_intro}.
Here, similar to the weighted case described in Subsection \ref{ss:proof_strategy}, the main step is a bilinear interpolation argument (see also Subsection \ref{ss:Bird} for the abstract viewpoint).

\smallskip

We begin by describing the setting. 
Let $(e_n)_{n\geq 1}$ and $(f_n)_{n\geq 1}$ be orthonormal systems in $L^2(\cO)$. As in the previous section, we will prove estimates for
$\|M_g R_{\mu}\|_{\g(L^2(\Domm),L^q(\Domm))}$ where $M_g$ is the multiplication operator by $g\in L^\eta(\Domm)$, and $R_{\mu}$ admits the decomposition 
\begin{equation}
\label{eq:T_decomposition_schatten_section}
R_{\mu}=\sum_{n\geq 1} \mu_n\, e_n\otimes f_n.
\end{equation}
We consider the standard $\ell^\zeta$ norm on the sequence $\mu$. 

\smallskip

The following result provides an estimate for non-trace class noise in the case when the information $\mu\in \ell^\zeta(\Sf)$ for an orthonormal system $\Sf  = (f_n)_{n\geq 1}$ is \textit{not} available, but instead only $\mu\in\ell^\zeta$ is known. It is clear that, due to the missing factor $-\tfrac{1}{\zeta}$ in the first inequality in \eqref{eq:parametersnew} below (compare to \eqref{eq:conditionsmain delta}), this loss of information also  results in a loss of regularity.

\begin{theorem}[Embeddings for Gaussian series with unweighted sequences]
\label{thm:generalONB}
Let $\Domm$ be an open set in $\R^d$.
Let $s>0$, $q\in (1, \infty)$, $\eta\in [2, \infty)$, $\zeta\in [2, \infty)$ be such that
\begin{align}
\label{eq:parametersnew}
\frac{s}{d} + \frac{1}{q} \geq \frac{1}{\eta} +\frac12\qquad \text{and}  \qquad  \frac{1}{\eta} +\frac{1}{\zeta}>\frac{1}{q}.
\end{align}
Then for each $g\in L^{\eta}(\mathcal{O})$ and $\mu \in \ell^\zeta$, the operator $M_g R_{\mu}:L^2(\mathcal{O})\to \wt{H}^{-s,q}(\mathcal{O})$ is $\gamma$-radonifying and 
\[\|M_g R_{\mu}\|_{\gamma(L^2(\mathcal{O}),\wt{H}^{-s,q}(\mathcal{O}))}
\lesssim_{d,s,q,\eta,\zeta}
         \|g\|_{L^\eta(\cO)} \|\mu\|_{\ell^\zeta}  .
\]
\end{theorem}

Note that in case of equality in  the first part of \eqref{eq:parametersnew}, the second part of \eqref{eq:parametersnew} can be restated as $\frac{s}{d}>\frac{1}{2} - \frac{1}{\zeta}$. This shows that $\mu \in \ell^\zeta$ for smaller $\zeta$ (in particular compared to $\zeta = \infty$) allows considering a smaller $s$.

\smallskip

As in Sections \ref{s:gamma_consequences} and \ref{s:main_result}, the above result and Proposition \ref{p:necessary_schatten} admit a variant where $\wt{H}^{-s,q}(\Domm)$ is replaced by the homogeneous spaces $\dot{H}^{-s,q}(\R^d)$. We leave the details to the interested reader.

\smallskip

In the following, we state a converse result, which shows again that the information $\mu\in\ell^\zeta$ alone is not so useful. Indeed,
if $\mu\in\ell^\zeta$ and $g\in L^\eta(\Domm)$, then the conditions ensuring $M_g R_{\mu}\in \g(L^2(\Domm),L^q(\Domm))$ are \emph{independent} of $\zeta$. In particular, choosing the corresponding sequence $\mu$ from the unweighted space $\ell^\zeta$ does \emph{not} allow one to recover a condition like \eqref{eq:schatten_section_sharp_condition}, and therefore to capture scaling properties of SPDEs as analyzed in Subsection \ref{ss:scaling_intro}.

\begin{proposition}[Necessary conditions in the unweighted setting]
\label{p:necessary_schatten}
Let $\Domm$ be an open set in $\R^d$.
    Let $s>0$, $q\in(1,\infty)$, $\eta\in [2,\infty)$ and $\zeta\in [2,\infty)$. Suppose that, for all $g\in L^\eta(\cO)$ and $\mu\in \ell^\zeta$,
    \begin{align}
    \label{eq: necessity 2 parameters general ONB}
        \|M_gR_{\mu}\|_{\gamma(L^2(\cO),H^{-s,q}(\cO))}
        \lesssim \|g\|_{L^\eta(\mathcal{O})} \|\mu\|_{\ell^\zeta},
	 \end{align}
	 where $\mu$ is the sequence corresponding to $R_{\mu}$ as in \eqref{eq:T_decomposition_schatten_section}.
    Then the first part of \eqref{eq:parametersnew} holds.  
    Moreover, if $\Domm=\R^d$, then the second part of \eqref{eq:parametersnew} holds as well where we allow equality as well. 
\end{proposition}

Note that the condition \eqref{eq:parametersnew} excludes the sharp case \eqref{eq:schatten_section_sharp_condition}, which comes out from the scaling analysis on the stochastic heat equations with non-trace class noise, see Subsection \ref{ss:scaling_intro}. In particular, the Schatten class is not the appropriate class for obtaining sharp bounds on $M_g T $.

The independence of $\zeta$ in the first part of \eqref{eq:parametersnew} shows that the estimate \eqref{eq: necessity 2 parameters general ONB} becomes better as $\zeta\to \infty$. However, in the case $\zeta=\infty$, Proposition \ref{prop:M_g_Sobolev} and \ref{lem:zetainfty} already give the estimate \eqref{eq: necessity 2 parameters general ONB} for $\mu\in \ell^\infty$. Therefore, from this perspective, knowing that $\mu\in \ell^\zeta$ does not allow for improving the bounds on $M_g R_{\mu}$ given in Theorem \ref{thm:MgTmu delta}. The only exception, which we do not expect to be important in applications to SPDEs, is that the second condition in \eqref{eq:parametersnew} becomes less restrictive as $\zeta$ gets smaller, and allows for $\eta\geq q$.

\smallskip

The proofs of Theorem \ref{thm:generalONB} and Proposition \ref{p:necessary_schatten} are given in Subsections \ref{ss:generalONB_proof} and \ref{ss:necessary_schatten_proof}  below, respectively.

\subsection{Proof of Theorem \ref{thm:generalONB}}
\label{ss:generalONB_proof}
Here, we argue as in the proof of Theorem \ref{thm:MgTmu delta} given in Subsection \ref{ss:proof_generalized_growth}. As in the latter, we use bilinear interpolation. For the endpoint case $\zeta=\infty$, we will again use Lemma \ref{lem:zetainfty} (that is also contained in Theorem \ref{thm:MgTmu delta}), where one needs the following as a replacement of Lemma \ref{lem:l2 2 delta} concerning the endpoint case $\zeta=2$.  
\begin{lemma}
\label{lem:parameternew}
Let $\Domm$ be an open set in $\R^d$. 
Suppose that
\begin{align}
\label{eq:parametersnew lemma}
s\geq 0, \quad q\in (1, \infty), \quad   \frac{s}{d} + \frac{1}{q} = \frac{1}{\eta}+\frac12, \quad  \eta\in [2, \infty), 
\end{align}
Then for each $g\in L^{\eta}(\mathcal{O})$ and $\mu\in \ell^2$, the operator $M_g R_{\mu}:L^2(\mathcal{O})\to H^{-s,q}(\mathcal{O})$ is $\gamma$-radonifying and 
\[\|M_g R_{\mu}\|_{\gamma(L^2(\mathcal{O}),H^{-s,q}(\mathcal{O}))}\lesssim \|g\|_{L^\eta(\mathcal{O})} \|\mu\|_{\ell^2},\]
where the constant depends only on the parameters $(d,s,q,\eta)$. 
\end{lemma}
\begin{proof}
Let $\frac{1}{p} = \frac{s}{d} + \frac{1}{q} =\frac{1}{\eta}+\frac{1}{2}$. Note that by Sobolev embedding and H\"older's inequality 
\begin{align*}
  \|M_g R_{\mu}\|_{\gamma(L^2(\mathcal{O}),H^{-s,q}(\mathcal{O}))} & \lesssim \|M_g R_{\mu}\|_{\gamma(L^2(\mathcal{O}),L^p(\mathcal{O}))} \\ 
  & =  \Big\|g \sum_{n\geq 1} \gamma_n \mu_n f_n \Big\|_{L^2(\Omega;L^p(\mathcal{O}))}\\ 
  & \leq \|g\|_{L^\eta(\mathcal{O})} \Big\|\sum_{n\geq 1} \gamma_n \mu_n f_n \Big\|_{L^2(\Omega;L^2(\mathcal{O}))} 
 =\|g\|_{L^\eta(\mathcal{O})} \|\mu\|_{\ell^2},
\end{align*}
where we used orthogonality in the last line.  
\end{proof}

\begin{proof}[Proof of Theorem \ref{thm:generalONB}]
Note that the case $\zeta = 2$ is covered by Lemma \ref{lem:parameternew}.
To prove the result, we use multilinear complex interpolation in the parameters $(s, q, \eta, \zeta)$ in a similar way as in Theorem \ref{thm:MgTmu delta}, i.e.\ by applying Proposition \ref{prop: interpolation}.
Note that Lemma \ref{lem:parameternew} gives the result for the parameter $(s_0, q_0, \eta_0, \zeta_0)$ with $\zeta_0 = 2$. Moreover, Lemma \ref{lem:zetainfty} gives the result for $(s_1, q_1, \eta_1, \zeta_1)$ with $\zeta_1=\infty$, where the second condition in \eqref{eq:parametersnew} is automatically satisfied due to
\begin{align}
\label{eq:sharp equality s_1}
    \frac{1}{q_1}<\frac{s_1}{d}+\frac{1}{q_1} = \frac{1}{\eta_1} + \frac{1}{2}.
\end{align}
Let $(s, q, \eta, \zeta)$ be as in the theorem statement. As before \eqref{eq:cond1}, \eqref{eq:cond2}, \eqref{eq:cond3}, \eqref{eq:cond4} need to hold, but here, \eqref{eq:cond5} needs to be replaced by 
\begin{align}
\label{eq:cond5new}
\frac{s_0}{d} + \frac{1}{q_0} = \frac{1}{\eta_0} +\frac12.
\end{align}
This again gives a way to write the unknowns $(s_0, q_0, \eta_0)$ and $(s_1, q_1, \eta_1)$ in terms of $(q_1, \eta_1)$.
Since $\zeta_0 = 2$ and $\zeta_1=\infty$, we again have $\theta = 1-\frac{2}{\zeta}$ as in \eqref{eq:condtheta}, where $\zeta\in (2, \infty)$. 
First, we note that due to \eqref{eq:sharp equality s_1} and \eqref{eq:cond5new}, the first condition in \eqref{eq:parametersnew} is satisfied.
It remains to check that under the second condition in \eqref{eq:parametersnew}, there exist $q_0\in [1, \infty)$, $\eta_0\in [2, \infty)$, $s_0\geq 0$, $s_1>0$, for  $(\eta_1, q_1)$ satisfying $2<\eta_1<q_1$, where $(s_0,q_0,\eta_0,\zeta_0)$ satisfy \eqref{eq:parametersnew lemma} and $(s_1,q_1,\eta_1,\zeta_1)$ satisfy \eqref{eq:boundHsqsharp}.

Due to \eqref{eq:cond2}, the conditions $0<\frac{1}{\eta_0}\leq \frac12$ are equivalent to 
\begin{align}
\label{eq:restrqeta2new}
\frac{1}{\theta \eta} - \frac{1}{2\theta}+\frac12\leq \frac{1}{\eta_1}<\frac{1}{\theta \eta}.
\end{align} 
Due to \eqref{eq:cond1}, the conditions $0<\frac{1}{q_0}\leq 1$ are equivalent to 
\begin{align}
\label{eq:restrqeta3new}
\frac{1}{\theta q} - \frac{1}{\theta}+1\leq\frac{1}{q_1}<\frac{1}{\theta q}.
\end{align} 
Due to \eqref{eq:cond5new}, $s_0\geq 0$ is equivalent to $\frac{1}{2} + \frac{1}{\eta_0}- \frac{1}{q_0}\geq 0$. Due to \eqref{eq:cond1} and \eqref{eq:cond2} this can be rewritten as
\begin{align}\label{eq:restrqeta4new}
\frac{1}{\eta_1}\leq \frac{1}{2\theta} - \frac{1}{2} + \frac{1}{\theta \eta} - \frac{1}{\theta q} + \frac{1}{q_1}. 
\end{align}
Finally, from \eqref{eq:cond3}, we see that $s_1>0$ follows from $q_1>\eta_1$. 
The existence of an admissible $\eta_1$ satisfying  $2<\eta_1<q_1$ together with \eqref{eq:restrqeta2new} and \eqref{eq:restrqeta4new} is then equivalent to 
\begin{align}
\label{eq:restreta_1 combined}
\max\Big\{\frac{1}{q_1}, \frac{1}{\theta \eta} - \frac{1}{2\theta}+\frac12\Big\}<\min\Big\{\frac{1}{\theta \eta}, \frac{1}{2\theta} - \frac{1}{2} + \frac{1}{\theta \eta} - \frac{1}{\theta q} + \frac{1}{q_1}, \frac12\Big\}.
\end{align}
For the right-hand side above to be positive for any $q_1>1$, we require that 
\begin{align}
\label{eq:restrqnew}
\frac{1}{2\theta} - \frac{1}{2} + \frac{1}{\theta \eta} - \frac{1}{\theta q}>0,
\end{align}
Hence, if \eqref{eq:restrqnew} holds,
we can find an admissible $q_1$ such that \eqref{eq:restrqeta3new} holds 
if
\begin{align}\label{eq:restrqnew2}
1 - \frac{1}{\theta} +\frac{1}{\theta q} <\min\Big\{\frac12,\frac{1}{\theta \eta}, \frac{1}{\theta q}\Big\}.
\end{align}
One can check that \eqref{eq:restrqnew} is equivalent to the second condition in \eqref{eq:parametersnew}, due to $\theta = 1-\tfrac{2}{\zeta}$. Moreover, 
the first part of the minimum in \eqref{eq:restrqnew2} leads to $\frac{1}{q} - \frac{1}{\zeta} <\frac12$ which is also true due to the second condition in \eqref{eq:parametersnew} and $\eta\geq 2$. The second part of the minimum in \eqref{eq:restrqnew2} leads to $\frac{1}{\eta}  +\frac{2}{\zeta} > \frac{1}{q}$ which also holds due to the second condition in \eqref{eq:parametersnew}. The third part of the minimum does not lead to any restrictions. 
\end{proof}

\subsection{Proof of Proposition \ref{p:necessary_schatten}}
\label{ss:necessary_schatten_proof}

\begin{proof}[Proof of Proposition \ref{p:necessary_schatten}]
We divide the proof into two steps.

\smallskip

\emph{Step 1: Necessity of $\frac{s}{d} + \frac{1}{q} \geq  \frac{1}{\eta} +\frac12$}.
In this step, we argue as in the proof of Proposition \ref{p:necessityLinfty_hom}. 
Indeed, arguing in the proof of the latter result, it is enough to consider functions supported in the unit ball $B_1$ with center at the origin. Let $T = e\otimes h = \mu e\otimes \frac{h}{\|h\|_{L^2}}$ for $\|e\|_{L^2(B)}=1$ and $e,h\in C^\infty_{{\rm c}}(B)$ and where $\mu_1 = \|h\|_{L^2}$ and $\mu_n=0$ for all $n\geq 2$. 
Since $\|\mu\|_{\ell^\zeta}=\|h\|_{L^2(\Domm)}$, similarly to the proof of Proposition \ref{p:necessityLinfty_hom}, we obtain the following analogue of \eqref{eq:estimate_full_space_one_side}:
\begin{equation}
\label{eq:estimate_full_space_one_side_unweighted}
\|g h\|_{H^{-s,q}(\R^d)}\leq C\|g\|_{L^\eta(\R^d)} \|h\|_{L^2(\R^d)}
\end{equation}
where $C>0$ is independent of $g,h\in C^\infty(B)$.

Similar to the proof of Proposition \ref{p:necessityLinfty_hom}, the estimate \eqref{eq:estimate_full_space_one_side_unweighted} and a scaling argument readily yield the desired inequality.

\smallskip

\emph{Step 2: Necessity of $\frac{1}{\eta} + \frac{1}{\zeta}\geq \frac1q$ in case $\cO=\R^d$.}
    Let $\phi\in C^\infty_{{\rm c}}((0,1)^d)$ with $\|\phi\|_{L^2(\R^d)}=1$ and define an orthonormal system via $e_k = \phi(\cdot -k)$ for $k\in \bZ^d$. Let $\psi\in C^\infty_{{\rm c}}((0,1)^d)$ such that $\psi = 1$ on $\supp \phi$, let an integer $N\geq 1$ and define $g = \sum_{1\leq |k| \leq N} \psi(\cdot-k)$, so that clearly $ge_k = e_k$, for $1\leq |k|\leq N$. Let, moreover, $\mu_k = 1$ for all $1\leq |k|\leq N$ and zero otherwise.
    Then, 
    \begin{align*}
    \|M_g R_\mu \|_{\gamma(L^2,H^{-s,q})}  
    &\geq \Big\|\sum_{k\geq 1}\gamma_k g\sum_{1\leq |n|\leq N}(e_k,e_n)e_n \Big\|_{L^2(\Omega;H^{-s,q})} 
    = \Big\|\sum_{1\leq |k|\leq N}\gamma_k e_k \Big\|_{L^2(\Omega;H^{-s,q})} 
    \\& = \Big\|\sum_{1\leq |k|\leq N}\gamma_k (1-\Delta)^{-s/2} e_k \Big\|_{L^2(\Omega;L^q)} 
     \eqsim \Big\|\Big(\sum_{1\leq |k|\leq N} |(1-\Delta)^{-s/2} e_k|^2 \Big)^{1/2}\Big\|_{L^q} 
    \\& \geq \Big(\sum_{1\leq |k|\leq N} \int_{(0,1)^d+k} |(1-\Delta)^{-s/2} e_k|^q \,\di x\Big)^{1/q}
     =N^{d/q} \|(1-\Delta)^{-s/2}\phi\|_{L^{q}((0,1)^d)} ,
    \end{align*}
    where we used \eqref{eq: square fct char gamma} and that 
       $|(1-\Delta)^{-s/2}e_k|^2$ is dominated by its sum.
    On the other hand,
\begin{align*}
\|g\|_{L^{\eta}(\R^d)} \|\mu\|_{\ell^\zeta} =\|\psi\|_{L^{\eta}(\R^d)} N^{d/\eta} N^{d/\zeta}. 
\end{align*}
Therefore, if \eqref{eq: necessity 2 parameters general ONB} holds, then letting $N\to \infty$, we find that $\frac{1}{\eta} + \frac{1}{\zeta}\geq \frac1q$. 
\end{proof}

\section{Selected applications to random fields and SPDEs}
\label{sec: SPDE section}
This section aims to show how the main results of the paper can be used to study random fields and parabolic SPDEs. 
This section is organized as follows:
\begin{itemize}
    \item Subsection \ref{ss:random_field} -- Applications to random fields, including Mat\'ern-type fields.
    \item Subsection \ref{ss:sharp_bounds_SPDEs} -- Sharp bounds for solutions to parabolic SPDEs.
    \item Subsection \ref{subsection SPDE comparison} -- Comparison of our results on SPDEs with the existing literature. 
\end{itemize}
We point out that in Subsection \ref{ss:sharp_bounds_SPDEs}, we are not aiming to formulate the most general result on parabolic SPDEs, which we can obtain using the results in the previous sections, but rather to illustrate what type of estimates one obtains, and how all the parameters interact. In particular, our results can be applied to more general second-order operators, see Remark \ref{rem:generalization_SPDE}\eqref{it:generalization_SPDE2}.

We start with the application to the regularity of white noise and Mat\'ern noise, which appear as stochastic elliptic PDEs.

\subsection{Spatial random fields}
\label{ss:random_field}
Here, we discuss implications of our results for important choices of random fields. 
For detailed theory on spatial random fields, the reader is referred to 
\cite{AdlTay, AzFag20, cioica2010spatial, CDDKLRRS, FagFalUn, LangSch, LRL, RueHel} and references therein.  We start our discussion with spatial white noise. 

\smallskip

We begin with a relatively simple case concerning the regularity of spatial white noise on an open set $\cO\subseteq \R^d$. Recall that a spatial white noise is formally given by the Gaussian series $\xi=\sum_{n\geq 1} f_n \g_n$, where $(f_n)_{n\geq 1}$ is an orthonormal basis of $L^2(\cO)$ and $(\g_n)_{n\geq 1}$ a sequence of standard independent Gaussian random variables on a probability space $(\O,\mathcal{A},\P)$.

For $p=2$, the following is a consequence of Propositions \ref{prop:M_g_Sobolev} and \ref{prop:M_g_Sobolev_hom} (see also Lemma \ref{lem:zetainfty}) and \eqref{eq:homogeneous_spaces_Mg_R_bounds_step1} in the proof of Theorem \ref{thm:MgTmu delta_hom}. General $p\in [1, \infty)$ follows from the Kahane-Khintchine inequalities. 

\begin{proposition}[Regularity of white noise with a multiplicative coefficient]
Let $\cO$ be either a bounded domain with smooth boundary, or $\cO\in \{\T^d,\R^d\}$.
Let $\xi$ be a spatial white noise on $\cO$. 
Then, for all $p\in [1,\infty)$, $q\in (2, \infty)$,  $\eta\in (2, q)$, $s\in (\frac{d}{2},d)$ such that $\frac{s}{d}+\frac1q \geq \frac{1}{\eta} + \frac{1}{2}$, the Gaussian series $g\,\xi =g\,\sum_{n\geq 1} f_n \g_n$ converges in $L^p(\O;H^{-s,q}(\cO))$ and moreover
\begin{align*}
    \big(\E \|g \,\xi \|_{\wt{H}^{-s,q}(\mathcal{O})}^p\big)^{1/p} \lesssim_{s,\eta,p,q,d}  \|g\|_{L^\eta(\mathcal{O})}.
\end{align*}
Finally, if $\cO=\R^d$ and $\frac{s}{d}+\frac1q = \frac{1}{\eta} + \frac{1}{2}$, then the above holds with $H^{-s,q}(\R^d)$ replaced by $\dot{H}^{-s,q}(\R^d)$.
\end{proposition}

Note that in the case of bounded domains, one can take $g=1$. The previous yields $\xi\in H^{-s,q}(\mathcal{O})$ for any $q\in (2, \infty)$ and any $s>\frac{d}{2}$. This is well known and optimal in Sobolev spaces. The reader is referred to \cite{ArOh,V11_regularity_Besov} for improvements in terms of Besov spaces in the case $\cO=\T^d$. 

\smallskip

Next, we turn our attention to the more interesting case of Matérn fields. These are widely used in spatial data analysis \cite{SteinML}, and in machine learning \cite{RasWil}. We use the construction through a spatial SPDE as introduced in \cite{LRL}. 
Here, we obtain spatial regularity results for such random fields in the case of a multiplicative coefficient $g$. Below 
$(e_n)_{n\geq 1}$ and $(\lambda_n)_{n\geq 1}$ denote the eigenfunctions and corresponding normalized eigenvalues of the Laplacian on a bounded domain $\cO$, respectively. In particular, for the Dirichlet Laplacian $\Delta_{{\rm Dir}}$ (see Subsection \ref{sss:func_analytic_setting_heat} for details), one has $(1-\Delta_{{\rm Dir}})^{-\beta} f=\sum_{n\geq 1} (1+\lambda_n)^{-\beta}(f,e_n)_{L^2(\cO)}e_n$ for $\beta>0$ and $f\in L^2(\cO)$.

\begin{theorem}[Regularity of Matérn-type fields with a multiplicative coefficient]
\label{thm:regulartity_matern_with_multiplicative_g}
Let $\cO$ be either a bounded domain with smooth boundary, or $\cO\in \{\T^d,\R^d\}$. Let $\xi$ be a spatial white noise on $\cO$. Let $\alpha> 0$, $p\in [1, \infty)$, $q\in (1, \infty)$,  $\eta\in (1, q)$, $s\in (0,d)$ and $g\in L^\eta(\cO)$. The following assertions hold:
\begin{enumerate}[{\rm(1)}]
\item\label{it:regulartity_matern_with_multiplicative_g1} If $\cO$ is a bounded smooth domain, $\alpha\leq d-\frac12$, $\frac{1}{\eta} + \frac{1}{2} -\frac{s}{d} - \frac{1}{q}  <\frac{\alpha}{2d-1}$ and $\frac{1}{\eta} - \frac12<\frac{\alpha}{2d-1}$, then $g\,(1-\Delta_{{\rm Dir}})^{-\alpha/2} \xi$ is well-defined in $L^2(\O,H^{-s,q}(\cO))$ and 
$$
    \big(\E \|g \,(1-\Delta_{{\rm Dir}})^{-\alpha/2}\xi \|_{\wt{H}^{-s,q}(\mathcal{O})}^p\big)^{1/p} \lesssim_{s,\eta,p,q,d,\alpha}  \|g\|_{L^\eta(\mathcal{O})}.
$$
\item\label{it:regulartity_matern_with_multiplicative_g2} If $\cO\in \{\R^d,\T^d\}$, $\alpha\leq \frac{d}{2}$, $\frac{1}{\eta} + \frac{1}{2} -\frac{s}{d} - \frac{1}{q}  <\frac{\alpha}{d}$
and $\frac{1}{\eta} - \frac12<\frac{\alpha}{d}$, then 
$g\,(1-\Delta)^{-\alpha/2} \xi$ is well-defined in $L^2(\O,H^{-s,q}(\cO))$ and 
$$
    \big(\E \|g\, (1-\Delta)^{-\alpha/2}\xi \|_{H^{-s,q}(\mathcal{O})}^p\big)^{1/p} \lesssim_{s,\eta,p,q,d,\alpha}  \|g\|_{L^\eta(\mathcal{O})}.
$$
\end{enumerate}
\end{theorem}

As usual, in \eqref{it:regulartity_matern_with_multiplicative_g2}, $(1-\Delta)^{-\alpha/2}$ is defined by means of the Fourier multipliers with symbol $m_\alpha=(1+4\pi^2|\cdot|^2)^{-\alpha/2}$, see \eqref{eq:Fourier_multiplier}.
From Theorem \ref{thm:regulartity_matern_with_multiplicative_g}, one sees that on bounded domains the regularity of the (generalized) Mat\'ern field can be lower than in the cases $\T^d$ or $\R^d$.

\begin{proof}
\eqref{it:regulartity_matern_with_multiplicative_g1}: By expanding the noise with respect to the normalized eigenvalues $(e_n)_{n\geq 1}$ of the Dirichlet Laplacian on $\cO$, we obtain $\xi=\sum_{n\geq 1} (1+\lambda_n)^{-\alpha/2} e_n\g_n$. Hence, we can apply Theorem \ref{thm:MgTmu delta} to obtain the claim. Indeed, using the Weyl asymptotic $\lambda_n \sim n^{2/d}$ and the sup-norm bound $\|e_n\|_{L^\infty(\mathcal{O})} \lesssim n^{(d-1)/2d}$, the condition $\mu \in \ell^\zeta(\Sf)$ becomes (see \eqref{eq:ellzeta_F})
\begin{align}
    \sum_{n\geq 1} (1+\lambda_n)^{-\alpha \zeta/2} \|e_n\|_{L^\infty(\mathcal{O})}^2 
    \lesssim \sum_{n\geq 1} n^{-(\alpha\zeta)/d} n^{(d-1)/{d}} < \infty.
\end{align}
Convergence is ensured if the exponent is strictly less than $-1$, which leads to the condition  $\alpha > \frac{2d-1}{\zeta}$ on the noise smoothness parameter. Thus, the claim in \eqref{it:regulartity_matern_with_multiplicative_g1} follows from Theorem \ref{thm:MgTmu delta}.

\smallskip

\eqref{it:regulartity_matern_with_multiplicative_g2}: The case $\cO=\T^d$ follows as in \eqref{it:regulartity_matern_with_multiplicative_g1} by using that the normalized eigenvalues of the Laplacian satisfies $\|e_n\|_{L^\infty(\T^d)}\equiv 1$ (see also Subsection \ref{ss:periodic_setting}). In the case $\cO=\R^d$, it follows by Theorem \ref{thm:Randomfield} noticing that $m_\alpha = (1+4\pi^2|\cdot |^2)^{-\alpha/2}\in L^\zeta(\R^d)$ if and only if $\alpha>\frac{d}{\zeta}$. 
\end{proof}

\subsection{Sharp bounds for stochastic heat equations}
\label{ss:sharp_bounds_SPDEs}
In this subsection, we consider the following parabolic SPDEs: 
\begin{equation}
\label{eq:SPDE no drift}
\left\{
\begin{aligned}
\di u  & = \Delta u\, \di  t +  g R \,\di W &\ \  \text{ on }&\cO,\\ 
u(0)&=0 &\text{ on }&\cO.
\end{aligned}
\right.
\end{equation}
where $W$ denotes a cylindrical Brownian motion in $L^2(\cO)$ (see e.g.\ \cite[Definition 2.2]{AV25_survey}) on a filtered probability space $(\O,\mathcal{A},(\mathcal{F}_t)_{t\geq 0},\P)$, the process $
    g:(0,T)\times\Omega \to L^\eta(\cO)$ for an $\eta >1$ is progressively measurable, and in case $\partial\cO$ is non-empty, we supplement the above problem with Dirichlet boundary conditions, i.e. $u=0\ \  \text{ on }\partial\cO$.
Recall that $W$ can be identified with space-time white noise on $\cO$. Moreover, we interpret the function $g$ in the expression $g R \,\di W$  as a multiplication operator by $g$, i.e.\ $M_g R\, \di W$.

We will focus on the following two situations depending on the structure of the noise:
\begin{enumerate}[(a)]
    \item\label{eq:settnigRa} \emph{(Gaussian series -- Theorem \ref{thm: spde with delta no time}).} The noise is of the form
    \[ \sum_{n\geq 1} \mu_n f_n(x) w_n(t) = R_\mu  W(t),\]
where $(w_n)_{n\geq 1}$ are independent standard Brownian motions, $\mu$ a coloring sequence, $(e_n)_{n\geq 1}$ and $(f_n)_{n\geq 1}$ are an orthonormal basis of $L^2(\cO)$, and
\begin{equation}
\label{eq:Tmu_def2}
R_\mu = \sum_{n\geq 1} \mu_n e_n\otimes f_n \in \cL(L^2(\cO), L^\zeta(\cO)).
\end{equation}
\item\label{eq:settnigRb} \emph{(Fourier multipliers -- Theorem \ref{thm:heatMatern}).} The domain is the whole space $\cO=\R^d$, and the noise is of the form
$$
(\cF^{-1}m )*W= T_m W,
$$
where $\F^{-1}$ is the inverse Fourier transform, and $T_m$ the Fourier multiplier with symbol $m$, see \eqref{eq:Fourier_multiplier}. In particular, this covers Mat\'ern fields (see Subsection \ref{ss:random_field}), that is, for some $\alpha>0$,
$$
(1-\Delta)^{-\alpha/2} W = (\F^{-1}m_\alpha )*W \quad  
\text{ where }\quad m_\alpha(\cdot):=(1+4\pi^2|\cdot|^2)^{-\alpha/2}.
$$
\end{enumerate}

There is a very extensive literature on the stochastic heat equation \eqref{eq:SPDE no drift}. For details on the setting in which a non-trace class noise is constructed via the eigensystem of the leading operator, which in a certain sense is related to \eqref{eq:settnigRa} from the above, the reader is referred to \cite{C03, cerrai2009khasminskii, cerrai2011averaging,cerrai2025nonlinear,cerrai2003large, Salins22, Salinstrans}. A detailed comparison is included in Subsection \ref{sss:comparison_new_2}. The whole space case, in some sense, is contained in both \eqref{eq:settnigRa} and \eqref{eq:settnigRb}. For references on the whole space setting with non-trace class noise related to both settings, the reader is referred to  \cite{Kry} and \cite{Dal99,Khos,Salinsunbdd,sanz2002holder}, respectively.

\subsubsection{Functional analytic setting for stochastic heat equations}
\label{sss:func_analytic_setting_heat}

In this subsection, we illustrate the functional analytic setting for the Dirichlet Laplacian, which will be needed here. Here, we begin by illustrating the case of bounded domains, as the periodic and whole-space cases are easier.   
From  \cite[Example A.4]{AV25_survey}, one can view the Dirichlet  Laplacian as a mapping 
\begin{equation}
\label{eq:Delta_Dir_def}
\Delta_{\rm Dir}: H_{\rm Dir}^{1-s,q}(\cO)\subseteq H_{\rm Dir}^{-1-s,q}(\cO)\to H_{\rm Dir}^{-1-s,q}(\cO)
\end{equation}
where $(H_{\rm Dir}^{\sigma,q}(\cO))_{\sigma\in \R}$ is the so-called Sobolev tower associated to the Dirichlet Laplacian. Here, we do not review the full construction. The reader is referred to \cite[Appendix A]{AV25_survey}. For future convenience, let us recall that one has the following identification:
\begin{align}
\label{eq:basic_equality_Sob_tower_SPDEs}
H^{2,q}_{\rm Dir}(\cO) = 
H^{2,q}(\cO)
\cap 
H_0^{1,q}(\cO)
\quad \text{ and }\quad 
H^{\sigma,q}_{\rm Dir}(\cO) = 
H_0^{\sigma,q}(\cO) \ \text{ for }\sigma\in (0,1+\tfrac{1}{q})\setminus \{\tfrac{1}{q}\}, 
\end{align}
see \cite[eq.\ (A.3)]{AV25_survey}. Here, as usual, $H^{\sigma,q}_0(\cO)$ denotes the closure of the test functions $C^\infty_{{\rm c}}(\cO)$ in $H^{\sigma,q}(\cO)$.
Since $\Delta_{\rm Dir}$ generates an analytic $C_0$-semigroup $S$ on $H_{{\rm Dir}}^{-1-s,q}(\cO)$ for any $s\in\R$ and $q\in (1,\infty)$, we can consider a mild solution $u$ of \eqref{eq:SPDE no drift} and this has the form
\begin{align}
\label{eq: mild sol}
    u(t) = \int_0^t S(t-s) M_{g(s)}\, R \,\di W(s),\quad t\in [0,T].
\end{align}
Here, with a slight abuse of notation, we do not display the dependence of the semigroup $S$ on $(s,q)$ as for different values of $(s,q)$ the semigroups are compatible.

It is well-known that in the case either $\cO=\T^d$ or $\cO=\R^d$, the above construction can also be performed by ignoring the boundary conditions. In particular, the solution to \eqref{eq:SPDE no drift} is given by \eqref{eq: mild sol} where $(S(t))_{t\geq 0}$ is the usual heat semigroup on $\T^d$ and $\cO=\R^d$, respectively. 

\smallskip

In the following section, we derive sharp bounds for the mild solution $u$ given in \eqref{eq: mild sol} depending on the assumptions on $g$ and $R$. For simplicity, we only consider $R$ (thus $\mu$ and $m$ below) which is time-independent. Comments on the possible applications to nonlinear SPDEs, and to more general differential operators are given in Remark \ref{rem:generalization_SPDE} below.

\subsubsection{Sharp estimates for stochastic heat equations}
\label{ss:results_heat_eq_general}
Here, we provide estimates for the solution to \eqref{eq:SPDE no drift} in Bessel potential spaces, based on the estimates derived in the present paper. 
We present two results, following the cases \eqref{eq:settnigRa} and \eqref{eq:settnigRb} illustrated at the beginning of Subsection \ref{ss:sharp_bounds_SPDEs}. 
As for the first case, we mainly focus on the weighted sequence space, as the unweighted setting does not capture the scaling of the stochastic heat equation (see Subsection \ref{ss:scaling_intro} and Section \ref{s:weight_necessary}).

\begin{theorem}[Weighted coloring sequences -- $\Sf\subseteq  L^\infty(\cO)$]
\label{thm: spde with delta no time}
Let $\cO$ be either a bounded $C^2$-domain, or $\cO\in \{\R^d,\R^d_+,\T^d\}$.
Let $\Sf=(f_n)_{n\geq 1}$ be an orthonormal basis of $L^2(\cO)$ such that $f_n\in L^\infty(\cO)$ for all $n\geq 1$.
    Fix $T>0$, $q\in [2, \infty)$ $\eta\in (1, q)$, $\zeta\in (2, \infty)$,  $s\in(0,d)$, and assume that $s\in (0,2)$  if $\partial\cO$ is not empty.  Suppose that 
\begin{align}\label{eq:sharpsqetazeta}
\frac{s}{d} + \frac{1}{q}  = \frac{1}{\eta} + \frac{1}{2} - \frac{1}{\zeta} \qquad \text{ and } \qquad   \frac{1}{\eta} - \frac{1}{\zeta}<\frac12.
\end{align}
Let $p\in (2,\infty)$, or $p\in [2,\infty)$ in case $q=2$.
Then, for all $R=R_\mu$ as in \eqref{eq:Tmu_def2} with $\mu\in\ell^\zeta(\Sf)$, and progressively measurable $g\in L^p(\R_+\times \O;L^\eta(\cO))$, the mild solution $u$ to \eqref{eq:SPDE no drift} on $[0,T]$ satisfies
\begin{align}
  \E\|u\|_{L^p(0,T;H^{1-s,q}(\cO))}^p 
  + \E\|u\|_{C([0,T];B^{1-s-2/p}_{q,p}(\cO))}^p
    \lesssim_{T,s,d,q,\eta,\zeta}\|\mu\|_{\ell^\zeta(\Sf)}^p\,
   \E\|g\|_{L^{p}(0,T;L^\eta(\cO))}^p.
\end{align}
Moreover, if $\cO=\R^d$, then for $\mu$ and $g$ as before, it holds that  
\begin{align}
  \E\|u\|_{L^p(\R_+;\dot{H}^{1-s,q}(\R^d))}^p 
  + \E\|u\|_{C([0,\infty);\dot{B}^{1-s-2/p}_{q,p}(\R^d))}^p
    \lesssim_{s,d,q,\eta,\zeta}\|\mu\|_{\ell^\zeta(\Sf)}^p\,
   \E\|g\|_{L^{p}(\R_+;L^\eta(\R^d))}^p.
\end{align}
\end{theorem}

Clearly, if instead of the first in \eqref{eq:sharpsqetazeta}, one has $\frac{s}{d} + \frac{1}{q}  \geq \frac{1}{\eta} + \frac{1}{2} - \frac{1}{\zeta}$, then the first estimate in Theorem \ref{thm: spde with delta no time} still holds.

The space $B^{1-s-2/p}_{q,p}(\cO)$ appearing in the above result is a Besov space and could be defined as the real interpolation space $(H^{-1-s,q}(\cO),H^{1-s,q}(\cO))_{1-\frac1p,p}$, and similarly for $\dot{B}^{1-s-2/p}_{q,p}(\R^d)$. For details on real interpolation of such spaces, see e.g.\ \cite[Chapter 14]{HNVW3} and \cite[Chapter 6]{BeLo}. However, it is worth noting that, by Sobolev embeddings for Besov spaces, one has, for all $p\in (1,\infty]$,
\begin{equation}
\label{eq:Besov_embedding_C}
B^{\sigma}_{q,p}(\cO)\embed C(\overline{\cO})\  \text{ provided }\  \sigma>d/q.
\end{equation}
Hence, under appropriate assumptions on the parameters, the above result also allows us to obtain maximal estimates for the process in space and time. On the latter point, a comparison with recent results is given in Subsection \ref{sss:comparison_new_2} below.

\smallskip

As typical in the maximal regularity approach to SPDEs, the estimates in Theorem \ref{thm: spde with delta no time} can be strengthened to optimal space-time estimates involving fractional Sobolev spaces in time, see e.g.\ \cite[Theorem 1.2(1)]{MaximalLpregularity} and \cite[Definition 3.7]{AV25_survey}. In particular, these imply  H\"older regularity jointly in space and time up to an optimal borderline exponent. 
Power weights in time can be included as well, see \cite{AV19}. 
Finally, we also mention that the moment and integrability in time in the estimates of Theorem \ref{thm: spde with delta no time} can be chosen differently, see \cite[Corollary 7.4]{NVW13}.
We leave the details to the interested reader. 
The same results hold for the Neumann Laplacian, and even for more general boundary conditions.

\begin{proof}[Proof of Theorem \ref{thm: spde with delta no time}]
We begin by proving the first claimed estimate.
For brevity, we only discuss the case  where $\partial\cO$ is not empty, as the other case is simpler. 
In this case, note that by duality (see \eqref{eq:dualHtilde}) and restricting the functionals, every element of $\wt{H}^{-s,q}(\cO)= H^{s,q'}(\cO)^*$ restricts to an element of $H_{\rm Dir}^{-s,q}(\cO) = H_{\rm Dir}^{s,q'}(\cO)^*$ and 
\begin{equation}
\label{eq:identification_of_negative_sob_spaces}
\|f|_{H^{s,q'}_{{\rm Dir}}(\cO)}\|_{H_{\rm Dir}^{-s,q}(\cO)} \lesssim \|f\|_{\wt{H}^{-s,q}(\cO)}.
\end{equation}
However, note that, on the one hand, for some choices of the parameters it may happen that 
a nonzero $u\in \wt{H}^{-s,q}(\cO)$ gives a zero $u$ in $H_{\rm Dir}^{-s,q}(\cO)$. On the other hand, for all $g\in L^\eta(\cO)$ and $f\in L^2(\cO)$, the distribution $(M_{g} R_\mu) f\in \wt{H}^{-s,q}(\cO)$ is uniquely determine by its action on $H_{\rm Dir}^{s,q'}(\cO)$.

Now, by stochastic maximal regularity, see for instance \cite[Section 3]{AV25_survey}, it suffices to note that
    \begin{align*}
    \E\int_0^T \|M_{g(t)}R_\mu\|_{\gamma(L^2(\cO),H^{-s,q}_{\rm Dir}(\cO))}^p \,\di t
        &\stackrel{\eqref{eq:identification_of_negative_sob_spaces}}{\lesssim} \E\int_0^T \|M_{g(t)}R_\mu\|_{\gamma(L^2(\cO),\wt{H}^{-s,q}(\cO))}^p \,\di t\\
        & \ \lesssim_{T,s,d,q,\eta,\zeta} \|\mu\|_{\ell^\zeta(\Sf)}^p\,
        \E\int_0^T 
        \|g(t)\|_{L^\eta(\cO)}^p \,\di t,
    \end{align*}
    where we applied Theorem \ref{thm:MgTmu delta} in the second inequality.
    The first claim of Theorem \ref{thm: spde with delta no time} now follows by  the stochastic maximal $L^p$-regularity of the Dirichlet Laplacian, which is a consequence of \cite[Theorems 1.1 and 1.2]{MaximalLpregularity} and \cite[Example A.3 and Proposition A.1(3)]{AV25_survey}. 

The final assertion follows similarly by using \cite[Theorem 6.5]{ALV21} (see also Example 6.6 there).
\end{proof}

In case $\Sf\not\subseteq L^\infty(\cO)$, i.e.\ if there exists $n$ such that $f_n\notin L^\infty(\cO)$, then Theorem \ref{thm: spde with delta no time} is not applicable. However, with the same strategy, we can apply Theorem \ref{thm:generalONB} to obtain the same result for any orthonormal system $(f_n)_{n\geq 1}$, $\mu$ from the unweighted $\ell^{\zeta}$ under the more restrictive condition on the parameters $(s,q,\eta, \zeta)$ in \eqref{eq:parametersnew}.

\smallskip

Finally, we consider the case \eqref{eq:settnigRb} discussed at the beginning of this subsection. 

\begin{theorem}[Fourier multipliers and Mat\'ern-type noises]\label{thm:heatMatern}
Let $\cO = \R^d$.
Fix $T>0$, $q\in [2, \infty)$, $\eta\in (1, q)$, $\zeta\in [2, \infty]$ and $s\in(0,d)$.
Suppose that \eqref{eq:sharpsqetazeta} holds.
Let $p\in (2,\infty)$, or $p\in [2,\infty)$ in case $q=2$.
Then, for all Fourier multipliers $R=T_m$ with symbol $m\in L^\zeta(\R^d)$ (see \eqref{eq:Fourier_multiplier}), and progressively measurable $g\in L^p(\R_+\times \O;L^\eta(\R^d))$, the mild solution $u$ to \eqref{eq:SPDE no drift} on $[0,T]$ satisfies
\begin{align}
  \E\|u\|_{L^p(0,T;H^{1-s,q}(\R^d))}^p
  + \E\|u\|_{C([0,T];B^{1-s-2/p}_{q,p}(\R^d))}^p
&    \lesssim_{T,s,d,q,\eta,\zeta}\|m\|_{L^\zeta(\R^d)}^{p}
   \E\|g\|_{L^{p}(0,T;L^\eta(\R^d))}^p.
\\   \E\|u\|_{L^p(\R_+;\dot{H}^{1-s,q}(\R^d))}^p
  + \E\|u\|_{C(\R_+;\dot{B}^{1-s-2/p}_{q,p}(\R^d))}^p
    & \lesssim_{s,d,q,\eta,\zeta} \|m\|_{L^\zeta(\R^d)}^{p}
   \E\|g\|_{L^{p}(\R_+;L^\eta(\R^d))}^p.
\end{align}
\end{theorem}

Note that, if $R=(1-\Delta)^{-\alpha/2}=T_{m_\alpha}$ for some $\alpha>0$ where $m_\alpha=(1+4\pi^2|\cdot|^2)^{-\alpha/2}$ (Mat\'ern field), then the above estimates hold true for 
all $\zeta>\frac{d}{\alpha}$.
It is worth noticing that to obtain sharpness (or, in other words, scaling invariance), one needs to modify the symbol $m_\alpha$ at a logarithmic scale in the same spirit of Subsection \ref{ss:scaling_intro}.

For convenience, we collect further comments in the following remark.
\begin{remark}\
\label{rem:generalization_SPDE}
\begin{enumerate}[\rm (1)] 
 \item \emph{(A glimpse into applications to nonlinear SPDEs).}
 It is well-known that sharp estimates, such as in Theorems \ref{thm: spde with delta no time} and \ref{thm:heatMatern}, are central in the study of nonlinear SPDEs. Typically, the construction of solutions to nonlinear problems often involves linearization and fixed-point techniques. In the context of stochastic maximal $L^p$-regularity, such estimates have been proven to be a powerful tool to study SPDEs in \emph{critical spaces}, see \cite{AV25_survey}.
 As is evident in corresponding applications (see e.g.\ \cite{AVreaction-local,AV20_NS}), after the linear theory is established, the nonlinear estimates are handled with standard techniques such as H\"older inequality and Sobolev embeddings. 
 The $L^\eta(\cO)$-norm appearing in the previous results is very useful for this. Indeed, we apply our result to the nonlinearity of the form $g=G(u)=|u|^h$ where $h>1$. Then, for instance, from Theorems \ref{thm: spde with delta no time}, we have 
\begin{align}
    \E\|u\|_{L^p(0,T;H^{1-s,q}(\cO))}^p
    +\E\|u\|_{C([0,T];B^{1-s-2/p}_{q,p}(\cO))}^p
    &\lesssim
    \E\|G(u)\|_{L^p(0,T;L^\eta(\cO))}^p\lesssim \E\|u\|_{L^{ph}(0,T;L^{h\eta}(\cO))}^{ph}.
\end{align}
Hence, now one can apply a Sobolev embedding $H^{\sigma,q}(\cO)\embed L^{h \eta}(\cO)$ to bound the nonlinearity with norms appearing on the left-hand side of the same estimate.
Note that, as $h>1$, one can choose $\eta<q$ such that $h\eta>q$, and in this situation, the previous Sobolev embedding can be made \emph{sharp}. 
Of course, due to the results in \cite{AV19_QSEE_1,AV19_QSEE_2,AV25_survey}, it is not needed to perform the fixed point argument again. Far-reaching applications of our results to nonlinear SPDEs are given in our forthcoming work \cite{AGVlocal}.

       \item\label{it:generalization_SPDE2} \emph{(Second order operators in divergence form).}
       The Laplace operator in \eqref{eq:SPDE no drift} can also be replaced by a second-order operator in divergence form like $$u\mapsto \nabla \cdot (a\cdot \nabla u),$$ 
       for some uniformly elliptic matrix field $a\in C(\overline{\cO};\R^{d\times d})$. 
       Even in the case of a $C^1$-domain, using 
\cite[Theorem 11.5]{DivH}, one can check that the proofs of Theorems \ref{thm: spde with delta no time} and \ref{thm:heatMatern} also extend to this situation in case $s\leq 1$, see \cite[Lemma 5.4]{Hornung} for details on the applications of the results in \cite{DivH}. If $a\in C^{\alpha}(\cO;\R^{d\times d})$ and $\cO$ is a $C^{1,\alpha}$-domain for some $\alpha>0$, then one can also take $s<1+\min\{\frac{1}{q},\alpha\}$. 
In particular, the regularity of the domain $\cO$ in Theorem \ref{thm: spde with delta no time} can be weakened up to include a restriction on the smoothness parameter $s$.
\end{enumerate}
\end{remark}

\subsection{Comparison to related results for SPDEs}
\label{subsection SPDE comparison}
Our estimates from Theorem \ref{thm:MgTmu delta} potentially have consequences for several papers on nonlinear SPDEs with non-trace class noise (see e.g. \cite{C03,cerrai2003large, SchnVerdamp, SchnVerbdr} and subsequent papers). As mentioned above, applications to nonlinear SPDEs with non-trace class noise will be presented in \cite{AGVlocal}. 

\smallskip

Before we compare our results any further to results in the literature, we would like to point out a general principle which can be used to compare our results to semigroup formulations of estimates such as the one in Theorem \ref{thm:MgTmu delta}. 

Let $A = \Delta_{\rm Dir}$ be as explained in Subsection \ref{sss:func_analytic_setting_heat}. Thus, the operator acts on the space $H^{-1-s,q}_{\rm Dir}(\cO)$, with domain $ H^{1-s,q}_{\rm Dir}(\cO)$, see \eqref{eq:Delta_Dir_def}. Let $S$ denote the analytic semigroup generated by $A$.  For simplicity, we assume that $A$ is invertible, and thus $S$ is exponentially stable with $\|S(t)\|_{\cL(L^q(\cO))}\lesssim e^{-\delta t}$ for a $\delta >0$ (this is true if $\cO$ is bounded, and otherwise it suffices to shift the operator). Under the assumptions of Theorem \ref{thm:MgTmu delta}, from the latter, one gets, for all $t>0$,
\begin{equation}\label{eq:SgT}
\begin{aligned}
\|S(t) g  R_{\mu}\|_{\gamma(L^2(\cO),L^q(\cO))}
& \lesssim \|(-A)^{s/2}S(t)\|_{\calL(L^q(\cO))} \|(-A)^{-s/2} g  R_{\mu}\|_{\gamma(L^2(\cO),L^q(\cO))}
\\ & \lesssim t^{-s/2}\|\mu\|_{\ell^\zeta(\Sf)} \|g\|_{L^{\eta}(\cO)} .
\end{aligned}
\end{equation}
In part of the literature, the above type of estimate is the way to reformulate results like the one in Theorem \ref{thm:MgTmu delta}. 

Conversely, note that if $\|S(t) g R_{\mu}\|_{\gamma(L^2(\cO),L^q(\cO))}\lesssim t^{-s/2}$ for all $t>0$, then 
using a standard formula $(-A)^{-\theta} = C_{\theta} \int_0^\infty t^{\theta-1} S(t) \,\di t$ for $\theta\in (0,1)$ (see \cite[Lemma 6.1.5]{MarSanz}), and that $(-A)^{\sigma/2}:  L^q(\cO)\to H^{-\sigma,q}_{\rm Dir}(\cO)$ is an  isomorphism with inverse $(-A)^{-\sigma/2}$ (see \cite[Example A.4]{AV25_survey}), we find that
\begin{align*}
\|g R_{\mu}\|_{\gamma(L^2(\cO),H^{-\sigma,q}_{\rm Dir}(\cO))} &\eqsim \|(-A)^{-\sigma/2} g R_{\mu}\|_{\gamma(L^2(\cO),L^q(\cO))}
\\ & \lesssim \int_0^\infty t^{\sigma/2-1} \|S(t) g R_{\mu}\|_{\gamma(L^2(\cO),L^q(\cO))} \,\di t
\\ & \lesssim \int_0^1 t^{\sigma/2-1 -s/2}\,  \di t +\int_1^\infty e^{- \delta  t} \,\di t,
\end{align*}
where the latter is finite if $\sigma\in (s,2)$. 

Therefore, both our formulation and the one via the semigroup decay are equivalent up to a loss of some $\varepsilon>0$ in \emph{one direction}. In other words: our formulation in terms of negative Sobolev spaces (and hence, extrapolation spaces, see \cite[Appendix A]{AV25_survey} for the terminology) has the advantage that one can directly apply stochastic maximal $L^p$-regularity to it without losing any $\varepsilon>0$ of smoothness.

\subsubsection{Related $L^2$-estimates} 
Below, we connect our conditions to the ones in \cite{cerrai2009khasminskii,cerrai2011averaging,cerrai2017averaging}. For a discussion on  \cite[Lemma 3.3]{C03}, the reader is also referred to Remark \ref{r:schatter_and_heat_eq}.  The condition $\mu\in \ell^\zeta(\Sf)$ was introduced in \cite{cerrai2009khasminskii} 
and an estimate of the form \eqref{eq:SgT} was derived in the special case $\eta = q=2$:
\begin{align}
\label{eq: estimate Cerrai}
\|S(\cdot)M_{g(\cdot)}R_\mu\|_{L^2(0,T;\cL_2(L^2(\cO)))}\lesssim_T   \|\mu\|_{\ell^\zeta(\Sf)}\|g\|_{L^2((0,T)\times \cO)}.
\end{align}
under  the additional condition that for $\beta\in (0,\infty)$,
\begin{align}
\hspace{-1cm}
\label{eq: condition Cerrai}
\sum_{n\geq 1} \lambda_n^{-\beta}\|f_n\|_{L^\infty(\cO)}^2<\infty\quad
\text{and $\zeta=\infty$ if $d=1$,}\quad\text{or}\quad
\frac{\beta(\zeta - 2)}{\zeta}<1
\quad\text{for $\zeta>2$, if $d\geq 2$.}
\end{align}
From \cite{grieser2002uniform}, on a bounded smooth domain $\cO\subseteq \R^d$, we have
$\|f_n\|_{L^\infty(\cO)}\lesssim \lambda_n^{\frac{d-1}{4}}$ uniformly in $n\geq 1$. Moreover, by the Weyl law, one has $\lambda_n\eqsim n^{\frac{2}{d}}$, so that all together we may assume 
\begin{align}
    \|f_n\|_{L^\infty(\cO)}\lesssim n^{\frac{1}{2}-\frac{1}{2d}}\quad\text{uniformly in  \ $n\geq 1$.}
\end{align}
In order to ensure that \eqref{eq: condition Cerrai} holds, i.e. $\sum_{n\geq 1} n^{-\frac{2}{d}\beta}n^{1-\frac{1}{d}}<\infty$ together with $\tfrac{\beta(\zeta -2)}{\zeta}<1$, one then effectively requires that in $d\geq 2$,
\begin{align}
\label{eq: condition zeta cerrai}
    \zeta<2+\frac{4}{2d-3}.
\end{align}
From Theorem \ref{thm:MgTmu delta} (see \eqref{eq:SgT}) we  see
\begin{align}
\|S(t)M_gR_\mu\|_{\cL_2(L^2(\cO))} &\eqsim 
\|S(t)M_gR_\mu\|_{\gamma(L^2(\cO),L^2(\cO))}
\lesssim t^{-s/2}\|\mu\|_{\ell^\zeta(\Sf)}\|g\|_{L^2(\cO)},
\end{align}
where, to ensure square integrability in time and the conditions in Theorem \ref{thm:MgTmu delta}, we require that 
\begin{align}
\label{eq: comparison cerrai parameters}
1>s\geq \frac{d}{\eta}-\frac{d}{\zeta},\qquad \eta<2.
\end{align}
We observe that for $d\geq 2$ this holds for any $\zeta\in (2,\infty)$, so long as $\eta$ is sufficiently close to $2$.
Moreover, for $d=1$ and $\zeta=\infty$ we may apply Proposition \ref{prop:M_g_Sobolev}, which only requires that $s>\tfrac{1}{2}$. Together with the restriction from \eqref{eq: comparison cerrai parameters} for $d\geq 2$, we get the condition
\begin{align}
\label{eq: final condition comparison cerrai}
\zeta\in (2,\infty],\ \quad d=1,
\quad\text{or}\qquad
\zeta<\frac{2d}{d-2},\ \quad d\geq 2,
\end{align}
which gives a major improvement compared to \eqref{eq: condition zeta cerrai}. 
 Moreover, under the condition \eqref{eq: comparison cerrai parameters}, we obtain \eqref{eq: estimate Cerrai} without the additional condition \eqref{eq: condition Cerrai} and without assuming that the $f_k$ are the eigenfunctions of the operator $S$ or its generator. We note further that since the embedding $L^2(\cO)\hookrightarrow L^\eta(\cO)$, for $\eta<2$, fails in general unbounded domains, we also require the boundedness of $\cO$ to obtain \eqref{eq: estimate Cerrai}.

\subsubsection{Related estimates only under the condition $\mu\in\ell^\zeta(\Sf)$.}
\label{sss:comparison_new_2}
Estimates for the stochastic convolution \eqref{eq: mild sol} were also obtained in the very recent article \cite{cerrai2025nonlinear} under the condition that 
\begin{align}
\label{eq: zeta condition comparison 2025}
\mu\in\ell^\infty
\quad\text{if $d=1$ and}\quad
\mu\in\ell^\zeta(\Sf)
\quad\text{where}\quad \zeta<\frac{2d}{d-2}\quad\text{if $d\geq 2$,}
\end{align}
and the condition that the $(f_k)_{k\geq 1}$ are the eigenfunctions of the differential operator involved. There, an estimate of the form 
\begin{align}
    \label{eq: cerrai 2025 estimate}
    \E\sup_{t\in [0,T]}\|u\|^p_{C(\overline{\cO})}\lesssim_{p,T} \E\sup_{t\in [0,T]}\|g(t)\|_{C(\overline{\cO})}^p
\end{align}
for a $p>1$ sufficiently large, 
is required for the well-posedness of a stochastic convolution treated in that paper. 
Let us point out that \eqref{eq: cerrai 2025 estimate} is a special case of our results, which we can even improve due to the optimality of our setting. More precisely, the condition $\zeta<\frac{2d}{d-2}$ and Theorem \ref{thm: spde with delta no time} ensure the existence of $s<1$ such that, for all $q,p\in (2,\infty)$,
\begin{equation}
\label{eq:Lq_estimates_comparison}
\E\sup_{t\in [0,T]}\|u(t)\|_{B^{1-s}_{q,p}(\cO)}^p \lesssim_{q,p,s,T} \E \|g\|_{L^p(0,T;L^q(\cO))}^p.
\end{equation} 
Clearly, by taking $q$ so large that $B^{1-s}_{q,p}(\cO)\embed C(\overline{\cO})$ by \eqref{eq:Besov_embedding_C}, the above yields \eqref{eq: cerrai 2025 estimate}. Let us point out that the improvement is on both sides of the estimates, and in many applications to nonlinear reaction-diffusion with non-trace class noise, the estimate \eqref{eq:Lq_estimates_comparison} is more flexible. For instance, $q$ can be taken finite and critical spaces for reaction-diffusion are typically of this form, see e.g.\ \cite[Subsection 1.4]{AVreaction-local}.

Finally, we remark that the limiting case $\zeta=2d/(d-2)$ is also very relevant for our purposes, and it is the one that leads to optimal estimates in terms of scaling, cf.\ Subsection \ref{ss:scaling_intro} and hence critical situations in nonlinear SPDEs. In this respect, the estimate \eqref{eq:Lq_estimates_comparison} can be thought of as a \emph{sub-critical} one, while the one in Theorem \ref{thm: spde with delta no time} is \emph{critical}. 

\def\polhk#1{\setbox0=\hbox{#1}{\ooalign{\hidewidth
  \lower1.5ex\hbox{`}\hidewidth\crcr\unhbox0}}} \def\cprime{$'$}

\end{document}